\documentclass[a4paper,oneside,english,reqno,12pt]{amsart}
\usepackage[utf8]{inputenc}
\usepackage[T1]{fontenc}
\usepackage{graphicx}
\DeclareFontFamily{OMX}{mlmex}{}
\DeclareFontShape{OMX}{mlmex}{m}{n}{%
   <->mlmex10%
   }{}%
\usepackage{mlmodern}

\usepackage[hscale=0.666, vscale=0.72]{geometry}
\usepackage{babel}
\usepackage[numbers,sort&compress]{natbib}

\usepackage[np,autolanguage]{numprint}
\usepackage{hyperref}
\hypersetup{%
  colorlinks=true,%
  citecolor=[RGB]{120,29,126},%
  pdfauthor={Jean-François Burnol},%
  pdfsubject={Block-count constrained harmonic sums},%
  pdfstartview=FitH,%
  pdfpagemode=UseNone,%
}

\usepackage{amsmath,amsthm,amssymb}
\usepackage{mathtools}

\DeclarePairedDelimiterX\Iffint[2]{\lbrack\!\lbrack}{\rbrack\!\rbrack}{#1\dots#2}
\DeclarePairedDelimiterX\Ioo[2]{\lparen}{\rparen}{#1,#2}
\DeclarePairedDelimiterX\Iof[2]{\lparen}{\rbrack}{#1,#2}
\DeclarePairedDelimiterX\Ifo[2]{\lbrack}{\rparen}{#1,#2}
\DeclarePairedDelimiterX\Iff[2]{\lbrack}{\rbrack}{#1,#2}

\DeclarePairedDelimiterX\scalp[2]{\langle}{\rangle}{#1\mid#2}

\newcommand\emptyword{\epsilon}
\newcommand\ve{\varepsilon}

\newcommand\NN{\mathbb{N}}
\newcommand\ZZ{\mathbb{Z}}
\newcommand\QQ{\mathbb{Q}}

\newcommand\PP{\mathbb{P}}
\newcommand\Un{\mathbf{1}}
\newcommand\CC{\mathbb{C}}
\newcommand\EE{\mathbb{E}}

\newcommand\e{\mathsf{e}}
\newcommand\dt{\mathrm{d}t}
\newcommand\du{\mathrm{d}u}
\newcommand\dx{\mathrm{d}x}

\newcommand\drM{\mathrm{dM}}
\newcommand\drH{\mathrm{dH}}

\DeclareMathOperator\Span{span}

\DeclareMathOperator\Leb{Leb}
\DeclareMathOperator\Law{Law}
\DeclareMathOperator\Cyl{Cyl}
\DeclareMathOperator\Unif{Unif}

\DeclareMathOperator\bu{\mathbf{u}}
\DeclareMathOperator\bxi{\boldsymbol{\xi}}

\DeclareMathOperator\cA{\mathcal{A}}
\DeclareMathOperator\cB{\mathcal{B}}
\DeclareMathOperator\cC{\mathcal{C}}

\DeclareMathOperator\cF{\mathcal{F}}
\DeclareMathOperator\cG{\mathcal{G}}
\DeclareMathOperator\cI{\mathcal{I}}
\DeclareMathOperator\cH{\mathcal{H}}
\DeclareMathOperator\cL{\mathcal{L}}
\DeclareMathOperator\cR{\mathcal{R}}
\DeclareMathOperator\cK{\mathcal{K}}
\DeclareMathOperator\cW{\mathcal{W}}

\DeclareMathOperator\wD{\widehat{D}}
\DeclareMathOperator\wE{\widehat{E}}
\DeclareMathOperator\wH{\widehat{H}}
\DeclareMathOperator\wM{\widehat{M}}
\DeclareMathOperator\wK{\widehat{K}}
\DeclareMathOperator\wX{\widehat{X}}

\newcommand\brD{\mathbf{D}}
\newcommand\brE{\mathbf{E}}
\newcommand\brK{\mathbf{K}}
\newcommand\brV{\mathbf{V}}
\newcommand\brY{\mathbf{Y}}

\newcommand\wmu{\widehat{\mu}}
\newcommand\weta{\widehat{\eta}}

\DeclareMathOperator\sv{\mathsf{v}}

\DeclareMathOperator\sA{\mathsf{A}}
\DeclareMathOperator\sC{\mathsf{C}}
\DeclareMathOperator\sD{\mathsf{D}}
\DeclareMathOperator\sE{\mathsf{E}}
\DeclareMathOperator\sF{\mathsf{F}}
\DeclareMathOperator\sH{\mathsf{H}}
\DeclareMathOperator\sI{\mathsf{I}}
\DeclareMathOperator\sK{\mathsf{K}}
\DeclareMathOperator\sP{\mathsf{P}}
\DeclareMathOperator\sQ{\mathsf{Q}}
\DeclareMathOperator\sR{\mathsf{R}}
\DeclareMathOperator\sT{\mathsf{T}}

\newcommand\dmu{\mathrm{d}\mu}
\newcommand\deta{\mathrm{d}\eta}
\newcommand\dnu{\mathrm{d}\nu}

\newcommand\rd{\mathrm{d}}

\newcommand\rM{\mathrm{M}}
\newcommand\rH{\mathrm{H}}

\newcommand\fL{\mathfrak{L}}
\newcommand\fR{\mathfrak{R}}

\theoremstyle{plain}
\newtheorem{theo}{Theorem}
\newtheorem{prop}{Proposition}

\theoremstyle{definition}

\newtheorem{rema}{Remark}

\allowdisplaybreaks

\title[Spectral expansions of Irwin sums]{%
  Block-count constrained harmonic sums: spectral expansion and block-directed Euler--Maclaurin}

\author[J.-F. Burnol]{Jean-François Burnol}

\date{15 September 2026.}


\usepackage{setspace}
\newcommand\burnolgitlab{\href{https://gitlab.com/burnolmath/irwin}{Irwin gitlab site}}

\begin{document}
\addtocontents{toc}{\protect\hypersetup{hidelinks}}

\begin{abstract}
  We determine exactly how harmonic sums, taken over integers with a given
  number of occurrences of a fixed block of digits, depend on that number of
  occurrences. Language factorizations, combined with stochastic radix
  expansions, relate these harmonic sums to the iteration of an operator whose
  eigenvectors are the Lebesgue measure and distributional derivatives of
  singular measures. The dual picture is given by a Riesz basis of
  eigenpolynomials in a suitable Hardy space. The resulting modal expansion
  admits a natural interpretation as a block-directed Euler--Maclaurin formula
  interpolating between Euler--Maclaurin and Taylor expansions.
\end{abstract}

\maketitle

\onehalfspacing

\section{Introduction}

Let $b$ be an integer, greater than one, and referred to below as the base or
radix. Let $w$ be a non-empty word, of length $p=|w|$, with letters from the
alphabet $\Sigma_b=\{0, \dots, b-1\}$ of the base-$b$ digits. We denote by
$k_w(n)$, when $n$ is a non-negative integer, the number of occurrences of $w$
in the minimal radix-$b$ representation of $n$ (terminological details are
reviewed in the next section; in particular occurrences are counted with
overlaps).  

The paper is devoted to the quantities
\begin{equation}
  \label{eq:1}
  I(b,w,k) = \sum_{n>0, \; k_w(n)=k} \frac1n\;,
\end{equation}
which we call \emph{block-count constrained} harmonic sums or
\emph{block-Irwin sums}. This notion originates, for $w$ a single digit, in
\cite{irwin} (and \cite{kempner} for $k=0$). Irwin actually also considered
imposing simultaneously the counts of distinct digits, and this can be
extended to counting occurrences of multiple distinct given blocks.  As the
size of this paper indicates, though, there is already much to say when
considering only a single block $w$.

Kempner merely observes that the series $I(10,9,0)$ converges, and bounds it
by $90$; Irwin contributes a tighter estimate $22.4<I(10,9,0)<23.3$.  The
convergence (for $b=10$) is indeed extremely slow, and Baillie
\cite{baillie1979} developed an algorithm allowing one to obtain many decimal
digits of each $I(10,d,0)$.  Later, Schmelzer and Baillie
\cite{schmelzerbaillie} contributed a numerical algorithm to evaluate
$I(b,w,0)$ also when $w$ is a multi-digit block (see also
\cite{baillie2008}). The $k>0$ case was not included for multi-digit
blocks.  The author provided in \cite{burnolone42} a representation of $I(b,w,k)$
for $p=2$ and $k\leq1$ allowing the numerical computation via geometrically
convergent series, extending the earlier work \cite{burnolirwin} which covered
$p=1$ for every $k\geq0$.

We extend this here to all cases of $w$ and $k$.  But the principal
result, stated next in a simplified form, goes beyond the extension of
\cite{burnolirwin} and \cite{burnolone42} to all cases. It provides
\emph{another} manner of computing $I(b,w,k)$, elucidating at the structural
level the dependence of $I(b,w,k)$ on $k\geq1$. The case $k=0$ is in a class
of its own; we will handle it fully for $w$ a single digit, and only present
an abridged summary for the multi-digit case. This was done in the hope of keeping
this paper size from becoming excessive.

\begin{theo}[Modal expansion of block-count constrained harmonic
  sums]\label{thm:intromain}
  For any non-empty word $w$ of length $p=|w|$ in radix $b$, there exist
  eigenvalues $\lambda_j(b,w)$, $j\geq1$, satisfying
  \begin{equation}
    1 = \lambda_1(b,w)>\lambda_2(b,w)>\dots>\lambda_j(b,w)>\dots>0,
    \qquad \lambda_j(b,w)\to_{j\to\infty} 0,
  \end{equation}
  and coefficients $\alpha_j(b,w)$, $j\geq2$, such that, for every $k\geq 1$,
  \begin{equation}
    \label{eq:main}
    I(b,w,k) = b^p\log(b) + \sum_{j=2}^\infty
    \alpha_j(b,w)\lambda_j(b,w)^k\,,
  \end{equation}
  with an absolutely convergent series.  The eigenvalues are rational numbers
  determined by $b$ and the lengths of $w$ and its borders.  If $w$ has no
  borders, in particular if $p=1$,
  $\lambda_j(b,w)=(b^{pj}-b^{pj-j+1}+1)^{-1}$.
\end{theo}
See Theorem~\ref{thm:part1main} in \autoref{sec:part1main} for $p=1$, and
Theorem~\ref{thm:part2main} in \autoref{sec:part2main} for the general
case. The former uses exponent $k+1$, and the latter exponent $k-1$,
the shift is here absorbed into the
coefficients.

Allouche, Hu, and Morin \cite{allouchehumorin2024} proved $\lim_{k\to\infty}
I(b,w,k)=b^p \log(b)$, which, of course, is implied by
Equation~\eqref{eq:main}.  For $p=1$, the limit $b\log(b)$ was known since
Farhi \cite{farhi}.  The author gave proofs \cite{burnolirwin, burnolblocks} of
the limit theorem using ideas very different from those of \cite{farhi} and
\cite{allouchehumorin2024} (which themselves are completely distinct from one
another): the Irwin series $I(b,w,k)$ are obtained by integrating $1/x$ on
$\Ifo{b^{-1}}{1}$ against discrete measures $\mu_k$, depending on $b$ and $w$.
The limit theorem is derived from the weak convergence of $(\mu_k)$ to $b^p
\Leb$, where $\Leb$ is the Lebesgue measure on $\Ifo01$.

The operator with eigenvalues $(\lambda_j)_{j\geq1}$ was first discovered in
an attempt to reinterpret the limit theorem as a fixed point theorem.  But
measures are not enough: beyond the first mode given by $\Leb$, the
eigenvectors are found among compactly supported distributions, and are
derivatives of singular measures.  The transpose action has corresponding
monic eigenpolynomials $\cB_m$, $m\geq0$.
The proof of Theorem~\ref{thm:intromain} will take us on a journey involving
probability theory, spectral theory, language combinatorics and even analytic
function theory.

The paper is divided into three parts. The first one thoroughly handles the
$p=1$ case. It includes some developments which are not needed for the main
theorem, such as concrete quantitative bounds of numerical interest. But its
large size originates mainly in the desire to articulate a full, coherent, logical
development, so that, for example, we do not encounter exponentially tilted
probability measures out of the blue, or, to give another example, we do not
suddenly start working in a certain Hardy space of analytic functions.  Part~2
handles the combinatorially vastly more complex case of multi-digit blocks.
As is frequent in mathematics though, the greater generality has forced the
author into actually simpler and sharper proofs of the core technical analytic
estimates.

A concentration of measure phenomenon discovered in Part~1 is encountered
again here. It helps set up a sufficient analytic framework in which
Theorem~\ref{thm:intromain} follows from a Riesz basis property of the
(suitably normalized) eigenpolynomials $\cB_m$, $m\geq0$.  Finally, in Part~3,
we reuse the latter result to provide an interpolation, associated with every
pair $(b,w)$, between Euler--Maclaurin and Taylor expansions, and which we find
remarkable in itself.

The generalization of $I(b,w,k)$ to sums of inverse powers naturally arises
while investigating Theorem~\ref{thm:intromain}, but only Part~1 really
insists on this. The existence of a meromorphic continuation to the whole
complex plane is known from a theorem of Allouche, Mendès-France
and Peyrière on general automatic Dirichlet series \cite{allouchemendespeyriere2000}.
Such analytic continuation is not used in the present paper.

We include here the table of contents, so that the reader may get, already
at this stage, a more precise picture of the material encountered along the way.

\singlespacing

\tableofcontents

\onehalfspacing

\section{Notation and preliminaries}

Throughout the paper $b>1$ is a fixed integer.  $\NN$ is the set of positive
integers and $\NN_0 = \{0\} \cup \NN$.  The alphabet of digits in base $b$ is
$\Sigma_b=\{0,\dots, b-1\}$.

The \emph{word space} $\cW$ is the disjoint union $\sqcup_{j\in\NN_0}
\Sigma_b^j$. A \emph{word} (or \emph{string}, or \emph{block}) is an element
of the word space.  We may sometimes however use the term \emph{infinite word}
(or string) to refer to a sequence of digits. For $g\in \cW$, $|g|$ is the
index $j$ such that $g\in\Sigma_b^j$, and is called its \emph{length}.  The
single word of zero length is denoted $\emptyword$ and called the \emph{empty
  word}.  A \emph{language} is any subset of $\cW$. Concatenation of $w_1$ and
$w_2$ is written as the juxtaposition $w_1w_2$, and for two languages,
$\cL_1\cL_2$ is the set of such concatenations. Also $w^j=w\dots w$, $j$
times, $w^0=\emptyword$, and analogously for languages, in particular
$\cL^0=\{\emptyword\}$ (and by convention this holds also for
$\cL=\emptyset$).

If $g=rs$, we say that $r$ is a \emph{prefix} of $g$ and $s$ is a
\emph{suffix} of $g$. We also refer to $s$ as the \emph{complement} of $r$ (in $g$).
A \emph{border} of $w$ is a \emph{non-empty} word which is both a prefix and a
suffix of $w$, and is \emph{not equal} to $w$ itself (we are following here
\cite[\S1.9, Exercise 43]{alloucheshallitCUP2003}).  A \emph{complement},
without other qualification, always means the complement in $w$ of a border of
$w$. An \emph{overlap period} is an integer $j$ with $0<j<|w|$ such that the
first $|w|-j$ digits of $w$ form a border. Equivalently, shifting $w$ $j>0$
places to the right and writing it underneath the original $w$ creates an
overlap of length $|w|-j>0$. In still other terms, there is a word of length
$|w|+j\in (|w|, 2|w|)$ having $w$ both as a prefix and as a suffix.  The
Guibas--Odlyzko auto-correlation polynomial \cite{guibasodlyzko1981b} is $1 +
\sum_{j\text{ an overlap period}}T^j$. A word with no border is called
\emph{unbordered}.

For a given non-empty word $w$, $k_w:\cW\to\NN_{0}$ is the counting function
of occurrences of $w$ in other words: an \emph{occurrence} of $w=w_1\dots w_p$
in $g=g_1\dots g_q$ is an index $j\geq1$ such that $j+p-1\leq q$ and
$g_{j+h-1}=w_h$ for $1\leq h \leq p$.  So, occurrences are counted with
overlaps allowed. We set $\cW_k = \{g \in \cW, k_w(g)=k\}$. 

We let $n(g)$ be the integer represented by $g$ (most significant digit on the
left), in particular $n(\emptyword)=0$. Among all words $g$ with
$n(g)=n\in\NN_0$, there is only one with no leading zeros. It is called the
\emph{minimal representation} of $n$, for $n=0$ it is $\emptyword$ and for
$n>0$ the \emph{most significant digit of $n$} is the leftmost digit of the
minimal representation. The notation $\overline{g}$ may also be used as a
synonym for $n(g)$.

By slight
abuse of notation $k_w(n)$ is defined to be $k_w(g)$ with $g$ the minimal
representation of $n$. In particular $k_w(0)=0$, even
when $w=0$.  Note that we do not make any special definition
for cases like this where $w$ itself starts with a zero.
Celebrated examples include counting $1$'s in base $2$, and counting $11$'s
also in base $2$. The former is the \emph{binary sum-of-digits} function,
whose parity gives the Thue--Morse sequence.  And the parity sign ($1$ for
even, $-1$ for odd) of the number of occurrences of $11$ in the binary
representations of the non-negative integers gives the Rudin--Shapiro
sequence.  See \cite{alloucheshallit1988} and \cite[\S3.2--\S3.3, \S6.2,
\S17.2]{alloucheshallitCUP2003}.

Let us define for any word $g$ of length $i$, $g=d_1\dots d_i$ the real number
$x(g) = \frac{d_1}b+ \dots +\frac{d_i}{b^i}$; so $x(\emptyword)=0$, $x(g)=0$
if and only if $g=0^j$ for some $j$, and $0\leq x(g)\leq 1 - b^{-|g|}$, the
upper bound being reached when all digits of $g$ are equal to $b-1$.  The
\emph{cylinder} $\Cyl_g$ is defined as $\Ifo{x(g)}{x(g) + b^{-|g|}}$. It is
included in $\Ifo01$. The left end-point does not by itself determine $g$, due
to trailing zeros, but together with the length of the interval, it does.

From \cite{burnolblocks} (originally \cite{burnolirwin} for $p=1$),
the measure $\mu_k$, which is shorthand notation for
$\mu_{b,w,k}$, is the sum of the weighted Dirac masses
$b^{-|g|}\delta_{x(g)}$, taken over all words $g\in\cW_k$.  Due to trailing
zeros, several, even infinitely many, of these Dirac masses may be located at
the same point $x\in \Ifo01$.  The definition of $\mu_k$ is motivated by the
formula
\begin{equation}
  I(b,w,k) = \int_{\Ifo{b^{-1}}1} \frac{\dmu_k(x)}{x}\;.
\end{equation}
This formula is both trivial and the fundamental starting point of the entire
paper.

It is proved at the end of \cite{burnolblocks} that for any $b$-cylinder
$\Cyl_g$, there holds $\mu_k(\Cyl_g) = b^p b^{-|g|}$ as soon as $k > |g|$.
This implies the weak convergence of $\mu_k$ to $b^p\Leb$, with some quantitative
estimate.  This result will be completely superseded by the present paper.

Here, $\Leb$ is the Lebesgue measure on $\Ifo01$.  The same symbol is also the
Lebesgue measure on $\Iff01$, as it proves convenient at a certain stage to
move from $\Ifo01$ to $\Iff01$, in particular to use the language
of distribution theory.

When an expectation or an integral contains a leading $\Un_{X_k\geq b^{-1}}$,
it is understood that it is restricted to the event $X_k\geq b^{-1}$. Thus,
the rest of the formula can contain $X_k^{-1}$ without any special provision.

In some cases, quantities depending on multiple variables,
e.g.\@ $H_s(b,d,k)$ may be referred to simply as $H_s$ if the context is
clear.

The Pochhammer symbol $(z)_j$ is $\prod_{0\leq i < j}
(z+i)$. In particular $(z)_0=1$.

The notation $x(g)$ is extended to infinite words, i.e.\@ to sequences
$(d_i)_{i\geq1}$ of digits.  It then refers to $\sum_{i=1}^\infty
d_ib^{-i}\in\Iff01$.  Often in Part~2, we will use notation such as
$x(G_1G_2\dots)$ where $(G_j)_{j\geq1}$ is a sequence in $\cW$.  It means
$x((d_i)_{i\geq1})$ where the digits are listed in the order of the words
$G_j$, $j\geq1$.  If only finitely many of the $G_j$'s are non-empty (which
will never happen in the paper), there are only finitely many digits and $x$
has its previous meaning.

In Part~1, a digit $d$ in radix $b$ is fixed. And
throughout Part~2, a word $w$ of length $p\geq1$ is fixed.  Occasionally, the
hypothesis $p>1$ may be made.

Note that all displayed equations are numbered. This is deliberate.

\section{Summable triangular perturbations}

The proof of the main result requires a Hilbert space lemma on Riesz bases.
As this is very separate from the main course of the discussion, we locate it
here.

We actually start with a Banach space context of Schauder bases.  The lemma we
prove may be viewed as a very special case of Ringrose's theorem
\cite{ringrose1962} on the quasi-nilpotence of compact strictly triangular
operators. The result, as stated below, is probably a textbook exercise, but
we have not been able to locate its precise statement in the literature.  For
a Hilbert space we give an even more specialized lemma, which will serve
in the sequel.
\begin{prop}\label{prop:schauderbanach}
  Let $E$ be a Banach space and let $(e_n)_{n\ge0}$ be a normalized
  (i.e.\@ $\|e_n\|=1$ for every $n$)
  Schauder basis of $E$. Let $f_n = e_n + u_n$ for $n\geq0$, with
  $\sum_n \|u_n\| < \infty$.  Assume further 
  either that $u_n\in\Span\{e_m:m<n\}$ for every $n$, or
  that $u_n\in\overline{{\Span}\{e_m:m>n\}}$ for every
  $n$. Then there exists a bounded invertible linear map $S:E\to E$ such
  that $Se_n=f_n$ for every $n\geq0$.
\end{prop}
\begin{proof}
  We handle first the lower-triangular case, i.e.\@
  $u_n\in\overline{{\Span}\{e_m:m>n\}}$ for every $n\geq0$.  Let
  $(e_n^*)_{n\geq0}$ be the coordinate functionals, and $R_n = u_n\otimes
  e_n^*$, i.e.\@ $R_n(x) = e_n^*(x)u_n$.  It is known that the $e_n^*$ are
  uniformly norm-bounded as linear forms, say $\|e_n^*\|\leq C$, $C<\infty$.
  So $\sum_{n\geq0} \|R_n\| <\infty$ by the hypothesis.  Consequently,
  denoting $I$ the identity operator on $E$,
  \begin{equation}
    S = I + \sum_{n=0}^\infty R_n
  \end{equation}
  is a bounded operator. It verifies $S(e_n) = f_n$ for every $n$.  We now
  prove that $S$ is invertible. For this, observe by the triangular
  hypothesis that $R_m(R_n(x))=0$ if $m\leq n$.  In particular, $(I+ R_0)(I +
  R_1) = I + R_0 + R_1$ and, inductively,
  \begin{equation}
    (I + R_0)\dots (I + R_N) = I +\sum_{n=0}^N R_n\,.
  \end{equation}
  Let $S_N$ be this operator, so $S$ is the limit in operator norm of $(S_N)$.
  Each factor $I+R_n$ is invertible, hence, $S_N$ is invertible and
  \begin{equation}
    S_N^{-1} = (I - R_N)\dots (I - R_0).
  \end{equation}
  Let $T_N = S_N^{-1}$. For $M>N$ we have
  \begin{align}
    T_M - T_N&= \Bigl((I - R_M)\dots (I - R_{N+1}) - I\Bigr)T_N\\
   \|T_M- T_N\|&\leq \Bigl((1 + \|R_M\|)\dots (1 + \|R_{N+1}\|) -
   1\Bigr)\|T_N\|\\
              &\leq \prod_{n\leq M}(1+\|R_n\|) - \prod_{n\leq N}(1+\|R_n\|).
  \end{align}
  So the Cauchy condition for operator-norm convergence is verified. Let
  $T=\lim T_N$.  As $T_NS_N = S_NT_N = I$ for every $N$ we obtain $TS=ST=I$.

  The proof for the upper-triangular case is almost identical, the sole
  differences being that it reverses the order of the products.
\end{proof}
With $E$ a Hilbert space and $(e_n)$ an orthonormal basis, the conclusion is
that $(f_n)$ is a Riesz basis of $E$.  We give an even more specialized
statement in that case (which is surely an exercise in textbooks).
\begin{prop}\label{prop:rieszhilbert}
  Let $E$ be a separable Hilbert space with scalar product $\scalp{\cdot}{\cdot}$
  (complex linear in the second vector).  Let $(e_n)_{n\ge0}$ be an
  orthonormal basis.  Let $f_n = e_n + u_n$ for $n\geq0$, such that the matrix
  $(\scalp{e_i}{u_j})_{i,j}$ is strictly upper or strictly lower triangular,
  and such that the sum of the moduli of its coefficients is finite.  Then
  $\sum_n\|f_n - e_n\|<\infty$, $(f_n)$ is a Riesz basis and the biorthogonal
  system $(f_n^*)$ is such that $\sum_n \|f_n^* - e_n\|<\infty$.
\end{prop}
\begin{proof}
  We have $\|u_n\|^2 = \sum_i |\scalp{e_i}{u_n}|^2 \leq (\sum_i
  |\scalp{e_i}{u_n}|)^2$, so $\sum_n \|u_n\|< \infty$. From Proposition
  \ref{prop:schauderbanach} there exists a bounded invertible operator $S$
  such that $S(e_n)=f_n$.  As
  $\delta_{n,m}=\scalp{f_m^*}{f_n}=\scalp{f_m^*}{Se_n}=\scalp{S^*(f_m^*)}{e_n}$
  one has $f_m^*=(S^*)^{-1}e_m$ and $f_m^* - e_m = (S^*)^{-1}(e_m - S^* e_m)$.
  We compute
  \begin{equation}
    \|e_m- S^*e_m\|^2
    = \sum_n |\scalp{e_n}{e_m - S^*e_m}|^2 
     = \sum_n |\delta_{n,m} - \scalp{Se_n}{e_m}|^2.
  \end{equation}
Moreover $\scalp{Se_n}{e_m} = \delta_{n,m} + \scalp{u_n}{e_m}$. So
\begin{equation}
  \|e_m- S^*e_m\| =\sqrt{\sum_{n} |\scalp{u_n}{e_m}|^2}\leq \sum_{n} |\scalp{e_m}{u_n}|.
\end{equation}
This implies $\sum_m \|e_m - S^*e_m\| < \infty$, and then $\sum_m
\|f_m^* - e_m\| < \infty$ as $(S^*)^{-1}$ is bounded.
\end{proof}
In the concrete cases we will consider, both $(\|e_m - f_m\|)$ and $(\|e_m -
f_m^*\|)$ will be shown to be bounded by sequences geometrically decreasing
to zero.  Assuming this property for one does not imply it for the other, so
it will have to be established for both.

\part{Return operator and the modal expansion of Irwin sums}

Throughout this Part, the word $w$ is a digit $d$.  The measures $\mu_k$ have
been defined in Notation and Preliminaries.  Integrating $\frac{\Un_{b^{-1}\leq
    x<1}(x)}{x}$ against them computes the Irwin sums $I(b,d,k)$.  The case
$(b,d,k)=(2,1,0)$ is special as it is the only one for which the restriction
of $\mu_0$ to $\Ifo{b^{-1}}{1}$ is the zero measure (indeed
$\mu_0=2\delta_0$). We will keep including that case for a while, but will
soon exclude it, so as to avoid having to extend statements again and again
with special provisions.
 
\section{Random variables}

An elementary direct computation (\cite{burnolirwin}) shows that $b^{-1}\mu_k$
is a probability measure on $\Ifo01$.  A nicer proof would be to
exhibit $b^{-1}\mu_k$ as the law of some suitable $X_k$. There is a very
nice construction of such a random variable, which we explain now.

First, we define random variables taking values
in the word space $\cW = \sqcup_{j\geq0} \Sigma_b^j$. Here is for $k=0$.
Consider infinitely many i.i.d.\@ random variables $d_i$, $i\geq1$, with
values in $\Sigma_b$, which are uniformly distributed there.  Define $\wX_0$ with
values in $\cW$ as follows: if $d_1=d$, then $\wX_0=\emptyword$; else, if
$d_2=d$, then $\wX_0 = d_1$; else, if $d_3=d$, then $\wX_0=d_1d_2$, etc. In
case none of the $d_i$ equals $d$, which happens with probability zero, we let
$\wX_0$ take the value $\emptyword$.  The words in the image of $\wX_0$ are
exactly those of the \emph{language of no-occurrence} of the digit $d$.  The
probability that $\wX_0$ takes such a value $g$ is $b^{-|g|-1}$. Indeed, if
$g=e_1\dots e_i$ is such a word, $\PP(\wX_0 = g) = \PP(d_1=e_1)\dots
\PP(d_i=e_i)\PP(d_{i+1}=d)$.  This is $b^{-1}$ times the measure defined in
\cite{burnolirwin} on words, whose push-forward to $\Ifo01$ under the map
$x:\cW\to\Ifo01$ is $\mu_0$.  We let $\wmu_0$ be that
measure on $\cW$ which is thus $b$ times the ``law'' of $\wX_0$.

How to define $\wX_k$ for $k>0$ is now clear: it is the longest word
accumulating the random digits until the $(k+1)$st occurrence of $d$, and this
terminal $d$ is not kept in the word.  If the random trials never provide
$k+1$ occurrences of $d$, then $\wX_k$ takes as value $\emptyword$. Then the
push-forward under $\wX_k$ of the probability space measure is $b^{-1}\wmu_k$.

We choose such i.i.d. random variables $d_i$, $i\in\NN$, as the functions on
$(\Ifo 01,\Leb)$, which give the digits of the canonical radix $b$ expansion
of $x\in\Ifo01$.  So the $\wX_k$ are defined on $(\Ifo01, \Leb)$ itself.  We
then compose with $x:\cW\to \Ifo{0}{1}$ to create real-valued random variables
$X_k$.  It proves convenient to alter $X_k$ so that $X_k = x \circ \wX_k$ is
only true almost surely: for those $x$ whose base-$b$ expansion do not have at
least $k+1$ occurrences of $d$, we set $X_k(x)=x$, and not $X_k(x)=0$. By
construction $\Law(X_k)$ is $b^{-1}$ times the measure $\mu_k$ defined in
\cite{burnolirwin}.  So always $X_0\leq X_1\leq \dots X_j \leq \dots \leq x$,
and $\lim X_j(x) = x$ for every $x$.

At this point we get a completely elementary, yet powerful, representation
of the Irwin harmonic sum as an expectation value:
\begin{equation}
  \label{eq:EEd}
  I(b,d,k) = b \EE(\frac{\Un_{X_k\geq b^{-1}}}{X_k}).
\end{equation}
Let us examine the event $A_k = \{X_k\geq b^{-1}\}$. If $k>0$, the definition of $X_k$
implies that one can recover $d_1$ as its first $b$-digit, so this is the same as
the event $\{d_1>0\}$. For $k=0$, we must distinguish $d=0$ from $d>0$. In
the former case, $\{X_0\geq b^{-1}\}$ exactly means that $d_1$ is not zero.
In the latter case, we must remove from $A=\{d_1>0\}$ the event $\{d_1=d\}$.
In particular $A_0$ is empty if and only if $(b,d)=(2,1)$. 
So Equation~\eqref{eq:EEd} gives:
\begin{align}
  \label{eq:EEd2a}
  I(b,d,k) &=
    b \int_{b^{-1}}^1 \frac{\dx}{X_k(x)} \qquad (k>0 \text{ or } k=0, d=0),
\\
\label{eq:EEd2b}
  I(b,d,0) &=  b \int_{x\in \Ifo{b^{-1}}1, \lfloor bx\rfloor \neq d}
                 \frac{\dx}{X_0(x)}
    \qquad (d>0),
\end{align}
As $(X_k)_{k\geq0}$ is a pointwise non-decreasing sequence with limit $x$, the
multiplicative inverses, which are bounded above by $b$ on $\Ifo{b^{-1}}{1}$
for $k>0$ or for $k=0$ with $d=0$, give a non-increasing sequence. Hence the
$(I(b,d,k))$ sequence is non-increasing for $k\geq1$, and also for $k\geq0$ if
$d=0$, and converges to $b\log(b)$.  More precisely, $X_{i+1}(x)>X_i(x)$
almost surely if $d>0$, and if $d=0$, we have $X_{i+1}(x)>X_i(x)$ for all $x$
such that there is one non-zero digit between the $(i+1)$st and the $(i+2)$nd
occurrences of $d$.  For example, this is the case everywhere on the
half-open interval $\Ifo{x_0}{x_0+b^{-i-3}}$, choosing
$x_0=0.1\underline{0\dots0}_{i+1\text{ zeros}}1$.  We conclude that
$I(b,0,i+1)<I(b,0,i)$, for any $i\in \NN_0$.

This is Farhi's Theorem~from \cite{farhi}.  It is striking that we have done
only a little bit of playing around with radix $b$ expansions and counting
occurrences, and then we obtained Farhi's theorem solely on the basis of
dominated convergence, with no computations whatsoever. The second proof given
in \cite{burnolirwin} was based on cylinder masses and was also almost
computation-free, but the new construction based on the random variables $X_k$
will allow us to establish precise quantitative estimates. This is our next goal.

\section{Return times and tail variables}

Let us write $U(x)=x$. Let $\tau_k\in\NN\cup\{\infty\}$ be the index of the
$(k+1)$st occurrence of digit $d$ in the radix $b$ expansion of $x$, or
$\infty$ if no such occurrence is realized.  We write (with the convention
$b^{-\infty}=0$):
\begin{equation}
  \label{eq:UXZ}
  U = X_k + b^{-\tau_k}(d+Z_k),
\end{equation}
where $Z_k(x)\in \Ifo01$ is uniquely determined if $\tau_k<\infty$, and is set
to the value zero if not.  We observe that $Z_k$ is independent of the pair
$(X_k,\tau_k)$ and furthermore $Z_k\sim\Unif(\Ifo{0}{1})=\Leb$.  Also we have
$0\leq Z_k < 1$ everywhere on the probability space $(\Ifo01, \Leb)$. And
$\tau_k\geq k+1$ for every $k\in \NN_0$.

The variable $\tau_0$ obeys a geometric law $\PP(\tau_0=j) =(1 - b^{-1})^{j-1}
b^{-1}$ for $j\in \NN$ ($\PP(\tau_0=\infty)=0$).  It has expectation value
$b$, finite moments of all orders, and its moment-generating function
(which converges for real-valued $z$ if $z<\log\frac{b}{b-1}$) is:
\begin{equation}
  \EE(\e^{z\tau_0}) = \frac{1}{b\e^{-z} - b +1 }\;.
\end{equation}
Let us define, for $s>\log_b(b-1)$:
\begin{equation}\label{eq:lambdas}
  \lambda_s = \EE(b^{-(s-1)\tau_0}) = \frac1{b^{s} - b +1}\;,
\end{equation}
In particular, for $m\in\NN$:
\begin{equation}\label{eq:lambdam}
  \lambda_m = \EE(b^{-(m-1)\tau_0}) = \frac1{b^{m} - b +1}\;,
\end{equation}
so $1= \lambda_1>\lambda_2>\lambda_3>\dots \to 0^+$.

We set $\tau_{-1}=0$.  The successive increments $\tau_{j}-\tau_{j-1}$,
$j\in\NN_0$, between return times are mutually independent and have the same
geometric law as $\tau_0$.  Hence $\EE(b^{-m\tau_k})=\lambda_{m+1}^{k+1}$.

\section{\texorpdfstring{Comparing $U^{-1}$ with $X_k^{-1}$}{Comparing 1/U
    with 1/Xk}}

The case $(b,d,k)=(2,1,0)$ is from here on excluded from the main flow of
presentation, and will be considered, if at all, only in a few scattered additional
remarks.

Let us compare $U^{-1}$ with $X_k^{-1}$, with the aim of integrating on the
event $A_k=\{X_k\geq b^{-1}\}$. This event is $\Ifo{b^{-1}}1$ for $k\geq1$ and
$\Ifo{b^{-1}}{1}\setminus \Ifo{db^{-1}}{(d+1)b^{-1}}$ for $k=0$. We let
\begin{equation}
  G_1(b,d,k) = \EE(\Un_{A_k}U^{-1}) =
  \begin{cases}
    \log(b) - \log(1+\frac1d)& (k=0, d>0),\\
    \log(b)& (\text{else}).
  \end{cases}
\end{equation}

Assuming only $X_k>0$, we can rewrite Equation~\eqref{eq:UXZ}, on the
full-measure set $\tau_k<\infty$, as
\begin{equation}\label{eq:nk}
  U = (1+q_k)X_k\qquad q_k = \frac{d+Z_k}{b n_k}\qquad n_k= b^{\tau_k-1}X_k.
\end{equation}
As $X_k>0$, $n_k$ is a \emph{positive} integer. In the left-over event
$\tau_k=\infty$, we set $Z_k=0$, $n_k=\infty$, $q_k=0$, so Equation
\eqref{eq:nk} always holds for $X_k>0$.  And it implies $0\leq q_k<1$, i.e.\@
$X_k\leq U<2X_k$ as soon as $X_k>0$.  We note in passing that $X_k>0$ implies
$\tau_k\geq2$: indeed, $\tau_k\geq k+1$ so $\tau_k(x)=1$ requires $k=0$, and
moreover implies $X_0(x)= 0$.  So we always have $n_k\geq b X_k$.

On $\{n_k>1\}\subset\{X_k>0\}$, we have $q_k<\frac12$.  For $n_k=1$, we have
$q_k<(d+1)b^{-1}$. So $\rho_k=\sup_{A_k} q_k<1$ except possibly if $d=b-1$ and
moreover $n_k=1$ is possible.  Assuming $d=b-1$, $n_k=1$ on $A_k$ can happen
only if $\tau_k=2$ and $X_k=b^{-1}$. As $\tau_k\geq k+1$, $k\in\{0,1\}$. For
$k=0$, this forces the first digit of the expansion, which is $1$, to not be
$d$.  So $b\geq3$.  Conversely, let $x_0=b^{-1}+(b-1)b^{-2}$.  On
$\Ifo{x_0}{x_0+b^{-2}}$, we have, for $d=b-1>1$, $\tau_0=2$, $X_k = b^{-1}$,
$n_k=1$, and $q_k$ can be arbitrarily close to $1$.  For $k=1$, the conditions
$n_1=b^{\tau_1-1}X_1=1$ and $X_1\geq b^{-1}$ imply $\tau_1=2$ and
$X_1=b^{-1}$. This means that $d=1$, and as we assumed $d=b-1$, we are in the
case $(b,d,k)=(2,1,1)$, and on the interval
$\Ifo{0.11_2}{1}=\Ifo{\frac34}{1}$, with $X_1=\frac12$. So $q_1$ can be
arbitrarily close to $1$.  In conclusion, the quantity $\rho_k=\sup_{A_k}
q_k\leq 1$ is $1$ if and only if $(b,d,k)$ is either $(b,b-1,0)$ with
$b\geq3$, or is $(2,1,1)$.

We note for future use that, in the enumerated cases where
$\rho_k=\sup_{A_k}q_k=1$, we have identified the interval corresponding to
${n_k=1}$. On its complement in $A_k$, we have
$0\leq q_k<\frac12$.

Whether or not $\rho_k<1$, we always have $q_k<1$ and can expand pointwise:
\begin{equation}\label{eq:invU}
  \Un_{X_k\geq b^{-1}}U^{-1} =
  \Un_{X_k\geq b^{-1}}\sum_{j=0}^\infty (-1)^j X_k^{-1-j} b^{-j\tau_k}(d+Z_k)^j\;.
\end{equation}
Partial sums are alternately pointwise upper and lower bounds for the
left-hand side.  These bounds are strict if $q_k>0$, which is certain if
$d>0$, and true almost surely on the event of integration if $d=0$.
Further, as $0\leq q_k<1$ all partial sums are bounded above pointwise by
$X_k^{-1}$.  Thus, in all cases, even those for which $q_k$ is not uniformly
bounded away from $1$, we can integrate termwise, which gives
\begin{equation}\label{eq:EEinvU}
  G_1(b,d,k) = b^{-1}I(b,d,k) + \sum_{j=1}^\infty 
     (-1)^j \frac{(d+1)^{j+1} - d^{j+1}}{j+1}
     \EE\bigl(\Un_{X_k\geq b^{-1}}\frac{b^{-j\tau_k}}{X_k^{j+1}}\bigr).
\end{equation}
We know from the preceding discussion that the absolute values of the terms
are strictly decreasing toward zero. And, in particular
\begin{equation}\label{eq:approx1}
  0 < I(b,d,k) - bG_1(b,d,k)  <  b (d + \frac12)
 \EE\bigl(\Un_{X_k\geq b^{-1}}\frac{b^{-\tau_k}}{X_k^{2}}\bigr).
\end{equation}
A (crude)  upper bound for the right-hand side is
\begin{equation}
  b^3(d+\frac12)\EE(b^{-\tau_k}) = \frac{b^3(d+\frac12)}{(b^2-b+1)^{k+1}}\;.
\end{equation}
This gives, for the first time, a quantitative version of Farhi's theorem,
which is in line with the numerical explorations of Baillie
(\cite{baillie2008}) in this $|w|=1$ case (for $|w|>1$, \cite{baillie2008}
only reports results on $k=0$).  We note the appearance of the
quantity $\EE(b^{-\tau_k}) = \lambda_2^{k+1}$ in this estimate.  Our next goal
is to show that it gives the correct dependence on $k$ indeed, for the first
order asymptotic of $I(b,d,k)-b \log(b)$ as $k$ tends to infinity.

\section{Higher inverse powers}

Let us examine the quantities $\EE\bigl(\Un_{X_k\geq
  b^{-1}}\frac{b^{-j\tau_k}}{X_k^{j+1}}\bigr)$ arising in Equation
\eqref{eq:EEinvU}.  For this purpose, it is convenient to use $\wX_k$ which
has values in the word space, and is such that $X_k = x\circ \wX_k$ (almost
surely). Then the expectation can be computed as an integral over the word
space against the push-forward measure $(\wX_k)_*(\Leb)$.  The words
contributing to that integral are those words $g$ whose first digit is
non-zero, and which contain exactly $k$ occurrences of the digit $d$.  And, by
definition of $\wX_k$, for $g=\wX_k(t)$, $0\leq t <1$, such that
$\tau_k(t)<\infty$, one has $\tau_k(t)=|g|+1$.  And $X_k(t) = x(\wX_k(t)) =
b^{-|g|}n(g)$.  So $b^{-j\tau_k(t)}X_k(t)^{-j-1} =
b^{-j|g|-j}b^{j|g|+|g|}n(g)^{-j-1}$.  As the weight of such $g$ in
$(\wX_k)_*(\Leb)$ is $b^{-|g|-1}$, the total expectation is exactly
$b^{-j-1}\sum^{(k)} n^{-j-1}$, where the sum is over all positive integers
whose minimal $b$-representation contains exactly $k$ occurrences of the digit
$d$.

Let us thus define in general, for real $s$ such that the series converges:
\begin{equation}
  \label{eq:Hsd}
  H_s(b,d,k) = \sum_{n>0\text{ contains exactly $k$ occurrences of $d$}} \frac1{n^s}\;.
 \end{equation}
 If $d>0$, there are $\binom{\ell-1}k(b-2)(b-1)^{\ell -k-1}+
 \delta_{k\geq1}\binom{\ell-1}{k-1}(b-1)^{\ell -k}$ admissible integers with
 $\ell$ digits, so $H_s$ converges if and only if $\sum_{\ell}P_{k,b}(\ell)
 (b-1)^\ell b^{-s \ell}<\infty$, for some non-zero polynomial $P_{k,b}$, with
 the sole exception of $b=2$, $d=1$, $k=0$.  That is, the series converges for
 real $s$ if and only if $s > \log_b(b-1)$, except when $(b,d,k)=(2,1,0)$ in
 which case the series is empty.  For $d=0$, similar counts show that
 there is no special provision and the condition for convergence is always $s
 > \log_b(b-1)$.

 Our earlier considerations for $s=j+1$, $j\in \NN_0$, apply similarly
 for any real $s$ such that $s> \log_b(b-1)$, and give the general identification
 \begin{equation}
   \label{eq:EEHs}
   H_{s}(b,d,k) = b^{s}\EE\bigl(\Un_{X_k\geq b^{-1}}\frac{b^{-(s-1)\tau_k}}{X_k^{s}}\bigr).
 \end{equation}
 % In this paper we do not consider the matter of analytic continuation. We will
 % mostly be interested in integer exponents but consider for now the general
 % real $s$ for which there is convergence, for the study of the large-$k$
 % behavior.
 We will soon combine this with Equation~\eqref{eq:EEinvU}, but first make a
 digression on some other quantities.

\section{The quantities \texorpdfstring{$G_s$}{Gs}}

It seems logical to consider the replacement of
$X_k^{-s}$ by $U^{-s}$ in Equation~\eqref{eq:EEHs}.  Let us set
 \begin{equation}\label{eq:EEGs}
   G_s(b,d,k) = \EE\bigl(\Un_{X_k\geq b^{-1}}\frac{b^{-(s-1)\tau_k}}{U^{s}}\bigr).
 \end{equation}
What we did for $s=1$ works similarly, using the Newton binomial series in
the identity $U^{-s} = X_k^{-s} \bigl(1 + q_k)^{-s}$ on the event
 $A_k = \{X_k\geq b^{-1}\}$.  We get
\begin{equation}\label{eq:invUs}
\Un_{A_k}  U^{-s} 
  = \Un_{A_k} \sum_{j=0}^\infty (-1)^j \frac{(s)_j}{j!}b^{-j\tau_k}X_k^{-s-j} (d+Z_k)^j\;.
\end{equation}
Pointwise, the partial sums give alternately upper and lower bounds (which are
strict if $d>0$ and almost surely so if $d=0$) for the left-hand side.  In
view of Equation~\eqref{eq:EEHs} we multiply everything by $b^{-(s-1)\tau_k}$ and
then take expectations.  If interchange of integration and summation is legitimate, we obtain:
\begin{equation}\label{eq:GsseriesEE}
G_s(b,d,k)
  = \sum_{j=0}^\infty (-1)^j \frac{(s)_j}{j!}b^{-s-j}H_{s+j}(b,d,k)\EE\bigl((d+Z_k)^j\bigr)\;.
\end{equation}
This is actually legitimate and gives a geometrically convergent series,
except possibly in the cases encountered in the discussion, prior to Equation
\eqref{eq:invU}, of when $\rho_k=\sup_{A_k}q_k$ can be equal to $1$.
The interchange of integration and summation is not always legitimate in those
cases (the integrated terms may fail to tend to zero), but we identified the
interval corresponding to the contribution of the integer $n=1$ to
$H_{s+j}(b,d,k)$ as the sole problematic part of the integration domain.
Consequently, after cutting out this interval, we again obtain an absolutely
convergent series, now expressed in terms of the
$H_{s+j}(b,d,k)-1$, $j\in\NN$, rather than the $H_{s+j}$ themselves.
The left-hand side is then not $G_s(b,d,k)$ but a slight modification of it,
which, in Equation~\eqref{eq:Gsseries} of the next section, would amount to
removing the contribution from $n=1$.  We skip the details, as we will not use
this in the sequel.

The alternating nature of the integrand means that in all cases, we are
guaranteed to obtain, depending on parity of the number of terms, either upper
or lower bounds of $G_s(b,d,k)$.  In particular:
\begin{equation}\label{eq:Hsapprox1}
  0
  <  b^{-s}H_s(b,d,k) - G_s(b,d,k)
< s b^{-s-1}(d+\frac12)H_{s+1}(b,d,k).
\end{equation}
A (crude) upper bound for the right-hand side, using only $X_k\geq b^{-1}$, is: 
\begin{equation}\label{eq:ub}
  s (d+\frac12)b^{s+1}\EE(b^{-s\tau_k}) = \frac{s(d+\frac12)b^{s+1}}{(b^{s+1}-b+1)^{k+1}}\;.
\end{equation}
We can improve this slightly by not removing the $\Un_{A_k}$ after
replacing $X_k^{-s-1}$ by $b^{s+1}$.  So we examine in general
$\EE\bigl(\Un_{A_k}b^{-s\tau_k}\bigr)$.  We decompose along
sub-intervals $\Ifo{ab^{-1}}{(a+1)b^{-1}}$, with $1\leq a$, and $a\neq d$ if
$k=0$.  Suppose first $d=0$. We obtain:
\begin{equation}\label{eq:ed0}
  \EE\bigl(\Un_{A_k}b^{-s\tau_k}\bigr) = b^{-1-s}\frac{b-1}{(b^{s+1}-b+1)^{k+1}}\;.
\end{equation}
For $d>0$, we get, if $k\geq1$:
\begin{equation}\label{eq:ed>0k1+}
  \EE\bigl(\Un_{A_k}b^{-s\tau_k}\bigr) = b^{-1-s}
\frac{b^{s+1}-1}{(b^{s+1}-b+1)^{k+1}}\;,
\end{equation}
and, for $k=0$:
\begin{equation}\label{eq:ed>0k0}
  \EE\bigl(\Un_{X_0\geq b^{-1}}b^{-s\tau_0}\bigr) = b^{-1-s}\frac{b-2}{b^{s+1}-b+1}\;.
\end{equation}
The upper-bound from Equation~\eqref{eq:ub} can thus be improved, depending on
the case, by a factor equal to the ratio of either $b-1$, $b^{s+1}-1$, or $b-2$
to $b^{s+1}$.  In all cases, except excluded $(b,d,k)=(2,1,0)$, the dependence
on $k$ of $\EE\bigl(\Un_{A_k}b^{-s\tau_k}\bigr)$ is entirely in the
factor $(b^{s+1}-b+1)^{-k}$.  Comparing this with the main approximation,
which is $\EE\bigl(\Un_{A_k}b^{-(s-1)\tau_k} U^{-s}\bigr)= (1 +
\theta_k (b^s-1))\EE(\Un_{A_k}b^{-(s-1)\tau_k})$ with some
$0<\theta_k < 1$, we can already assert (for fixed $b$, $d$, and $s$):
\begin{equation}
  H_s(b,d,k) \sim_{k\to\infty} b^sG_s(b,d,k) = 
  b^s\EE\bigl(b^{-(s-1)\tau_k}\Un_{A_k}  U^{-s}\bigr),
\end{equation}
and we can bound explicitly, if desired, the relative error.

\section{Series representation of \texorpdfstring{$G_s$}{Gs}}

We can compute $G_s(b,d,k)$, defined in Equation~\eqref{eq:EEGs}, as a discrete
series.  For this, we reformulate Equation~\eqref{eq:nk} into the form, which
is valid
for almost all $t$ in $\{X_k\geq b^{-1}\}$:
\begin{equation}
  t = b^{-\tau_k(t)+1}n_k(t) + b^{-\tau_k(t)}(d+Z_k(t)).
\end{equation}
This says exactly that $t$ belongs to the cylinder set defined by the word
$gd$, where $g$ is the minimal representation of the positive integer
$n=n_k(t)$. The possible $n$'s are the positive integers having exactly $k$
occurrences of the digit $d$, a condition which we write as $k_d(n)=k$. From
the above, $\Leb$-almost every $t\in \{X_k\geq b^{-1}\}$ falls into one of those
cylinders, which are pairwise distinct.  On this cylinder $b^{-\tau_k+1}$ is
constant and we can parametrize $t$ as $b^{-\tau_k}(bn + d+ u)$ with $0\leq u
<1$, so the cylinder contributes to $G_s(b,d,k)$ the quantity
\begin{equation}
  b^{-(s-1)\tau_k}\int_{\Cyl_{gd}}t^{-s}\dt = \int_0^1 (bn+d+u)^{-s} \du.
\end{equation}
Thus, for $s\neq 1$:
\begin{equation}\label{eq:Gsseries}
  G_s(b,d,k)= \frac{1}{s-1}\sum_{n>0, k_d(n)=k} \Bigl((bn+d)^{-s+1}-(bn+d+1)^{-s+1}\Bigr),
\end{equation}
and
\begin{equation}\label{eq:G1series}
  G_1(b,d,k) = \sum_{n>0, k_d(n)=k} \log(1 + \frac1{bn+d}),
\end{equation}
where $G_1(b,d,k)$ is $\log(b)$ if $k\geq1$, or if $k=d=0$, and is $\log(b) -
\log(1 + d^{-1})$ if $k=0$ and $d>0$.  Equation~\eqref{eq:G1series} is
immediately equivalent with the case $|w|=1$ of \cite[Prop.\@ 1]{allouchehumorin2024}.
Each term can be expanded in inverse powers of $bn+d$ and it is then
legitimate to regroup like powers and interchange the summations. We do not pursue
the matter here, as it is not indispensable to our continued discussion.

\section{Some singular measures and an extended Farhi theorem}

We know from Equation~\eqref{eq:EEHs} that Equation~\eqref{eq:EEinvU} involves
the Dirichlet series $H_{j}(b,d,k)$, which generalize to $j\geq2$ the Irwin
harmonic sum:
\begin{equation}\label{eq:G1andtheHm}
  G_1(b,d,k) = b^{-1}I(b,d,k) + \sum_{j=1}^\infty 
     (-1)^j \frac{(d+1)^{j+1} - d^{j+1}}{j+1} b^{-j-1}H_{j+1}(b,d,k).
\end{equation}
Recall that $G_1(b,d,k)$ is $\log(b)$ if $k\geq1$ or if $k=d=0$ and is
$\log(\frac{bd}{d+1}) $ for $k=0$, $d>0$. A quick check confirms that
expanding the right-hand side of Equation~\eqref{eq:G1series} termwise in
inverse powers of $n$ and collecting (not caring about justification) the like
powers does reconstitute Equation~\eqref{eq:G1andtheHm}.

The representation of the $H_j$'s as expectation values (Equation
\eqref{eq:EEHs}) and the knowledge of the quantities $\EE(\Un_{A_k}
b^{-j\tau_k})$ already informs us that the $H_j$'s define distinct scales for the
large-$k$ limit.  But so far we have not eliminated the possibility that
the ratio of $H_{j+1}$ to $b^{j+1} \EE(\Un_{A_k} b^{-j\tau_k})$
oscillates indefinitely between $1$ and $b^{j+1}$. Only for $j=0$ do we
know that there is actually a limit, which is $b\log(b)/(b-1)$, which indeed belongs to 
$\Ioo{1}{b}$.

Let us look more closely at $H_2 = b^2 \EE(\Un_{X_k\geq
  b^{-1}}b^{-\tau_k}X_k^{-2})$.  Recall that
$\EE(b^{-\tau_k})=\lambda_2^{k+1}$.  The problem is to handle the large-$k$
limit of the quotient $H_2/\lambda_2^{k+1}$. Consider the absolutely continuous measure
$\PP_{2,k}$ whose density with respect to $\Leb$ is $\lambda_2^{-k-1}b^{-\tau_k}$, so that $\PP_{2,k}$ is also a probability measure.  Then
\begin{equation}
  H_2 = b^2 \lambda_2^{k+1} \EE_{\PP_{2,k}}(\Un_{A_k}X_k^{-2}).
\end{equation}
But now both the integration
measure and the random variable vary with $k$.

Let us consider the cylinders of constancy of $X_k$, which (for $X_k\geq
b^{-1}$) are the same cylinders $\Cyl_{gd}$, where $k_d(g)=k$, used in the
reduction of $G_s$ to a discrete series. We are computing
\begin{equation}
  H_2 = b^2 \lambda_2^{k+1} \sum_{g=a_1\dots a_L\in \cW_k, a_1>0} 
                           \frac{\PP_{2,k}(\Cyl_{gd})}{x(g)^2}\;.
\end{equation}
Let $g$ be $a_1\dots a_L$, $L\geq 1$, $a_1>0$. The probability assigned by
$\PP_{2,k}$ to $\Cyl_{gd}$ is $b^{-2L-2}(b^2-b+1)^{k+1}$.  There is one
multiplicative factor $b^2 -b+1$ for each occurrence of $d$ in $gd$.  It is as
if we considered that each of the $L+1$ first digits of the radix $b$
expansion had the probability $(b^2-b+1)/b^2$ to be $d$, and $1/b^2$ to be any
given $a\neq d$, and that they were independent random variables.

The solution is now obvious: we make $\Iff{0}{1}$ into a new probability space
(with no atom at $1$) via the push-forward from the infinite product
$\Sigma_b^\NN$, where each $\Sigma_b$ has probability weights
\begin{equation}
  q_2(a)=
  \begin{cases}
    (b^2-b+1)b^{-2},& a=d,\\
    b^{-2},& a\neq d.
  \end{cases}
\end{equation}
Writing
\begin{equation}
  \pi(a_1,a_2,\ldots)=\sum_{i=1}^\infty a_i b^{-i}\,,
\end{equation}
we define (shortening the more precise notation $\nu_{b,d,2}$ to $\nu_2$):
\begin{equation}
  \nu_2=\pi_*\bigl(q_2^{\otimes\NN}\bigr).
\end{equation}
This is an exponential change of measure, also called exponential tilting
(cf.\@ \cite[\S1b, Eq.~(1.2)]{asmussenGlynn2007}; more specifically, for
change of measure at stopping times, see also
\cite[Ch.~XIII, \S3]{asmussen2003}).

It is mutually singular with $\Leb$ because the almost sure frequencies of the
digits under its law differ from those realized by $\Leb$.  The key point now
is that $\PP_{2,k}(\Cyl_{gd})=\nu_2(\Cyl_{gd})$: the dependence on $k$ has
magically disappeared.  It is important that the stopping times $\tau_k$ are
the same functions as earlier and that it is still true under $\nu_2$ that
$\tau_k<\infty$ almost surely. In other words, the cylinders $\Cyl_{gd}$ are
again, apart from a left-over set of zero measure also for $\nu_2$, a
partition of the event $\{X_k\geq b^{-1}\}$ (which is $\Ifo{b^{-1}}{1}$ if
$k\geq1$ or $k=d=0$).  We keep the notation $X_k$ even when considering this
function as a random variable under the new probability measure $\nu_2$.  Of
course, this is slightly ambiguous, but decorating the notation with an extra
symbol is heavy, and switching to another letter such as $Y_k$ will make
generalization from $s=2$ to any $s$ problematic.

Let $\EE_2$ be the expectation for the probability measure $\nu_2$.  From the
cylinder decomposition we have:
\begin{equation}\label{eq:H2nu}
  H_2 = b^2 \lambda_2^{k+1} \EE_2(\Un_{A_k}X_k^{-2}).
\end{equation}
Dominated convergence delivers then immediately:
\begin{equation}
  H_2(b,d,k)\sim_{k\to\infty} b^2 
           \lambda_2^{k+1}
          \int_{\Ifo{b^{-1}}{1}}\frac{\dnu_2(x)}{x^2}\;.
\end{equation}
Here is the theorem for general real $s$ for which $H_s$ converges.
\begin{prop}\label{prop:farhis}
  Let $s>\log_b(b-1)$.
  Let $\nu_s$ (or more precisely $\nu_{b,d,s}$) be the probability measure
  on $\Ifo01$ such that the digits of the radix $b$ expansion are i.i.d., with
  probability $(b^s-b+1)/b^{s}$ to take the value $d$, and 
  $b^{-s}$ to take any other given digit value $a\neq d$.  Let
  \begin{equation}
    \label{eq:Fs}
    F_s(b,d) = \int_{\Ifo{b^{-1}}{1}}\frac{\dnu_{s}(x)}{x^s}\;,
  \end{equation}
  The Dirichlet series with constrained block counts
  $H_s(b,d,k)$ has the following asymptotic as $k\to\infty$:
  \begin{equation}
      H_s(b,d,k)\sim_{k\to\infty} b^s \lambda_s^{k+1}F_s(b,d),
  \end{equation}
  where $\lambda_s = \int_0^1 b^{-(s-1)\tau_0(x)}\dx = (b^s - b + 1)^{-1}$.
 \end{prop}
\begin{proof}
  The proof is a direct extension of the one given for $s=2$.  First, one establishes:
\begin{equation}\label{eq:Hsnu}
  H_s = b^s \lambda_s^{k+1} \EE_s(\Un_{A_k}X_k^{-s}),
\end{equation}
then, the
  limit of $\lambda_s^{-k-1}H_s$ is obtained from dominated convergence.
\end{proof}
It is worthwhile to mention the following remarkable property underpinning
Equation~\eqref{eq:Hsnu}.  The masses which $\nu_s$ assigns to cylinders
$\Cyl_{gd}$ are, up to a common factor depending on $k_d(g)$ only, the $s$-th
powers of the cylinder widths:
  \begin{equation}\label{eq:cylnus}
    \forall g\in \cW, \quad \nu_s(\Cyl_{gd}) = \lambda_s^{-k_d(g)-1}\Leb(\Cyl_{gd})^{s}.
  \end{equation}
  Indeed, this mass is determined by the $L+1$ digits of $gd$ and depends only
  on the count of occurrences of $d$ in this word. Multiplying out the individual
  probabilities, we get $\nu_s(\Cyl_{gd})=b^{-s(L+1)}(b^s-b+1)^{k_d(g)+1}$.
  This is the stated formula.

As corollary to Equation~\eqref{eq:cylnus}, we mention the following notable fact:
\begin{equation}
  \EE_{s_1}(b^{-(s_2 - s_1)\tau_k}) = \lambda_{s_2}^{k+1}\lambda_{s_1}^{-k-1}.
\end{equation}
Indeed, the contribution to the left-hand side of a cylinder $\Cyl_{gd}$ with
$k_d(g)=k$ is $b^{-(s_2-s_1)(|g|+1)}\lambda_{s_1}^{-k-1}b^{-s_1(|g|+1)}$ and
the $s_1$ cancels out in the exponents. Note that the sign of $s_2 - s_1$ is
indifferent.

\section{\texorpdfstring{$F_s$ in terms of $H_{s+j}$, $j\geq0$}
         {Fs in terms of H(s+j), j>=0}}

We will assume that $k\geq1$ or that $k=d=0$, so
that $\{X_k\geq b^{-1}\}=\Ifo{b^{-1}}{1}$ and the left-hand side of Equation
\eqref{eq:G1andtheHm} is $\log(b)$.  We drop indicating explicitly $(b,d,k)$ to
shorten equation.  From the properties of Equation~\eqref{eq:G1andtheHm}, we have
\begin{equation}\label{eq:enc}
  -b^{-2}(d^2+d+\frac13)H_3<I(b,d,k) - b \log(b) - b^{-1}(d+\frac12)H_2< 0.
\end{equation}
Equation~\eqref{eq:Hsapprox1}, which is derived from Equation~\eqref{eq:invUs}
via integration against Lebesgue measure (after multiplication by
$b^{-(s-1)\tau_k}$ on both sides), gives in the case $s=2$:
\begin{equation}\label{eq:foo1}
  0 <  b^{-2}H_2- G_2 < 2 b^{-3}(d+\frac12)H_{3}.
\end{equation}
But this is not directly about the difference $b^{-2}H_2-
\lambda_2^{k+1}\EE_2\bigl(\Un_{A_k}U^{-2}\bigr)$.  What we really
want is to integrate Equation~\eqref{eq:invUs} with \emph{no} extra factor and
against $\nu_2$, not $\Leb$.  The key formula for us is indeed Equation
\eqref{eq:H2nu} and its general form~\eqref{eq:Hsnu}.

But first, we have to check what happens in general when we integrate
$\Un_{A_k}b^{-(s_1-s)\tau_k}X_k^{-s_1}$ against $\nu_s$, for two
allowed real numbers $s$ and $s_1$ (we will be applying this with $s_1=s+j$,
$j\in \NN_0$).  Let $g=a_1\dots a_L$, $a_1>0$, containing exactly $k$
occurrences of $d$. On $\Cyl_{gd}$ we have $X_k=n(g)b^{-L}$ and
$\nu_s(\Cyl_{gd})=\lambda_s^{-k-1}\Leb(\Cyl_{gd})^s=(b^s-b+1)^{k+1}b^{-(L+1)s}$. So,
\begin{equation}
  \EE_s(\Un_{\Cyl_{gd}}b^{-(s_1-s)\tau_k}X_k^{-s_1})=
\underbrace{b^{-(s_1-s)(L+1)}b^{s_1L}b^{-s(L+1)}}_{{}=b^{-s_1}}n(g)^{-s_1}\lambda_s^{-k-1},
\end{equation}
and a similar computation gives:
\begin{equation}
  \EE_{s_1}(\Un_{\Cyl_{gd}}X_k^{-s_1})=\lambda_{s_1}^{-k-1} b^{-s_1}n(g)^{-s_1}.
\end{equation}
Thus 
\begin{equation}
  \label{eq:ss1}
  \EE_s(\Un_{\Cyl_{gd}}b^{-(s_1-s)\tau_k}X_k^{-s_1}) = \lambda_s^{-k-1}\lambda_{s_1}^{k+1}
\EE_{s_1}(\Un_{\Cyl_{gd}}X_k^{-s_1}).
\end{equation}
Consequently:
\begin{equation}
  \label{eq:EEsHs1}
  \EE_s(\Un_{A_k}b^{-(s_1-s)\tau_k}X_k^{-s_1}) = \lambda_s^{-k-1}b^{-s_1}H_{s_1}.
\end{equation}
This is particularly pleasant: it means that the termwise integration against
$\nu_s$ of~\eqref{eq:invUs} gives terms expressed with some $(s,j)$-dependent
prefactor, the normalized Dirichlet values $b^{-s-j}H_{s+j}(b,d,k)$, moments
with respect to $\nu_s$ of $d+Z_k$, and a single common multiplicative factor
$\lambda_s^{-k-1}$ for all. This gives:
\begin{equation}\label{eq:FsHs}
  \lambda_s^{k+1}F_s = 
  \sum_{j=0}^\infty (-1)^j \frac{(s)_j}{j!} b^{-s-j}H_{s+j}\int_0^1(d+x)^j\dnu_s(x).
\end{equation}
We used of course that the tail $Z_k$ from Equations~\eqref{eq:UXZ} and
\eqref{eq:nk} is $\nu_s$-independent from $X_k$ and $\tau_k$, and that its law
on $(\Ifo{0}{1},\nu_s)$ is $\nu_s$.  To be precise, as partial sums are
alternately upper and lower bounds for the left-hand side, the series
converges if and only if its general term goes to zero.  Recall from the
discussion after Equation~\eqref{eq:invU}, that as $k\geq1$, or $k=d=0$, the
original pointwise series is absolutely and uniformly convergent on
$A_k=\{X_k\geq b^{-1}\}$, except possibly if $(b,d,k)=(2,1,1)$.  It can
actually be shown in the latter case that the integrated term does \emph{not}
tend to zero if $s\geq \log_2(3)$ (in particular for $s=2$). We don't need
this, so we skip it.  Also, the arguments from the discussion following
Equation~\eqref{eq:GsseriesEE} apply here identically and show tha we can
still obtain an absolutely convergent series in this bad case, but using
$H_{s+j}-1$ instead, $j\in\NN$, and with a slightly modified left-hand side.
As soon as $k\geq2$, there is no special case to consider, so we skip details.

There is an alternative to Equation~\eqref{eq:FsHs} which requires no special
provisions concerning convergence.  In place of $U = X_k(1+q_k)$ from Equation
\eqref{eq:nk}, we set (still assuming $X_k>0$), $Y_k = X_k + b^{-\tau_k}d$,
$n_k'= b n_k +d \geq b+d$, and $U = Y_k(1+q_k')$, $0\leq
q_k'=\frac{Z_k}{n_k'}<(b+d)^{-1}$. We obtain, details being left to the
reader, a geometrically convergent series ($k>0$ or $k=d=0$):
\begin{equation}\label{eq:FswhHs}
  \lambda_s^{k+1}F_s = 
  \sum_{j=0}^\infty (-1)^j \frac{(s)_j}{j!} \Bigl(\int_0^1x^j\dnu_s(x)\Bigr)
             b^{-s-j}\widetilde{H}_{s+j}\,,
\end{equation}
with
\begin{equation}
  \label{eq:whHs}
  \widetilde{H}_s(b,d,k) = \sum_{\substack{n>0\\k_d(n)=k}} \frac{1}{(n+\frac{d}{b})^s}\;.
\end{equation}
For $k=0$ with $d>0$, one needs to replace $F_s$ in the left-hand side of
Equation~\eqref{eq:FswhHs} by an integral excluding the interval
$\Ifo{db^{-1}}{(d+1)b^{-1}}$.  Let us define:
  \begin{equation}
    \label{eq:Fs*}
    F_s^*(b,d) = \int_{\Ifo{b^{-1}}{1}\setminus\Ifo{db^{-1}}{(d+1)b^{-1}}}\frac{\dnu_{s}(x)}{x^s}\;,
  \end{equation}
One can also modify the definition of
$\widetilde{H}_s(b,d,0)$ with an extra contribution given by $n=0$. This can
be done in a way covering all situations via the summation conditions $n\geq
0$, $k_d(bn+d)=k+1$ (note that $k_d(0)$ is defined as $k_d(\emptyword)=0$ and
cannot be $k+1$ if $bn+d=0$), and one checks that the series expressing $F_s$
in terms of these $\widetilde{\widetilde{H}}_s(b,d,k) = \sum_{n\geq0,
  k_d(bn+d)=k+1}(n+db^{-1})^{-s}$ always converges (but only conditionally if
$b\geq3$, $d=1$, $k=0$; and it would diverge for $s\geq\log_2(3)$ with
$(b,d,k)=(2,1,0)$ which was excluded earlier). We will make no
use of this alternative approach in the sequel (but will encounter again the
quantities $F_s^*(b,d)$).

\section{Effective estimates for Irwin sums for large \texorpdfstring{$k$}{k}}

The web site \burnolgitlab{}
provides software to compute $I(b,d,k)$ to very high precision, but the
algorithm is not appropriate for large $b$ or large $k$ (we will have
occasion to recall later the underlying formulas from \cite{burnolirwin}).
In this section we provide a quantitative refinement of Farhi's limit
theorem. The resulting bounds are enough for any given fixed-point precision,
for all sufficiently large $k$.

Let us temporarily restrict again to $k>0$ or $k=d=0$ and return to Equation
\eqref{eq:FsHs}, which is formal in the exceptional case $(b,d,k)=(2,1,1)$.
Using it for $s=1, 2, 3, \dots$, we get a triangular system relating the
$\lambda_m^{k+1}F_m(b,d)$, $m\in \NN$, with the $b^{-m}H_m(b,d,k)$:
\begin{equation}\label{eq:triang}
   \forall m\in \NN\quad
\lambda_m^{k+1}F_m  = 
  \sum_{j=0}^\infty (-1)^j\frac{(m)_j}{j!} \Bigl(\int_0^1(d+x)^j\dnu_m(x)\Bigr)
  b^{-m-j}H_{m+j}.
\end{equation}
The fact that partial sums give alternately upper and lower bounds for the
left-hand side applies even with $(2,1,1)$, so the following variant of
\eqref{eq:foo1}, directly relevant to the limit theorem Proposition
\ref{prop:farhis}, holds:
\begin{equation}\label{eq:H2enc}
    0
    <  b^{-2}H_2- \lambda_2^{k+1}F_2
    < 2 \bigl(d + \frac{d+\frac12}{b}\bigr)b^{-3}H_{3}
\end{equation}
We used the fact that the first moment of $\nu_2$ is $(b-1)^{-1}((b^2 - b+1)d+\sum_{a\neq
  d}a)b^{-2}$ which simplifies to $(d+\frac12)/b$.  We will have more
  to say about such moments later on.

Combining Equations~\eqref{eq:enc},~\eqref{eq:H2enc}, we get
\begin{align}
  I(b,d,k) &= b \log(b) + b(d+\frac12)\lambda_2^{k+1}F_2 + r_k,\\
\label{eq:epsilonk}
\frac{r_k}{b^{-3}H_3}&\in 
  \Ioo*{-b(d^2+d+\frac13)}{2b(d+\frac12)\bigl(d + \frac{d+\frac12}{b}\bigr)}.
\end{align}
and there only remains to bound
\begin{equation}
  b^{-3}H_3 = \lambda_3^{k+1}\int_{b^{-1}}^1 \frac{\dnu_3(x)}{X_k(x)^3}\;.
\end{equation}
On the interval $\Ifo{ab^{-1}}{(a+1)b^{-1}}$, we have $X_k\geq ab^{-1}$:
indeed, by the definition of $X_k$ this could fail only if $X_k=0$, which can
happen only for $k=0$ and $a=d>0$.  And:
\begin{equation}
  \nu_{b,d,3}\bigl(\Ifo{ab^{-1}}{(a+1)b^{-1}}\bigr) =
  \begin{cases}
    b^{-3}  & (a\neq d),
\\
    b^{-3} (b^3-b+1) &(a=d).
  \end{cases}
\end{equation}
So, let
\begin{equation}\label{eq:Theta}
  \Theta_{3}(b,d)
  =
  (b^3 - b +1)\Un_{d>0}d^{-3} + \sum_{\substack{1\leq a\leq b-1\\a\neq d}}a^{-3}
  \,.
\end{equation}
It follows, using again abridged notation for the indices, that
\begin{equation}\label{eq:H3bound}
  b^{-3} H_3 \leq \lambda_3^{k+1}\Theta_{3}(b,d).
\end{equation}
The inequality is strict except for $(b,d,k)=(2,1,1)$.
Combined with Equation~\eqref{eq:epsilonk}, this gives an explicit interval
where $r_k$ resides.  Its width is a multiple of $\lambda_3^{k+1}$, which
roughly means that we can assert the value of $I(b,d,k)$ with three more
$b$-digits of precision at each increase of $k$ by $1$.  For example, with
$b=16$, the formula is enough to give a float double precision estimate of
$I(16,d,k)$ already for $k\geq 5$.  But only if we know how to
evaluate numerically $F_2(16,d)$!  We will turn to this later.

For the cases $k=0$, $d>0$, which we have not handled yet, the difference is
that all integrals (before rescaling) are on $\Ifo{b^{-1}}1 \setminus
\Ifo{db^{-1}}{(d+1)b^{-1}}$ not on $\Ifo{b^{-1}}1$, in other terms we use the
quantities $F_s^*(b,d)$, not $F_s(b,d)$.  The comparison point for
$b^{-1}I(b,d,0)$ is now $\log(b) - \log(1+d^{-1})=F_1^*(b,d)$, the one for
$b^{-2}H_2(b,d,0)$ is $\lambda_2 F_2^*(b,d)$.  The estimate for $H_3$ is also
slightly modified.  We skip the details and give the result:
\begin{align}
  I(b,d,0) &= b \log(\frac{bd}{d+1}) 
    + b(d+\frac12)\lambda_2
      F_2^*(b,d) + r_0,\\
  \frac{r_0}{b^{-3}H_3}&\in \Ioo*{-b(d^2+d+\frac13)}
                             {2b(d+\frac12)\bigl(d + \frac{d+\frac12}{b}\bigr)},
\\
b^{-3}H_3(b,d,0) &< \lambda_3\sum_{1\leq a < b, a\neq d}a^{-3}\,.
\end{align}
As stated earlier, the ``empty case'' $(b,d,k)=(2,1,0)$ is excluded (so that we
can write strict inequalities).

\section{First glimpse of the modal series: asymptotics}
\label{sec:part1main}

In the previous section, Equation~\eqref{eq:triang} gave, for $k>0$ or
$k=d=0$, a triangular system relating (only formally in the exceptional case $b=2$,
$d=1$, $k=1$) $\lambda_m^{k+1}F_m(b,d)$ to $b^{-n}H_n(b,d,k)$, $n\geq m>0$. It
will allow us to decouple, in principle, $k$ from $(b,d)$ in the value of
$I(b,d,k)=H_1(b,d,k)$.  We assume $k\geq2$, since we are interested here in
the large-$k$ behavior.  But even in the special cases discussed earlier, we
still have estimates for the finite truncation remainders, so the big-$O$ in the
next paragraph could be made quantitive in the whole $k\geq1$ range ($k=0$
with $d>0$ requires the usual alteration of $F_m$ into $F_m^*$).

\begin{prop}\label{prop:asymp}
  There exist uniquely determined $c_j(b,d)\in\QQ(b)[d]$, $j\geq1$, such that
  there is an asymptotic expansion as $k\to\infty$:
  \begin{equation}\label{eq:asymp}
    b^{-1}I(b,d,k) \sim_{k\to\infty} \sum_{j=1}^\infty c_j(b,d)F_j(b,d)\lambda_j^{k+1}.
  \end{equation}
\end{prop}
\begin{proof}
Let $J>1$ and define:
\begin{equation}
  \begin{gathered}
    \sF_J =
    \begin{pmatrix}
      \lambda_1^{k+1}F_1\\\vdots\\\lambda_J^{k+1}F_J
    \end{pmatrix}, \quad \sH_J =
    \begin{pmatrix}
      b^{-1}H_1\\\vdots\\b^{-J}H_J
    \end{pmatrix}, \quad \sQ_J = (q_{ij})_{\substack{1\leq i \leq J\\1\leq
        j\leq J}}, 
   \\ q_{ij} = \delta_{j\geq
      i}(-1)^{j-i}\binom{j-1}{i-1}\EE_i\Bigl((d+U)^{j-i}\Bigr)
  \end{gathered}
  % \end{split}
\end{equation}
Writing the first $J$ identities as $\lambda_j^{k+1}F_j = (\sQ_J\sH_J)_j +
(\sR_J)_j$, we know that $(\sR_J)_j = O(H_{J+1})$ for every $j\in \{1, \dots,
J\}$.  Acting now with the inverse matrix we get, using its first row in particular
\begin{equation}
  (\sQ_J^{-1}\sF_J)_1 = b^{-1} I(b,d,k) + O(H_{J+1}).
\end{equation}
The implicit constant in $O(H_{J+1})$ may depend on $J$ (the pair
$(b,d)$ is fixed throughout that proof).  Of course, $\sQ_J^{-1}$
is the same as the $J\times J$ upper-left block of any $\sQ_{J'}^{-1}$ for
$J'>J$.  Let us write the first row of the inverse matrix in extension:
\begin{equation}
  \label{eq:cjbd}
  (\sQ_J^{-1})_1 = (c_j(b,d))_{1\leq j\leq J}\,.
\end{equation}
The matrix elements $q_{ij}$ are equal to $1$ on the diagonal, vanish below
the diagonal, and may be considered to be elements of $\QQ(b)[d]$.  Indeed,
this is the case of the moments $M_{i;k}$, $k\geq0$, of $\nu_i$ for every
$i\geq1$.  This is a consequence of the recurrence which they verify (which is
to be found in Equation~\eqref{eq:Mnmrec} in a later section and can be solved
in $\QQ(b)[d]$).  We know that $H_{J+1} =
O_{k\to\infty}(\lambda_{J+1}^{k+1})$, and $\lambda_1>\lambda_2>\dots$. This is
sufficient to provide the stated result.
\end{proof}
In the sequel we explain how to compute the coefficients $c_j(b,d)\in\QQ(b)[d]$,
and how to evaluate numerically
the $F_j(b,d)$ to arbitrary precision. Numerical comparison of the values
obtained from the series in Proposition~\ref{prop:asymp} with those computed
using the \burnolgitlab{} software, suggested that the
series was not only asymptotic but in fact convergent, \emph{for every
$k\geq1$ and also for $k=d=0$}. This constitutes the core result of the
present paper, which we will extend to multi-digit words in the second Part:
\begin{theo}[The modal expansion of Irwin sums]%
\label{thm:part1main}
For every $k\geq1$, and also for $k=d=0$, $b^{-1}I(b,d,k)$ is exactly
represented by the absolutely convergent series
  \begin{equation}\label{eq:exact}
    b^{-1}I(b,d,k) = \sum_{j=1}^\infty c_j(b,d)F_j(b,d)\lambda_j^{k+1}\,.
  \end{equation}
Further,
for $k=0$ and $d>0$, there holds, with the same $c_j(b,d)$:
  \begin{equation}\label{eq:exact*}
    b^{-1}I(b,d,0) = \sum_{j=1}^\infty c_j(b,d)F_j^*(b,d)\lambda_j.
  \end{equation}
again with an absolutely convergent series.
\end{theo}
We use ``modal'' in an, as yet, rather vague sense of ``eigenmodes''. Indeed,
Equation~\eqref{eq:exact} has a strong operator-theoretic flavor. It is
reminiscent of an eigenvalue expansion: it suggests the existence of a
(compact) operator $\cK$ (depending on $(b,d)$, as does everything else in
this paragraph) acting on a suitable Banach space, with a Schauder basis of
eigenvectors and associated eigenvalues $\lambda_m$, $m\geq1$, together
with a vector $v$ and a bounded linear functional $L$, such that
$I(b,d,k)=L(\cK^{k+1}v)$.  Further, the numerical explorations indicate that
the modal series are absolutely convergent.  This would follow from the
suitably normalized eigenvectors of $\cK$ forming an unconditional Schauder
basis.  The most convenient framework for this would be to have a Riesz basis
in a Hilbert space, so that the absolute convergence property of the series in
Equation~\eqref{eq:exact} follows from it being the scalar product between two
elements in $\ell^2$.  On the other hand, it will emerge quickly that it is
not the singular measures $\nu_m$, $m>1$, which can play the role of
eigenvectors: the Banach space of Borel measures on $\Ifo01$ only provides the
first eigenmode, which is $\Leb$.  To be more precise, $\Leb$ will be seen as
left eigenvector, and the other left eigenvectors will be the
$D^{m-1}(\nu_m)$, where $D$ denotes distributional differentiation, and the
right $\lambda_m$-eigenvector will be a monic polynomial $\cB_{m-1}$ (short
for $\cB_{b,d,m-1}$), of degree $m-1$.

But to motivate the actual constructions, and the distributional derivatives
$D^{m-1}(\nu_m)$ in particular, we have to recall some aspects of
\cite{burnolirwin} and \cite{burnollargebirwin}.  Indeed, the original route
which led the author to Proposition~\ref{prop:asymp}, and the architecture for
its numerical exploration, which suggested the validity of Theorem
\ref{thm:part1main}, already contain all the ingredients needed to suggest some
realizations of the operator $\cK$.  The eventual choice of a vector space,
which will turn out to be a Hilbert space (depending on both $b$ and $d$), for
which Equations~\eqref{eq:exact} and~\eqref{eq:exact*} can be proved through a
Riesz basis of eigenvectors will emerge from a closer look at the behavior of
the singular measures $\nu_m$ for large $m$ and its implications for the
(already known at this stage) $\cB_{m-1}$'s.

\section{Integration lemma, moment recurrences, and Stieltjes functions}

We recall here some facts from \cite{burnolirwin} and
\cite[\S2]{burnollargebirwin}.  The moments of the measure $\mu_k$ are denoted
\begin{equation}
  u_{k;m}=\int_{\Ifo{0}{1}} x^m\,\dmu_k(x),
  \qquad m\geq 0.
\end{equation}
Every measure $\mu_k$  has mass $b$, i.e.\@ $u_{k;0}=b$.

The integration lemma from \cite[\S2, Eqs.\@ (6)--(7)]{burnollargebirwin}
states that for every bounded (or more generally, integrable with respect to 
the measure in the left-hand side below)
function $g$ on $\Ifo{0}{1}$ one
has, for $k=0$,
\begin{equation}
  \int_{\Ifo{0}{1}} g(x)\dmu_0(x) = g(0)
  + \frac{1}{b}\sum_{\substack{0\leq a<b\\ a\neq d}} 
    \int_{\Ifo{0}{1}} g\left(\frac{a+x}{b}\right)\dmu_0(x),
  \label{eq:intlemk0}
\end{equation}
whereas, for $k\geq 1$,
\begin{equation}
  \begin{split}
    \int_{\Ifo{0}{1}} g(x)\dmu_k(x)
    ={}&
    \frac{1}{b}
    \sum_{\substack{0\leq a<b\\ a\neq d}}
    \int_{\Ifo{0}{1}} g\left(\frac{a+x}{b}\right)\dmu_k(x)
    \\
    &+
    \frac{1}{b}
    \int_{\Ifo{0}{1}} g\left(\frac{d+x}{b}\right)\dmu_{k-1}(x).
  \end{split}
  \label{eq:intlemkpos}
\end{equation}
In the earlier paper \cite[Lem.\@ 1]{burnolirwin}, this was written in terms of the
function $f$ on $\Ifo0b$ which is related to the $g$ used here via
$g(x)=f(bx)$.

Taking $g(x)=x^m$ gives triangular recurrences for the moments. Write
\begin{equation}
  \gamma_j = \sum_{\substack{0\leq a<b\\ a\neq d}} a^j.
  \label{eq:gamma}
\end{equation}
For $m\geq 1$, the $k=0$ recurrence is
\begin{equation}
  \bigl(b^{m+1}-b+1\bigr)u_{0;m}
  =
  \sum_{j=1}^{m} \binom{m}{j}\gamma_j u_{0;m-j},
  \label{eq:umrec0}
\end{equation}
which is Equation~(4) of \cite[\S5]{burnolirwin}. For $k\geq 1$ one has
\begin{equation}
  \begin{split}
    \bigl(b^{m+1}-b+1\bigr)u_{k;m}
    ={}&
    \sum_{j=1}^{m}
    \binom{m}{j}\gamma_j u_{k;m-j}
    \\
    &+
    \sum_{j=0}^{m}
    \binom{m}{j}d^j u_{k-1;m-j}.
  \end{split}
  \label{eq:umreck}
\end{equation}
This is \cite[Prop.\@ 6, Eq.\@ (5)]{burnolirwin}.

Following \cite[\S2, Eq.\@ (3)]{burnollargebirwin}, define the Stieltjes
functions $U_k$ as follows:
\begin{equation}\label{eq:Ukz}
  U_k(z)
  =
  \int_{\Ifo{0}{1}} \frac{\dmu_k(x)}{z+x},
  \qquad
  z\in\mathbb{C}\setminus[-1,0].
\end{equation}
It should be kept in mind that the notation
$U_k$ is shorthand for $U_{b,d,k}$, as is the case with the measures
$\mu_k$. We will use $U_k$ only with positive integer arguments.
For such $n\in\NN$, Proposition~4 of
\cite{burnolirwin} identifies $U_k(n)$ with the reciprocal sum over the
positive integers whose base-$b$ expansion starts with $n$ and whose remaining
digits contain exactly $k$ occurrences of $d$. In particular, decomposing the terms contributing to $I(b,d,k)$ according to their leading digit, we obtain:
\begin{equation}\label{eq:IU}
  I(b,d,k)
  =
  \sum_{\substack{1\leq a<b\\ a\neq d}} U_k(a)
  +
  \mathbf{1}_{\{d>0,\ k\geq 1\}}U_{k-1}(d).
\end{equation}
This is also what Equations~\eqref{eq:intlemk0} and
\eqref{eq:intlemkpos} say with $g(x) = \Un_{x\geq b^{-1}}x^{-1}$.
For $n>1$, each $U_k(n)$ can be computed as a geometrically convergent
series, once the moments have been
numerically evaluated from their triangular recurrence:
\begin{equation}
  U_k(n)
  =
  \sum_{m=0}^{\infty} \frac{(-1)^m u_{k;m}}{n^{m+1}}\;.
  \label{eq:Umoments}
\end{equation}
This also applies with $n=1$ but may be only a semi-convergent series in that case.
See \cite[Cor.\@ 1]{burnolirwin}.

Applying~\eqref{eq:intlemk0} and
\eqref{eq:intlemkpos} to the Stieltjes kernel gives
\begin{equation}
  U_0(n)
  =
  \frac{1}{n}
  +
  \sum_{\substack{0\leq a<b\\ a\neq d}}
  U_0(bn+a),
  \label{eq:U0raise}
\end{equation}
and, for $k\geq 1$,
\begin{equation}
  U_k(n)
  =
  U_{k-1}(bn+d)
  +
  \sum_{\substack{0\leq a<b\\ a\neq d}}
  U_k(bn+a).
  \label{eq:Ukraise}
\end{equation}
See \cite[\S2, Eqs.\@ (4)--(5)]{burnollargebirwin}.  This allows to reduce the
evaluation of any $U_k(n)$ to finitely many other Stieltjes transform values,
with indices $j\leq k$, at arbitrarily large integers, with also perhaps some
terms $N^{-1}$ coming from Equation~\eqref{eq:U0raise}.  This is the
``level-raising'' method which is the basis of the \burnolgitlab{} algorithm
for the numerical evaluation of $I(b,d,k)$.  It allows small-argument
Stieltjes functions to be moved to arbitrarily large arguments (up to the
price of an exponentially increasing number of terms as a function of the
number of raising steps, in general).

\section{Modal expansion of the moments}

Equations~\eqref{eq:umrec0} and~\eqref{eq:umreck} allow us to compute the moments
inductively, starting from $u_{0;0} = b$.  In particular $u_{k;0}= b$ for every $k\geq1$ follows.  And, at the next step, the equations are 
\begin{align}
u_{0;1} &= \frac b2 -(d+\frac12)\frac b{b^2 - b +1}\;,\\
u_{k;1} &= \frac{b(b^2-b)}{2(b^2 - b +1)} + \frac{u_{k-1;1}}{b^2 - b + 1}\quad (k\geq1)\,,
\end{align}
and the solution (\cite[Prop.\@ 2]{burnollargebirwin}) is, for every $k\geq0$,
\begin{equation}
  \label{eq:uk1}
  u_{k;1} = \frac b2 - \frac{b(2d+1)}{2(b^2-b+1)^{k+1}}\;.
\end{equation}
This formula was one of the key facts which originally motivated the author to
search for an identity such as the one from Theorem~\ref{thm:part1main}.  Keeping from
the already known results only the idea of decoupling the dependence on $k$
from the one on $(b,d)$, let us achieve it here for the moments. First, we
observe that $\frac b2$ is the first moment of $b\Leb$.  And modifying
Equation~\eqref{eq:umreck}, $k\geq1$, under the Ansatz $u_{k;m}= \theta_m$,
one finds the linear recurrence
\begin{equation}
  (b^{m+1} - b +1) \theta_m =
  \sum_{j=1}^m\binom{m}{j}(\sum_{a=0}^{b-1} a^j)\theta_{m-j} + \theta_m\,,
\end{equation}
which we prefer to write in the following form (which still has $\theta_m$ on both sides):
\begin{equation}
\label{eq:thetam}
  b^{m+1}\theta_m =
  \sum_{j=0}^m\binom{m}{j}(\sum_{a=0}^{b-1} a^j)\theta_{m-j}\,.
\end{equation}
One checks easily that the sequence $(1/(m+1))_{m\geq0}$ is a solution.  This
is the sequence of the moments of $\Leb$ and Equation~\eqref{eq:thetam} is a
consequence of the auto-similarity of $\Leb$ under rescaling by a factor $b$.  So
the shifted sequences $u_{k;m}'= u_{k;m} - b \theta_m$, $k\geq0$, also satisfy
the system~\eqref{eq:umreck}, with the additional property
$u_{k;0}'=0$ for every $k\geq1$ and
$u_{k;1}'=-\frac{b(2d+1)}{2(b^2-b+1)^{k+1}}$ for every $k\geq0$.  Let
\[
  D_2=b^2-b+1.
\]
It is natural to consider the equations verified by the quantities
$D_2^{k+1}u_{k;m}'$ and to then search for a stationary solution
which would be candidate for the large-$k$ behavior.

Looking back at Equation~\eqref{eq:umrec0}, we see that the computation of
$u_{0;m}$ involves a division by
\[
  D_{m+1}=b^{m+1}-b+1.
\]
And for each $k\geq1$, obtaining $u_{k;m}$ will require one additional
division by $D_{m+1}$.  So let us examine what happens when we search in
general for a solution of the recurrences~\eqref{eq:umreck} verifying the
Ansatz $u_{k;m}=D_{n}^{-k-1}\xi_m^{(n)}$ for some $n\geq1$.
Substitution gives
\begin{equation}
\begin{split}
  D_{m+1}D_n^{-k-1}\xi_m^{(n)}
  ={}&
  \sum_{j=1}^m
  \binom mj
  \left(\sum_{a\ne d}a^j\right)
  D_n^{-k-1}\xi_{m-j}^{(n)}
\\
  &+
  \sum_{j=0}^m
  \binom mj d^j
  D_n^{-k}\xi_{m-j}^{(n)}.
\end{split}
\end{equation}
Multiplication by $D_n^{k+1}$, and then addition of $(b-1)\xi_m^{(n)}$ to
both sides yields
\begin{equation}\label{eq:ximn}
  b^{m+1}\xi_m^{(n)}
  =
  \sum_{j=0}^m \binom{m}{j}\Bigl(\sum_{a\ne d}a^j + D_n d^j\Bigr)\xi_{m-j}^{(n)}\,,
\end{equation}
It looks similar to the recurrence which would apply to the moments $\int_0^1
t^m\dnu(t)$ of a measure $\nu$ which reproduces itself up to a multiple on
each  interval $\Ifo{ab^{-1}}{(a+1)b^{-1}}$ but with an extra factor
$D_n=b^{n}-b+1$ for $a=d$.  This is how the measure $\nu_n$ from
Proposition~\ref{prop:farhis} behaves.  So, let us define
\begin{equation}
  \label{eq:Mnum}
  M_{n;m} = \int_{\Ifo01} t^m \dnu_n(t) = \EE_n(U^m),
\end{equation}
where $U(t)= t = b^{-1}d_1 + b^{-1}Z(t)$, with $d_1$ the first digit,
and a tail $Z$ which has the same law as $U$.  We obtain
\begin{equation}
  M_{n;m} = \sum_{j=0}^m\binom {m}{j} b^{-j}\EE_n(d_1^j)b^{-(m-j)}M_{n;m-j}\,,
\end{equation}
and, as $\EE_n(d_1^j) = \sum_{a\neq d} b^{-n}a^j + b^{-n}D_n d^j$, this becomes
\begin{equation}\label{eq:Mnmrec}
  b^{m+n}M_{n;m} = \sum_{j=0}^m\binom {m}{j} \Bigl(\sum_{a\neq d} a^j+ D_n d^j\Bigr) M_{n;m-j}.
\end{equation}
This is not identical, if $n>1$, with Equation~\eqref{eq:ximn}, which says:
\begin{equation}
  b^{m+1}\xi_m^{(n)}
  =
  \sum_{j=0}^m \binom{m}{j}\Bigl(\sum_{a\ne d}a^j + D_n d^j\Bigr)\xi_{m-j}^{(n)}\,.
\end{equation}
Let us examine more closely what the above implies about
$\xi_m^{(n)}$ for $m=0$, $1$, $2$, \dots{} The $j=0$ term on the right-hand side is
$b^{n}\xi_m^{(n)}$. So, as long as $m+1<n$, we necessarily have
$\xi_m^{(n)}=0$.  The first possibly non-zero coefficient is
$\xi_{n-1}^{(n)}$, and its choice fixes uniquely the other values.  We thus
shift indices, setting, for $m\geq0$,
\begin{equation}
  \theta_{m}^{(n)} = \xi_{m+n-1}^{(n)},
\end{equation}
and reformulate Equation~\eqref{eq:ximn} in terms of them:
\begin{equation}
  b^{n+m}\theta_m^{(n)}
  =
  \sum_{j=0}^{m} \binom{n-1+m}{j}\Bigl(\sum_{a\neq d}a^j + D_n d^j\Bigr)\theta_{m-j}^{(n)}\,.
\end{equation}
This is not the same recurrence as Equation~\eqref{eq:Mnmrec}, but, using
\begin{equation}
  \binom{n-1+m}{j}=\frac{(n-1+m)!}{(n-1+m-j)!j!},
\end{equation}
it becomes:
\begin{equation}
  b^{n+m}\frac{\theta_m^{(n)}}{(n-1+m)!}
  =
  \sum_{j=0}^{m} \frac1{j!}\Bigl(\sum_{a\neq d}a^j + D_n d^j\Bigr)
                         \frac{\theta_{m-j}^{(n)}}{(n-1+m-j)!}\,.
\end{equation}
which can be compared with
\begin{equation}
  b^{m+n}\frac{M_{n;m}}{m!} = \sum_{j=0}^m\frac1{j!}
             \Bigl(\sum_{a\neq d} a^j+ D_n d^j\Bigr) \frac{M_{n;m-j}}{(m-j)!}\;,
\end{equation}
and yields for $m\geq0$, the normalization being $1=1$ at $m=0$:
\begin{equation}\label{eq:xinM1}
  \xi_{m+n-1}^{(n)} = \theta_m^{(n)} = \frac{(n-1+m)!}{(n-1)!m!} M_{n;m}.
\end{equation}
This can be written suggestively as
\begin{equation}\label{eq:xinM2}
  m\geq 0\implies \xi_{m}^{(n)} 
  = \frac1{(n-1)!}\int_{\Ifo01}(\frac{\rd^{n-1}}{\dt^{n-1}}t^m)\dnu_n(t),
\end{equation}
which says that the solution $(\xi_m^{(n)})$ to the stationary recurrence
Equation~\eqref{eq:ximn}, normalized by the condition $\xi_{n-1}^{(n)} = 1$,
is the moment sequence of $(-1)^{n-1}\frac{1}{(n-1)!}$ times the $(n-1)$st
distributional derivative of $\nu_n$.

Let, for every $k\geq0$ and $n\geq1$:
\begin{equation}
  \bu_k = (u_{k;m})_{m\geq0} \qquad \bxi_n =(\xi_m^{(n)})_{m\geq0}
\end{equation}
One can also imagine $\bu_k$ to be the formal power series $\sum_{m=0}^\infty
u_{k;m}z^m$.  For $|z|<1$ this is even an absolutely convergent series which
computes
\begin{equation}
  \int_{\Ifo01}\frac{\dmu_k(t)}{1 - z t} = \frac{-1}{z}U_k(\frac{-1}z).
\end{equation}
Or, we could also use the moment-generating functions, which are entire
functions.  We simply work with the infinite vectors $\bu_k$ and $\bxi_n$.

There are uniquely determined constants $a_n(b,d)$, $n\geq1$, such that 
\begin{equation}\label{eq:bu0}
  \bu_0 = \sum_{n=1}^\infty a_n(b,d) D_n^{-1}\bxi_n.
\end{equation}
Indeed, as $D_1=1$, $a_1(b,d) = b$ is chosen such that the identity holds for
the $m=0$ coefficient. Then let $\bu_0' = \bu_0 - a_1(b,d)\bxi_1$, and choose
$a_2(b,d)$ to be $D_2$ times the $m=1$ coefficient of $\bu_0'$.  We define
$\bu_0'' = \bu_0 - a_1(b,d)\bxi_1 - a_2(b,d) D_2^{-1} \bxi_2$ (whose $m=0$ and
$m=1$ coefficients both vanish) and iterate ad infinitum.

Equation~\eqref{eq:Mnmrec}, together with the Bernoulli polynomial expressions
for power sums, imply that each $M_{n;m}$ belongs to $\QQ(b)[d]$.  From
Equation~\eqref{eq:xinM1}, this is also the case for each $\xi_m^{(n)}$. The
moments of $\mu_0$ are themselves, according to Equation
\eqref{eq:umrec0}, elements of $\QQ(b)[d]$.  Thus the algorithm just given may
be carried out over $\QQ(b)[d]$, and the $a_n(b,d)$ may accordingly
be regarded as elements of this ring.

We next establish:
\begin{prop}\label{prop:modalmoments}
The coefficients $a_n(b,d)\in\QQ(b)[d]$, $n\geq1$, constructed above so that
Equation~\eqref{eq:bu0} holds, are such that, for every $k\geq0$,
\begin{equation}
\bu_k = \sum_{n=1}^\infty \frac{a_n(b,d)}{D_n^{k+1}} \bxi_n.
\end{equation}
The column vector $(D_n^{-1}a_n(b,d))_{n\geq1}$  is the
solution of an infinite lower-triangular system:
\begin{equation}\label{eq:ansys}
  \begin{pmatrix}
    1 & 0 & 0 & 0 & \cdots\\
    M_{1;1} & 1 & 0 & 0 & \cdots\\
    M_{1;2} & 2M_{2;1} & 1 & 0 & \cdots\\
    M_{1;3} & 3M_{2;2} & 3M_{3;1} & 1 & \cdots\\
    \vdots & \vdots & \vdots & \vdots & \ddots
  \end{pmatrix}
  \begin{pmatrix}
    D_1^{-1}a_1(b,d)\\
    \vdots\\
    D_n^{-1}a_n(b,d)\\
    \vdots
  \end{pmatrix}
= 
  \begin{pmatrix}
    \int_{\Ifo01} t^0 \dmu_0(t)\\
    \vdots\\
    \int_{\Ifo01} t^{n-1} \dmu_0(t)\\
    \vdots
  \end{pmatrix},
\end{equation}
where $M_{n,k} = \int_0^1 t^k \dnu_n(t)$.
\end{prop}
\begin{proof}
  By construction, the coefficients $a_n(b,d)$ are equivalently the unique
  solution to Equation~\eqref{eq:ansys}. Let $N\geq1$ and define:
  Let $N\geq1$ and define:
  \begin{equation}
    \forall k\geq 0,\ \forall m\geq0, \quad
    v_{k;m} = u_{k;m} - \sum_{n=1}^N a_n(b,d)\frac{\xi_m^{(n)}}{D_n^{k+1}}\;.
  \end{equation}
  Then the quantities $v_{k;m}$ satisfy the same recurrences Equation
 ~\eqref{eq:umreck} as the $u_{k;m}$.  By construction $v_{0;0} = v_{0;1} =
  \dots = v_{0;N-1} = 0$.  It then follows that $v_{1;0} = v_{1;1} = \dots
  = v_{1;N-1} = 0$, then $v_{2;0} = v_{2;1} = \dots = v_{2;N-1} = 0$, and so
  on for every $k\geq0$ by induction. This completes the proof.
\end{proof}
Following the algorithm, one obtains, after some algebra, for every $k\geq0$:
\begin{equation}
u_{k;2}
=
\frac{b}{3}
-\frac{\frac12(2d+1)^2}{D_2^{k+1}}
+
\frac{bd^2-\frac{b}{3}+\frac{(2d+1)^2}{2}}{D_3^{k+1}}\;.
\end{equation}
Let us recall for comparison the quantities $c_j(b,d)$, $j\geq1$ from
Proposition~\ref{prop:asymp}.  They are the coefficients of the first row of
the inverse of an upper-triangular matrix:
\begin{equation}\label{eq:cjbdtriang}
  \begin{split}
    \bigl(c_1(b,d),c_2(b,d),c_3(b,d),\ldots\bigr)
    =
\\
      \left[\;\begin{pmatrix}
        1
        & -\EE_1(d+U)
        & \EE_1((d+U)^2)
        & -\EE_1((d+U)^3)
        & \cdots\\
        0
        & 1
        & -2\EE_2(d+U)
        & 3\EE_2((d+U)^2)
        & \cdots\\
        0
        & 0
        & 1
        & -3\EE_3(d+U)
        & \cdots\\
        0
        & 0
        & 0
        & 1
        & \cdots\\
        \vdots
        & \vdots
        & \vdots
        & \vdots
        & \ddots
      \end{pmatrix}^{-1}\;\right]
    _{i=1,j\geq1} ,
  \end{split}
\end{equation}
where $\EE_n$ means integration against $\nu_n$ and $U$ is the function $U(t)=t$.
We will prove later the following relation:
\begin{equation}
  \forall n\geq 1\quad a_n(b,d) = (-1)^{n-1}b c_n(b,d).
\end{equation}

\section{The word space return operator on measures}

Both the asymptotic expansion for Irwin sums (Proposition~\ref{prop:asymp})
and the exact representation of the measure moments (Proposition
\ref{prop:modalmoments}) suggest to introduce an extra parameter $y$, in order
to study the combined measure $\sum_{k=0}^\infty y^k \mu_k$, with the hope
that it will have poles at $y=D_n=\lambda_n^{-1}$ for $n\geq1$ (if considered
in some suitable space, as we anticipate the distributional derivatives
$\nu_n^{(n-1)}$ will prove to be the associated eigenmode).  Let us do this
on the word space $\cW$, so we let $\wM(y) = \sum_{k=0}^\infty y^k \wmu_k$,
initially for $y$, possibly complex, of modulus less than $1$.  As
$\wmu_k(\cW)=b$, this series converges in the Banach space $E$ of
complex-valued measures (we use the Borel algebra of all subsets of $\cW$, and
$E$ is the dual of the Banach space of complex-valued functions going to zero
at infinity and equipped with the sup-norm).  Note that identifying a measure
with the weights it assigns to the words, $E$ is isometric with $\ell^1(\cW)$,
but we prefer to avoid this identification which is potentially confusing from
a functorial point of view.

Recall the counting function $k_d:\cW\to \NN_0$ and the notation $\cW_k = \{g
\in \cW, k_d(g)=k\}$.  By the definition of $\wmu_k$, there holds
\begin{equation}
  \wM(y) = \sum_{g\in\cW} y^{k_d(g)}b^{-|g|}\delta_g\,.
\end{equation}
We decompose according to the first digit:
\begin{align}
  \wM(y) &=
  \delta_\emptyword + \sum_{a\in igma_b} \sum_{g\in\cW} y^{k_d(ag)}b^{-1-|g|}\delta_{ag}
\\
&=\delta_\emptyword + b^{-1}\sum_{a\neq d} \sum_{g\in\cW} y^{k_d(g)}b^{-|g|}\delta_{ag}
+ y b^{-1}\sum_{g\in\cW} y^{k_d(g)}b^{-|g|}\delta_{dg}\,.
\end{align}
Let, for any word $s$, $P_s:\cW\to\cW$ be the prefixing operation by $s$. It
induces a push-forward action on $E$, which we denote also $P_s$, no confusion
should arise (linear combinations and scalar multiplications make sense only
for the actions on measures).  In both cases we have the composition rule
$P_rP_s=P_{rs}$.  Also, let $A = \sum_{0\leq a<b} P_a$.  With this notation:
\begin{equation}
  \wM(y) = \delta_\emptyword + b^{-1}(A-P_d)\wM(y) + b^{-1}yP_d\wM(y),
\end{equation}
or, more suggestively, and using $I$ for the identity operator on $E$:
\begin{equation}
  \bigl(I - b^{-1}(A- P_d) - b^{-1} y P_d\bigr)\wM(y) = \delta_\emptyword.
\end{equation}
As each $P_a$ is isometric, the operator norm of $b^{-1}(A-P_d)$ is
bounded by $(b-1)/b$ and we can rewrite the equation, being careful about
non-commutativity, as
\begin{equation}\label{eq:foo}
  \Bigl(I - b^{-1} y \bigl(I - b^{-1}(A- P_d)\bigr)^{-1}P_d\Bigr)\wM(y) 
= \bigl(I - b^{-1}(A- P_d))^{-1}\delta_\emptyword.
\end{equation}
Define
\begin{equation}
  \label{eq:wKd}
  \wK_d = b^{-1}\bigl(I - b^{-1}(A- P_d)\bigr)^{-1}P_d.
\end{equation}
Its operator norm $\|\wK_d\|$ is bounded (using Neumann series) by $b^{-1}(1 -
b^{-1}(b-1))^{-1}=1$.  More precisely, the Neumann series gives the representation
\begin{equation}\label{eq:wKdneumann}
  \wK_d= b^{-1}\sum_{i=0}^\infty b^{-i}\sum_{\substack{a_1,\dots, a_i\\ a_j\neq d, 1\leq j\leq i}}
         P_{a_1\dots a_i d}.
\end{equation}
The sets $a_1\dots a_i d\cW$, taken over all choices of $i$ and $(a_j)\in
(\Sigma_b\setminus\{d\})^i$ are pairwise disjoint, they actually give a
partition of $\bigcup_{k>0}\cW_k = \cW - \cW_0$. So $\wK_d$ is actually an
isometry of $E$: $\|\wK_d \nu\| =\|\nu\|$ for any $\nu\in E$.

We call $\wK_d$ the \emph{return operator} in view of the decomposition over
prefixes having a single occurrence of $d$, which is terminal.  We will keep
this terminology in other contexts, although it suggests a dynamical picture
which is not what we will move to the forefront.  We will also designate the
words $a_1\dots a_i d$ as \emph{next-return} words. We let $\cG_d$ be the
\emph{language of next-return words}, i.e.\@ the set of these words.
\begin{prop}\label{prop:KdW}
  Let the \emph{return operator} be defined on the Banach space of complex measures on the word-space $\cW$ as:
\begin{equation}
\wK_d = b^{-1}\bigl(I - b^{-1}(A- P_d)\bigr)^{-1}P_d.  
\end{equation}
There holds
\begin{equation}
    \label{eq:KdWmuk}
    \wK_d^{k}\wmu_0 = \wmu_k
  \end{equation}
  for every $k\geq1$.
\end{prop}
\begin{proof}
  Setting $y=0$, we deduce that the right-hand side of Equation~\eqref{eq:foo}
  is $\wmu_0$.  Multipliying on both sides by $(I - y\wK_d)^{-1}$ we get
  \begin{equation}
    \wM(y) = (I - y \wK_d)^{-1}\wmu_0.
  \end{equation}
  After expanding in a Neumann series, we identify, by unicity of coefficients
  of convergent power series in a Banach space, the coefficients of $y^k$, and
  this gives Equation~\eqref{eq:KdWmuk} (it would have been possible to make the
  whole argument purely at the formal level, which would have avoided some
  arguably extraneous considerations of topological vector spaces throughout).
\end{proof}

An obstruction to the hoped-for spectral picture is immediately visible: the
bounded linear operator $\wK_d:E\to E$ is not compact.  Indeed, being an
isometry, it maps the unit ball of $E$ isometrically onto the unit ball of its
range. As this range is infinite-dimensional, the image of the unit ball is
not relatively compact.  The same applies to all iterates.

Even worse, we can not even realize the first pole in this picture. Indeed,
suppose that, as $y\to1^-$, there is some limit $\nu$ in $E$ for
$(1-y)\wM(y)$.  For every word $g\in\cW$, we would have $\nu(\{g\}) =
\lim_{y\to1^-}(1-y)y^{k_d(g)}b^{-|g|} =0$, hence actually $\nu=0$.  On the
other hand, $(1-y)\wM(y)(\cW) = (1-y)\sum_{k=0}^\infty y^k\wmu_k(\cW) =b $ so
$(1-y)\wM(y)$ cannot converge in $E$ as $y\to1^-$.

\section{The return operator on the unit interval}

We anticipate that the residue at the pole should actually be realizable as
$(-b)$ times the Lebesgue measure $\Leb$ on $\Ifo01$ (or $\Iff01$)
We need to make sense of the prefix operators $P_s$ in this
context.  The push-forward from the word space to $\Ifo 01$ is done by the map
$x: g\mapsto n(g)/b^{|g|}$.  For any given word $s$, we have
\begin{equation}
  x(sg) = x(s) + b^{-|s|} x(g),
\end{equation}
which is an affine map from $\Ifo01$ (or $\Iff01$) to itself. As we shall not
return to the word space, we keep our earlier notation $P_s$, now for the
push-forward action of this affine map on Borel measures on the (closed) unit
interval. As for the affine map itself we denote it $\phi_s$.  So, we again
let $A = \sum_{a\in\Sigma_b} P_a$.  The operator $I - b^{-1}(A - P_d)$ is
invertible on the space of complex Borel measures on $\Iff01$, which we again
denote $E$ (it is equipped with the total variation norm).  We define the
return operator $K_d$ as expected from Equations~\eqref{eq:wKd} and
\eqref{eq:wKdneumann}:
\begin{equation}
  \label{eq:Kd}
   K_d=  b^{-1}\bigl(I - b^{-1}(A- P_d)\bigr)^{-1}P_d
      = \sum_{g\in \cG_d} b^{-|g|} P_{g}.
\end{equation}
and the ``renewal identity'' $\mu_{k+1} = K_d\mu_k$ holds for every $k\geq0$.

We observe that $\Leb$ is a fixed point for $K_d$, so we do have an
eigenvector with eigenvalue $\lambda_1=1$.  But $K_d$ is not a compact
operator. Indeed, for any $x\in \Ioo01$, the restriction of $K_d(\delta_x)$ to
the interval $\Ioo{b^{-1}d}{b^{-1}(d+1)}$ is $b^{-1}P_d(\delta_x) =
b^{-1}\delta_{b^{-1}(d+x)}$. So the total variation of $K_d(\delta_{x_1}) -
K_d(\delta_{x_2})$, $0<x_1<x_2<1$, is at least $2b^{-1}$ and $K_d$ is not a
compact operator.

Working on $\Ifo01$ is enough for discussing many aspects relative to the
measures $\mu_k$, and the measures with no atoms $\nu_n$ from
Proposition~\ref{prop:asymp}. But already with Equation~\eqref{eq:xinM2} we
have seen some distributional derivatives, and it is more convenient to work
on the segment $\Iff01$ in this context.  So we consider that $K_d$ acts on
the complex Borel measures on $\Iff01$.

We now prove that it does not have any eigenvalue $\lambda\neq1$.  Suppose to
the contrary that $\nu$ is an eigenvector of $K_d$ with eigenvalue
$\lambda\neq1$ and consider its moments $u_m(\nu)$, $m\geq0$.  Let $m$ be the
smallest non-negative integer such that $u_{m}(\nu)\neq0$.  For any prefix
operator $P_s$, one has for this specific $m$:
\begin{equation}
  u_m(P_s\nu) = \int_{\Iff01}\phi_s(x)^m\dnu(x) = b^{-m|s|}u_m(\nu).
\end{equation}
Thus,
$
  u_m(K_d\nu) = u_m(\nu)\sum_{g\in\cG_d} b^{-|g|} b^{-m|g|}
$.
Now,
\begin{equation}\label{eq:lmmass}
  \sum_{g\in\cG_d} b^{-|g|} b^{-m|g|}=
  \sum_{i=0}^\infty b^{-(m+1)(i+1)}(b-1)^i = \frac{b^{-m-1}}{1 - b^{-m-1}(b-1)} = 
  \lambda_{m+1}\,.
\end{equation}
So $\lambda u_m(\nu) = u_m(K_d\nu) = \lambda_{m+1}u_m(\nu)$ and
$\lambda=\lambda_{m+1}$.  As $\lambda\neq1$, $m>0$, and $\lambda\leq \lambda_2
= (b^2 - b +1)^{-1}$. Also, we note that $\lambda>0$.

For an affine map $\phi_g(x) = x(g) + b^{-|g|}x$, one has $\phi_g(1) = 1$
if and only if $x(g) = 1 - b^{-|g|}$, which happens exactly when $g$ contains
only the digit $b-1$.  So, if $d\neq b-1$, this cannot happen for $g$ a
next-return word, and consequently $K_d\nu(\{1\}) = 0$, hence, also
$\nu(\{1\})=0$, as $\lambda\neq0$.  If $d=b-1$, the sole next-return word
with $\phi_g(1)=1$ is $g=d$, and $1$ is the only antecedent, so $K_d\nu(\{1\})
= b^{-1} \nu(\{1\})$.  As $\lambda\leq \lambda_2 < b^{-1}$, we again conclude
that $\nu(\{1\})=0$.

The total variation of $K_d\nu$ is at least the sum of its total variations on
the cylinders $\Cyl_g$, $g\in \cG_d$ (as they are pairwise disjoint).  The
restriction of $K_d\nu$ to $\Cyl_g$ is that of $b^{-|g|}P_g(\nu)$: this uses
that no mass at the left end-point $x(g)$ can come from the image under some
other affine map $\phi_{g'}$, $g'\neq g$, of $\Iff01$, $g'\in \cG_d$, as
$\nu(\{1\})=0$. So its total variation there is
$b^{-|g|}|\nu|(\Ifo01)$. Summing over $g\in \cG_d$, we get $\|K_d\nu\|\geq
|\nu|(\Ifo01)$.

As $\nu(\{1\})=0$, one has $\|\nu\| = |\nu|(\Ifo01)$. So $\|K_d\nu\|\geq
\|\nu\|$ which is a contradiction with $\lambda=\lambda_{m+1}$, $m>0$.

\section{Eigenpolynomials and eigendistributions}

So, we need to enlarge the setting, and, from Equation~\eqref{eq:xinM2}, the
space of distributions on the segment $\Iff01$ is a candidate framework.
Distributions are dual to smooth functions. But it will turn out that the good
setting for progress is in analytic functions.  Having spent quite some time
with measures, we start with continuous functions, though.  We first transpose
each prefix operator $P_s$ into an operator $\sP_s$ acting on (continuous)
functions:
\begin{equation}
  \sP_s(f) = f\circ \phi_s\,.
\end{equation}
We then have for any complex measure $\nu$, and continuous function $f$,
$P_s(\nu)(f) = \nu(f\circ \phi_s) = \nu(\sP_s(f))$.  We only have to pay
attention that transposition reverses the ordering, so the definition for the
action on continuous functions on $\Iff01$ is, with $\sA = \sum_{a\in
  \Sigma_b} \sP_a$:
\begin{equation}\label{eq:sKd}
  \sK_d = b^{-1}\sP_d \bigl(\sI - b^{-1}(\sA- \sP_d)\bigr)^{-1}
    = \sum_{g\in \cG_d} b^{-|g|} \sP_{g}.
\end{equation}
We use $\sI$ as notation for the identity operator here.

Now, $\sK_d$ sends polynomial functions to polynomial functions.  More
precisely, it acts in an upper-triangular manner for the monomial basis of
$\CC[X]$.  And the diagonal element at position $(m,m)$, $m\geq0$ is
$\sum_{g\in\cG_d} b^{-|g|} b^{-m|g|}$ which, from Equation~\eqref{eq:lmmass},
is $\lambda_{m+1}$.  As the $\lambda_n$, $n\geq1$, are distinct, the action of
$\sK_d$ is diagonalizable on every $\CC[X]_{m}$, and there is for each
$m\geq0$ a unique monic eigenpolynomial $\cB_m$ of degree $m$: $\sK_d(\cB_m) =
\lambda_{m+1} \cB_m$. Of course $\cB_m$ actually depends on $(b,d)$ also.

We have probability measures $\nu_n$, $n\geq1$, at our disposal.  But we saw
that for $n>1$, it cannot be an eigenvector of the return operator (on
measures) $K_d$.  Let us check more closely to see what goes wrong.  Recall
that $\nu_n$ is the law of $Y = \sum_{j=1}^\infty d_j/b^j$ where the $d_j$ are
independent random digits which are identically distributed with $\PP(d_i =a) =
b^{-n}$ if $a\neq d$ and $D_n b^{-n}$ if $d_i=d$.  So $Y = b^{-1}(d_1 + Z)$
where $Z\sim Y$ and is independent of the first digit $d_1$.  We realized
these random variables on the interval $\Ifo01$, and thus always $0\leq Y <1$.
We now use conditional probabilities to evaluate, for any Borel set
$F\subset\Iff01$, the value of $\nu_n(F)$:
\begin{align}
  \nu_n(F) &= \PP(Y \in F) = \sum_{a\in\Sigma_b}\PP(d_1=a)\PP(b^{-1}(a+Z) \in F)
\\
&=\sum_{a\in\Sigma_b}\PP(d_1=a)\PP(\phi_a(Z) \in F)
\\
&=\sum_{a\in\Sigma_b}\PP(d_1=a)P_a(\nu_n)(F)
\\
&=b^{-n}(A - P_d)(\nu_n)(F) + D_n b^{-n} P_d(\nu_n)(F).
\end{align}
The conclusion is
\begin{equation}
  (I - b^{-n}(A - P_d))\nu_n = D_n b^{-n} P_d\nu_n\,,
\end{equation}
which we rewrite as
\begin{equation}\label{eq:nunsimilarity}
  b^{-n} P_d \nu_n = \lambda_n (I - b^{-n}(A - P_d))\nu_n\,.
\end{equation}
But the eigenvector equation for eigenvalue $\lambda_n$ is, from the
definition of the return operator $K_d$:
\begin{equation}
  b^{-1}P_d \sigma = \lambda_n (I - b^{-1}(A-P_d))\sigma.
\end{equation}
If only we had an operator $M$ which did $MP_d = b P_dM$, so also $MA=bAM$,
applying it $(n-1)$ times on Equation~\eqref{eq:nunsimilarity}, we would
conclude that $M^{n-1}\nu_n$ verifies the eigenvector equation (we still will
have to check that it is not zero) for the natural extension of $K_d$ as an
operator on distributions.  We do have such an operator $M$ on distributions:
it is the derivative. To check this last statement, we transpose to functions
and have to verify that $-f'\circ \phi_d = -b (f\circ \phi_d)'$ for any smooth
function $f$.  This is correct.

We have thus proven that $\nu_n^{(n-1)}$ is an eigendistribution for $K_d$
with eigenvalue $\lambda_n$.  In view of Equation~\eqref{eq:xinM2}, setting
\begin{equation}
  \label{eq:sigman}
  \sigma_n = \frac{(-1)^{n-1}}{(n-1)!}\nu_n^{(n-1)}\,,
\end{equation}
gives indeed a non-zero distribution.

Recall from Equation~\eqref{eq:xinM2} that $\sigma_n(t^{n-1})=1$ and,
naturally, $\sigma_n(t^j)=0$ for $j<n-1$. So $\sigma_n(\cB_{n-1})=1$.  And, on
the other hand, $\sigma_n(\cB_m)$ is certainly zero whenever $n\neq m+1$ due to
the eigenvector properties.

So $(\sigma_n)_{n\geq1}$ is the orthogonal (in the sense of duality) system to
$(\cB_{n-1})_{n\geq1}$.  As distributions with their compact support on
$\Iff01$ are completely determined by their moments, we conclude that we have
found all eigenspaces of $K_d$ acting on the space of distributions.  This
proves in a new way that $1$ is the sole eigenvalue of $K_d$ on measures: if
$\sigma_n$, $n>1$, were a measure, integrating $n-1$ times $\sigma_n$ in the
open interval $\Ioo01$ would give a bounded function, and $\nu_n$ would have a
density with respect to the Lebesgue measure.  But $\nu_n$ and $\Leb$ are
mutually singular due to distinct almost sure frequencies of digits.

\section{Concentration phenomenon and Riesz bases}

Under $\nu_n$, the radix digits $d_j$ are independent, with probability
$b^{-n}$ for each value different from $d$, whereas the probability of the
digit $d$, which is $1-(b-1)b^{-n}$, becomes exponentially larger in comparison when
$n$ increases. Thus, except for $n=1$, the most probable digits in a real number from the unit interval under the law $\nu_n$ are precisely
those of
\begin{equation}
x_d=\frac{d}{b-1}=0.ddd\ldots{}_b.
\end{equation}
This suggests to compare $\nu_n$ for large $n$ with the Dirac point mass
$\delta_{x_d}$, and to quantify in some way the difference to show it is
exponentially small.  For this, we want to bound the effect of $\nu_n$ as a
linear form on the monomial basis centered at $x_d$.  What we shall actually
examine are the off-diagonal elements of the lower-triangular matrix giving
the pairing between the eigendistributions $\sigma_n$ and the monic monomials
centered at $x_d$. By their definition, and Equation~\eqref{eq:xinM2},
\begin{equation}
\sigma_n\bigl((x-x_d)^{n+q-1}\bigr)
=
\binom{n+q-1}{q}
\int_0^1(x-x_d)^q\,d\nu_n(x),
\qquad q\geq0,
\end{equation}
so basically we are studying the matrix of the shifted moments
$\int_{0}^1 (x-x_d)^q\dnu_n$, but with binomial weights, and with
each column displaced down to obtain a lower-triangular shape.  Observe
that the values for $q=0$ are equal to $1$, and that $q$ indices the
sub-diagonals.

With $X_n$ a random variable having as law $\nu_n$, the absolute value of the
$n$th entry on the $q$-th subdiagonal is bounded by
$\binom{n+q-1}{q}\EE(|X_n-x_d|^q)$.  It has turned out to be simpler not to
estimate these moments separately, but to gather in one quantity all those for
a given $q$. In the next Part, where the single digit $d$ is replaced by a
block $w$, we will give individual bounds for each moment and confirm the
alluded-to exponential concentration phenomenon. But for now, let
\begin{equation}
S_q=
\sum_{n\geq1}\binom{n+q-1}{q}\EE\bigl(|X_n-x_d|^q\bigr).
\end{equation}
We shall show below that $S_q<\infty$ for each $q>0$.

The maximal possible distance of a digit from $d$ is $\max(d,b-1-d)$. Expressed
via $x_d$ this is $(b-1)R_d$ where
\begin{equation}
  R_d = \max(x_d, 1 - x_d).
\end{equation}
So for any real number $x=\sum_{j\geq 1}b^{-j}d_j$ we have the bound
\begin{equation}
  |x - x_d|\leq R_d \sum_{j\geq 1, d_j\neq d}\frac{b-1}{b^j}\;.
\end{equation}
Define the independent Bernoulli variables
$\eta_{n,j}=\Un_{\{d_j\ne d\}}$, which have success probability
$\PP(\eta_{n,j}=1)=(b-1)b^{-n}$, independent of $j$, and put
\begin{equation}\label{eq:Wn}
W_n=\sum_{j\geq1}\frac{b-1}{b^j}\eta_{n,j},
\end{equation}
so that $0\leq W_n\leq 1$.
We obtain simply
\begin{equation}
|X_n-x_d| = \left|\sum_{j\geq1} b^{-j}d_j - x_d\right| \leq R_dW_n,
\end{equation}
and, consequently,
\begin{equation}
\EE\bigl(|X_n-x_d|^q\bigr)\leq R_d^q \,\EE\bigl(W_n^q\bigr).
\end{equation}
We define the generating function (not knowing yet if $S_q<\infty$):
\begin{equation}
  F(z) = \sum_{q\geq1} S_q z^q,
\end{equation}
and consider, for $0\leq z < R_d^{-1}$, the non-negative double sum, doing the
interchange of the summations and using the binomial series (this is why we
take $z<R_d^{-1}$) and monotone convergence:
\begin{equation}\label{eq:Fz}
 \sum_{q\geq1} \sum_{n\geq1} z^q \binom{n+q-1}{q}\EE\bigl(|X_n-x_d|^q\bigr)
\leq \sum_{n\geq1} \left(\EE\bigl(\frac1{(1-z R_d W_n)^{n}}\bigr)-1\right).
\end{equation}
Write $\rho=zR_d\in \Ifo01$.

We show that the series on the right-hand side is finite. Write $W_{n,j}$ for
the partial sum in Equation~\eqref{eq:Wn} up to $j$ (so $W_{n,j}\leq 1 -
b^{-j}$, and $W_{n,0}=0$).  On the event $\eta_{n,j}=1$, there holds
\begin{equation}
  W_{n,j} = W_{n,j-1} + \frac{b-1}{b^j},
\end{equation}
hence
\begin{equation}
  \frac{1 - \rho W_{n,j}}{1 - \rho W_{n,j-1}} = 
  1 - \frac{\rho(b-1)b^{-j}}{1 - \rho W_{n,j-1}}
  \geq 1 - \frac{\rho(b-1)b^{-j}}{1 - \rho + \rho b^{1-j}}
  =\frac{1-\rho+\rho b^{-j}}{1-\rho+\rho b^{1-j}}\;.
\end{equation}
On the event $\eta_{n,j}=0$, the ratio is $1$, so for all $j\geq1$:
\begin{equation}\label{eq:rj}
  \frac{1 - \rho W_{n,j}}{1 - \rho W_{n,j-1}} \geq r_j^{\eta_{n,j}}\,,
  \qquad r_j = \frac{1-\rho+\rho b^{-j}}{1-\rho+\rho b^{1-j}}\;.
\end{equation}
This gives a lower bound by an absolutely convergent random product:
\begin{equation}
  1 - \rho W_n \geq \prod_{j=1}^\infty r_j^{\eta_{n,j}}\,.
\end{equation}
Observe now that $\EE(r_j^{-n \eta_{n,j}})= (b-1)b^{-n}r_j^{-n} + 1 - (b-1)b^{-n}$.
So, for $0\leq \rho < 1$ and $r_j$ as defined by Equation~\eqref{eq:rj} for $j\geq1$,
independence gives:
\begin{equation}\label{eq:Wnrj}
  \EE\bigl(\frac1{(1-\rho W_n)^{n}}\bigr)
\leq \prod_{j\geq 1}\Bigl(1 + (b-1)b^{-n}(r_j^{-n}-1)\Bigr).
\end{equation}
Some algebra leads to a simple result when summing over $n$ the
$j$-th term arising in the infinite products:
\begin{equation}
  \sum_{n=1}^\infty (b-1)b^{-n}(r_j^{-n}-1) = 
  \frac{b-1}{br_j-1}-\frac{b-1}{b-1} = \frac{\rho b^{1-j}}{1 - \rho}\;.
\end{equation}
Thus,
\begin{equation}
  \sum_{j=1}^\infty\sum_{n=1}^\infty (b-1)b^{-n}(r_j^{-n}-1) 
= \frac{b\rho}{(b-1)(1-\rho)}<\infty.
\end{equation}
This implies that $\sum_{j=1}^\infty(b-1)b^{-n}(r_j^{-n}-1) =
o_{n\to\infty}(1)$, hence, from Equation~\eqref{eq:Wnrj}
$\EE\bigl(\frac1{(1-\rho W_n)^{n}}\bigr) = 1 +
O(\sum_{j=1}^\infty(b-1)b^{-n}(r_j^{-n}-1))$, and finally from Equation
\eqref{eq:Fz}:
\begin{equation}\label{eq:finite}
  0\leq zR_d<1 \implies F(z)<\infty.
\end{equation}
This radius of convergence $R_d^{-1}$ for $F(z)$ is sharp. The $n=1$ term
alone gives, since $\nu_1$ is Lebesgue measure,
\begin{equation}
S_q
\geq
\int_0^1|x-x_d|^q\,dx
=
\frac{x_d^{q+1}+(1-x_d)^{q+1}}{q+1}.
\end{equation}
Consequently $ \limsup_{q\to\infty}S_q^{1/q}=R_d$.

Let, for $R>R_d$,  $\cH_R$ be the Hardy space $H^2(D(x_d,R))$.  Then,
\begin{equation}
e_m(z)=\frac{(z-x_d)^m}{R^m},
\end{equation}
are, for $m\geq0$, the vectors of its standard orthonormal basis.  On this
basis the functional $R^{n-1}\sigma_n$ has value $1$ at $e_{n-1}$, vanishes on
$e_0,\ldots,e_{n-2}$, and its coefficient at $e_{n+q-1}$ is
\begin{equation}
\binom{n+q-1}{q}\frac1{R^q} \int_0^1(x-x_d)^q\,d\nu_n(x).
\end{equation}
From Equation~\eqref{eq:finite} the sum of the absolute values of all its
off-diagonal coefficients is finite.
\begin{prop}\label{prop:riesz}
  In the Hardy space $\cH_R=H^2(D(x_d,R))$ for any $R>R_d=\max(x_d,1-x_d)$,
  with $x_d = d/(b-1)$, the normalized monic eigenpolynomials $\cB_m/R^m$,
  $m\geq0$, of the return operator $\sK_d$ are a Riesz basis, with dual family
  the $R^{n-1}\sigma_n$, $n\geq1$.
\end{prop}
\begin{proof}
  The summability of the off-diagonal coefficients, in view of Proposition
  \ref{prop:rieszhilbert}, implies that the system
  $(R^{n-1}\sigma_n)_{n\geq1}$  is a Riesz basis in the dual of
  $\cH^2_R$.  The normaized eigenpolynomials $(\cB_m/R^m)_{m\geq0}$ are the vectors of the
  dual system.
\end{proof}
Here, we have tacitly extended the scope of the return operator $\sK_d$,
originally defined by Equation~\eqref{eq:sKd} on continuous functions on the
segment $\Iff01$ (which is a compact subset of the open disk $D(x_d,R)$), to
the full Hardy space. It can indeed be shown a priori from the formula
defining it (note that one can find a single compact subset of $D(x_d,R)$
containing the images of the disk under all the affine maps involved)
that it is bounded and compact. Compactness can
also be seen as a corollary to the above Proposition.

As we indicated earlier, regarding the eigendistributions $\sigma_n$, the
monic polynomials $\cB_m$ turn out to be relatively exponentially close
to the monomials $(x-x_d)^m$ for the Hardy space $\cH_R$ norm.  This will be
proven later, in the more general context where the single digit $d$ is replaced with
a word $w$.

\section{Proof of the main Theorem for \texorpdfstring{$k\geq1$}{k>=1}}

We now prove Theorem~\ref{thm:part1main} for $k\geq1$. Recall first Equation
\eqref{eq:IU} which says, for $k\geq1$, $I(b,d,k) = \sum_{\substack{1\leq
    a<b\\a\neq d}}U_k(a) + \mathbf 1_{\{d>0\}}U_{k-1}(d)$.  It uses the
Stieltjes functions $U_k$, defined as $
U_k(n)=\int_{\Ifo01}\frac{\dmu_k(x)}{n+x} $.  We are going to expand the
kernels $1/(n+x)$ in terms of the eigenpolynomials $\cB_m$.  Let
$R>R_d=\max(x_d,1-x_d)$, where $x_d=d/(b-1)$, and suppose that $n$ is large
enough for $n+x_d>R$ to be true.  Then $x\mapsto(n+x)^{-1}$ belongs to
$\cH_R=H^2(D(x_d,R))$. Since the dual functionals to the normalized
eigenpolynomials $\cB_{j-1}/R^{j-1}$ are the $R^{j-1}\sigma_j$,
Proposition~\ref{prop:riesz}, together with
\begin{equation}\label{eq:sigmaStieltjesProof}
  \sigma_j\Bigl(\frac1{n+x}\Bigr)=(-1)^{j-1} \int_0^1\frac{\dnu_j(x)}{(n+x)^j},
\end{equation}
implies
\begin{equation}\label{eq:StieltjesRieszProof}
  \frac1{n+x}
  =
  \sum_{j=1}^{\infty}
  (-1)^{j-1}R^{j-1}\int_0^1\frac{\dnu_j(x)}{(n+x)^j}\frac{\cB_{j-1}(x)}{R^{j-1}},
\end{equation}
with convergence in $\cH_R$. And the sequence
$((-1)^{j-1}R^{j-1}\int_0^1\frac{\dnu_j(x)}{(n+x)^j})$ belongs to
$\ell^2(\NN)$. It is actually bounded termwise in absolute value
by $(R^{j-1}n^{-j})$ so is  even summable already if $n>R$.  

Recall the modal representation of the moments of the measure $\mu_k$: from
Proposition~\ref{prop:modalmoments}, together with Equation~\eqref{eq:xinM2},
and with Equation~\eqref{eq:sigman} defining $\sigma_n$,
\begin{equation}
m\geq0\implies  \mu_k(t^m) = \sum_{n=1}^{m+1}\lambda_n^{k+1}a_n(b,d)\sigma_n(t^m).
\end{equation}
The biorthogonality relation $\sigma_n(\cB_{j-1})=\delta_{nj}$ gives then
directly for every $j\geq1$ and $k\geq0$:
\begin{equation}\label{eq:Bmomentaj}
  \mu_k(\frac{\cB_{j-1}}{R^{j-1}}) = \frac{\lambda_j^{k+1} a_j(b,d)}{R^{j-1}}.
\end{equation}
Integrating Equation~\eqref{eq:StieltjesRieszProof} against $\mu_k$ leads to:
\begin{equation}\label{eq:UkmodalRiesz}
  U_k(n)
  =
  \sum_{j=1}^{\infty}
  (-1)^{j-1}\frac{a_j(b,d)\lambda_j^{k+1}}{R^{j-1}}R^{j-1}\int_0^1\frac{\dnu_j(x)}{(n+x)^j}\;.
\end{equation}
This scalar series is absolutely convergent. Indeed, integration against
$\mu_k$ is a bounded linear functional on $\cH_R$, so the quantities on the
right-hand side in Equation~\eqref{eq:Bmomentaj} form a sequence in
$\ell^2(\NN)$, and the right-hand side above is an $\ell^2$ scalar product
between two square-summable sequences.

Suppose first that $d>0$. Then $x_d>0$ and we can choose $R$ such that $R<|x_d -
(-1)|=1+x_d$ and $R>R_d=\max(x_d,1-x_d)$. Then the modal expansion of $U_k(n)$
applies to every $n\geq1$.  Thus, Equation~\eqref{eq:UkmodalRiesz} applies
simultaneously to every Stieltjes function occurring in the formula for
$I(b,d,k)$. Let us denote for shortening notation:
\begin{equation}
  A_j(n) = \int_0^1\frac{\dnu_j(x)}{(n+x)^j}\;.
\end{equation}
We obtain (keeping in mind that $A_j(n)$ depends on course also on $(b,d)$):
\begin{equation}\label{eq:IpositiveModal}
  I(b,d,k)
  =
  \sum_{j=1}^{\infty}
  (-1)^{j-1}a_j(b,d)\lambda_j^{k}
  \Bigl(
    A_j(d)
    +
    \lambda_j\sum_{\substack{1\leq a<b\\a\neq d}}A_j(a)
  \Bigr).
\end{equation}
Recall from Proposition~\ref{prop:farhis} the quantities:
\begin{equation}
  F_j(b,d) = \int_{b^{-1}}^1 \frac{\dnu_j(x)}{x^j}\;.
\end{equation}
Decomposing along the subintervals of length $b^{-1}$ and using the
auto-similarity rules for $\nu_j$ leads to
\begin{equation}
  F_j(b,d)
  =
  \lambda_j^{-1}A_j(d)
  +
  \sum_{\substack{1\leq a<b\\a\neq d}}A_j(a).
\end{equation}
Hence, we get the absolutely convergent scalar series
\begin{equation}\label{eq:Iaj}
  I(b,d,k)
  =
  \sum_{j=1}^{\infty}
  (-1)^{j-1}a_j(b,d)F_j(b,d)\lambda_j^{k+1},
\end{equation}
which completes the proof of Theorem~\ref{thm:part1main} for $k\geq 1$ and $d>0$.

Consider now $d=0$. So $x_d=0$ and $R_d=1$.  The starting formula (originally from \cite{burnolirwin}) now reads:
\begin{equation}
  I(b,0,k)=\sum_{a=1}^{b-1}U_k(a).
\end{equation}
Choosing $1<R<2$, we can not expand immediately the kernel $1/(1+x)$ as it does not belong to the Hardy space $\cH_R$. But as $k\geq1$,
Equation~\eqref{eq:Ukraise}, with $n=1$ and $d=0$, gives
\begin{equation}
  U_k(1)
  =
  U_{k-1}(b)
  +
  \sum_{a=1}^{b-1}U_k(b+a).
\end{equation}
Consequently,
\begin{equation}
  I(b,0,k)
  =
  U_{k-1}(b)
  +
  \sum_{a=2}^{b-1}U_k(a)
  +
  \sum_{a=1}^{b-1}U_k(b+a).
\end{equation}
The Stieltjes functions $U_{k-1}$ and $U_k$ are now used with integer arguments greater than one. We may therefore apply Equation~\eqref{eq:UkmodalRiesz} to every term. This gives
\begin{equation}\label{eq:IzeroModal}
  I(b,0,k)
  =
  \sum_{j=1}^{\infty}
  (-1)^{j-1}a_j(b,0)\lambda_j^{k}
  \Bigl(
    A_j(b)
    +\lambda_j\sum_{a=2}^{b-1}A_j(a)
    +\lambda_j\sum_{a=1}^{b-1}A_j(b+a)
  \Bigr).
\end{equation}
For $d=0$, the first-digit auto-similarity of $\nu_j$ gives
\begin{equation}
  A_j(1)
  =
  \lambda_j^{-1}A_j(b)
  +
  \sum_{a=1}^{b-1}A_j(b+a).
\end{equation}
Hence the quantity in parentheses in Equation~\eqref{eq:IzeroModal} is
\begin{equation}\label{eq:FjPhiZero}
  \sum_{a=1}^{b-1}A_j(a)
  =
  F_j(b,0).
\end{equation}
Thus, Equation~\eqref{eq:Iaj} holds also for $d=0$ (with absolute convergence).

By uniqueness of the Poincaré asymptotic expansion obtained in
Proposition~\ref{prop:asymp}, comparison with Equation~\eqref{eq:asymp}
gives
\begin{equation}\label{eq:cjbdajbd}
  c_j(b,d)
  =
  \frac{(-1)^{j-1}}{b}a_j(b,d),
  \qquad j\geq1,
\end{equation}
which is the relation announced after
Proposition~\ref{prop:modalmoments}.

Finally, substituting the above into
Equation~\eqref{eq:Iaj} yields
\begin{equation}
  b^{-1}I(b,d,k)
  =
  \sum_{j=1}^{\infty}
  c_j(b,d)F_j(b,d)\lambda_j^{k+1},
  \qquad k\geq1,
\end{equation}
with absolute convergence. This completes the proof of Theorem
\ref{thm:part1main} for $k\geq1$.

\section{Proof of the main Theorem for \texorpdfstring{$k=0$}{k=0}}

It remains to prove the two statements of Theorem~\ref{thm:part1main} concerning
$k=0$. Suppose first that $d>0$.  As in the preceding section, choose $R$ such that
$R_d<R<1+x_d$.
Every Stieltjes kernel $x\mapsto(n+x)^{-1}$, $n\geq1$, belongs then to
$\cH_R=H^2(D(x_d,R))$.  Taking $k=0$ in
Equation~\eqref{eq:UkmodalRiesz}, and using
$
  c_j(b,d)=(-1)^{j-1}b^{-1}a_j(b,d)
$,
gives, with the notation $A_j(n)$ of the preceding section,
\begin{equation}\label{eq:U0modal}
  U_0(n)
  =
  b\sum_{j=1}^{\infty}c_j(b,d)\lambda_jA_j(n),
  \qquad n\geq1,
\end{equation}
with absolute convergence.

At $k=0$, Equation~\eqref{eq:IU} reads
$
  I(b,d,0)
  =
  \sum_{\substack{1\leq a<b\\a\neq d}}U_0(a)
$.
Substitution of Equation~\eqref{eq:U0modal} and regrouping therefore give
\begin{equation}
  b^{-1}I(b,d,0)
  =
  \sum_{j=1}^{\infty}
  c_j(b,d)\lambda_j
  \sum_{\substack{1\leq a<b\\a\neq d}}A_j(a),
\end{equation}
again with absolute convergence.
For $a\neq d$, the first-digit auto-similarity of $\nu_j$ gives
\begin{equation}
  \int_{[a/b,(a+1)/b)}
  \frac{d\nu_j(x)}{x^j}
  =
  A_j(a).
\end{equation}
Consequently, from the definition in Equation~\eqref{eq:Fs*},
\begin{equation}
  F_j^*(b,d)
  =
  \sum_{\substack{1\leq a<b\\a\neq d}}A_j(a).
\end{equation}
We have thus proved $ b^{-1}I(b,d,0) = \sum_{j=1}^{\infty}
c_j(b,d)F_j^*(b,d)\lambda_j $, with absolute convergence, which is
Equation~\eqref{eq:exact*} from Theorem~\ref{thm:part1main}.

We turn to $d=0$.  Here $x_d=0$, and we choose arbitrarily $R$ in the interval $\Ioo12$.
For every $n\geq2$, the kernel $(n+x)^{-1}$ belongs to
$\cH_R=H^2(D(0,R))$, and Equation~\eqref{eq:U0modal} gives
\begin{equation}
  U_0(n)
  =
  b\sum_{j=1}^{\infty}c_j(b,0)\lambda_jA_j(n),
  \qquad n\geq2,
\end{equation}
with absolute convergence.  The only difficulty is therefore with $U_0(1)$.

Equation~\eqref{eq:U0raise}, with $d=0$ and $n=1$, gives
\begin{equation}\label{eq:U01raise}
  U_0(1)
  =
  1+\sum_{a=1}^{b-1}U_0(b+a).
\end{equation}
All the Stieltjes functions on the right-hand side are covered by the preceding
modal expansion.  It remains to express the constant $1$ in the same modal
form.

On the Banach space of complex Borel measures on $\Iff01$, the prefix
operators are isometries, and the Neumann series for $(I - b^{-1}\sum_{1\leq a
  <b}P_a)^{-1}$ gives
\[
  \Bigl(I-\frac1b\sum_{a=1}^{b-1}P_a\Bigr)^{-1}\delta_0 = \mu_0\,.
\]
Put $D_0 = I - b^{-1}\sum_{1\leq a <b}P_a$ on measures, and $\sD_0 = \sI -
b^{-1}\sum_{1\leq a <b}\sP_a$ on continuous functions on $\Iff01$.  By
transposition, for every continuous function $f$ on $\Iff01$,
\begin{equation}\label{eq:mu0D0}
  \int_{\Iff01} f(x)\,d\mu_0(x)
  =
  (\sD_0^{-1}f)(0).
\end{equation}
By Equation~\eqref{eq:sKd}, where all operators are acting on
$\cC_\CC(\Iff01,\|\cdot\|_\infty)$,
$
  \sK_0=b^{-1}\sP_0\sD_0^{-1}
$.
Since $(\sP_0h)(0)=h(0)$ for every continuous $h$, and
$\sK_0\cB_{j-1}=\lambda_j\cB_{j-1}$, Equation~\eqref{eq:mu0D0} gives
\begin{equation}
  \lambda_j\cB_{j-1}(0)
  =
  (\sK_0\cB_{j-1})(0)
  =
  \frac1b(\sD_0^{-1}\cB_{j-1})(0)
  =
  \frac1b\int_{\Iff01}\cB_{j-1}(x)\,\dmu_0(x).
\end{equation}

Taking $k=0$ in Equation~\eqref{eq:Bmomentaj} shows that the integral on the
right is $\lambda_j a_j(b,0)$.  We therefore obtain
\begin{equation}\label{eq:Bzero}
  \cB_{j-1}(0)
  =
  \frac{a_j(b,0)}{b}
  =
  (-1)^{j-1}c_j(b,0).
\end{equation}
We now apply the Riesz basis expansion to the kernel $(b+x)^{-1}$:
\begin{equation}
  \frac1{b+x}
  =
  \sum_{j=1}^{\infty}
  (-1)^{j-1}A_j(b)\cB_{j-1}(x)
  \qquad\text{in }\cH_R.
\end{equation}
Evaluation at $x=0$ (which of course is a continuous linear form on the Hardy space) and Equation~\eqref{eq:Bzero} give
\begin{equation}\label{eq:oneModal}
  1
  =
  b\sum_{j=1}^{\infty}c_j(b,0)A_j(b).
\end{equation}
This series is absolutely convergent, by the same argument of $\ell^2$-pairing
as in the previous section.
It remains to recombine these contributions.  The first-digit auto-similarity
of $\nu_j$, now with $d=0$, gives
\begin{equation}\label{eq:Ajone}
  A_j(1)
  =
  \lambda_j^{-1}A_j(b)
  +
  \sum_{a=1}^{b-1}A_j(b+a).
\end{equation}
Using Equation~\eqref{eq:oneModal} and the modal expansions of the
$U_0(b+a)$ in Equation~\eqref{eq:U01raise}, we regroup
terms mode by mode.  Equation~\eqref{eq:Ajone} then gives
\begin{equation}\label{eq:U01modal}
  U_0(1)
  =
  b\sum_{j=1}^{\infty}
  c_j(b,0)\lambda_jA_j(1).
\end{equation}
Thus $U_0(1)$ has exactly the same modal formula as the $U_0(n)$, $n\geq2$,
although its kernel itself does not belong to $\cH_R$.
Finally, Equation~\eqref{eq:IU} gives
$
  I(b,0,0)
  =
  \sum_{a=1}^{b-1}U_0(a)
$.
For $a\geq2$ we use Equation~\eqref{eq:U0modal}, and for $a=1$ we use
Equation~\eqref{eq:U01modal}.  Hence
\begin{equation}
  b^{-1}I(b,0,0)
  =
  \sum_{j=1}^{\infty}
  c_j(b,0)\lambda_j
  \sum_{a=1}^{b-1}A_j(a),
\end{equation}
with absolute convergence.  The first-digit decomposition of $\nu_j$ gives
$
  F_j(b,0)
  =
  \sum_{a=1}^{b-1}A_j(a)
$.
We conclude that
\begin{equation}
  b^{-1}I(b,0,0)
  =
  \sum_{j=1}^{\infty}
  c_j(b,0)F_j(b,0)\lambda_j,
\end{equation}
with absolute convergence.  This is Equation~\eqref{eq:exact} for $k=d=0$,
and completes the proof of Theorem~\ref{thm:part1main} in all cases.

\part{Stochastic radix expansions over
  languages and the modal expansion of block-Irwin sums}

\section{Introduction}

We now study in general $I(b,w,k)$ where $w$ is a (possibly) multi-digit
(non-empty) word $w_1\dots w_p$.  A major difference with the earlier $p=1$
framework is that we stop reasoning digit per digit: this could probably be
done using vector-valued objects, corresponding to the $p$ states of the
\emph{prefix automaton} (which we will describe later), but a core lesson of
\cite{burnolone42} is that it is always possible to reduce to a purely scalar context.
It was already known from this work that the recurrence relations for the
measures $\mu_k$ led to the introduction of some other measures
(this was also extended to $p=3$ by the author at that time, but
was not prepared for publication).  This
is best discovered using the $y$-combination
$\wM(y)=\sum_{k\geq0}y^k\wmu_k$ on the word space $\cW$, which also leads to
the discovery of the suitable return operator $\wK_w$ there.  The stochastic
picture will then be developed from the core \emph{languages} called
\emph{left}, \emph{bridge} (or \emph{renewal}) and \emph{tail}, and only later
transported to the unit interval.  The $k=0$ case is separate from $k\geq1$,
as it involves another language, the \emph{occurrence-free} one.

As is to be expected, the border structure of the word $w$ is very important,
and the Guibas--Odlyzko auto-correlation polynomial \cite{guibasodlyzko1981b}
will be present, and will even take a more powerful shape as a suitable
operator $C_w$ on measures on the word space.  It enters essentially in the
algebraic derivation of the return operator $\wK_w$ on such measures.
Moving to the unit interval, this provides the model for defining a return
operator $K_w$ on all Borel complex measures on $\Iff01$, and even on
distributions, which is where the eigenvectors are found.  Its transpose $\sK_w$ on
functions has eigenpolynomials $\cB_m$, which will again serve as an expansion
basis for the Stieltjes kernels $1/(n+x)$.  We do a finer study of the
deviation of the eigendistributions from the derivatives of the Dirac mass at
the ``most probable'' point $x_w\in[0,1]$. This deviation is shown to be
exponentially small as the modal index goes to infinity. The same holds for
the relative distance, also computed in a suitable Hardy space, of the
eigenpolynomials $\cB_m$ from the monomial basis centered at $x_w$.

\section{Measures, overlap and return operators, languages}

We work in the Banach space $E$ of finite complex measures on the word space
$\cW$. For a word $s$, let $P_s:\cW\to\cW$ be prefixing by $s$, and denote
also by $P_s:E\to E$ the induced linear isometry on measures. Thus
$P_rP_s=P_{rs}$.  The \emph{prefix algebra} is the algebra of bounded
operators $T$ of the form $T=\sum_{g\in \cW} c(g)P_g$ for some weights
$c(g)\in \CC$ which are summable.  The additive decomposition of such a $T$ over the
atomic prefixing operators $P_g$ is unique.  Any element $T$ of the
prefix algebra is uniquely determined by its action on $\delta_\emptyword$.
Multiplication corresponds to the $\ell^1$-convolution algebra generated by
the rule $P_rP_s = P_{rs}$.  For $T = \sum_{g\in \cW} c(g)P_g$, with
$c(\emptyword)=1$ and $\sum_{g\neq\emptyword}|c(g)|<1$, $T$ is invertible by the
Neumann series.

Recall that $k_w:\cW\to\NN_{0}$ is the occurrence count function. In
\cite[Lemma~5]{burnolblocks}, a simple a priori argument establishes
for every $k\geq0$
\begin{equation}
  \sum_{k_w(g)\leq k} b^{-|g|}\leq p(k+1)b^p\,.
\end{equation}
So for $|y|<1$, 
\begin{equation}
  \wM(y)
  =
  \sum_{Y\in\cW}
  y^{k_w(Y)}b^{-|Y|}\delta_Y
\end{equation}
is a complex measure on $\cW$, which is a finite positive measure for $0<y <
1$, and
\begin{equation}
  \wM(y) = 
  \sum_{k\geq0}y^k\wmu_k,
  \qquad
  \wmu_k
  =
  \sum_{k_w(Y)=k}b^{-|Y|}\delta_Y,
\qquad \wmu_k(\cW)<\infty.
\end{equation}
Deleting the first digit changes the occurrence count only if the word begins
with $w$. For such a word $X=wY$, the count occurring after deleting the first
digit is $k_w(wY)-1$. This leads to the modified counting function
\begin{equation}\label{eq:kwplus}
  k_w^+(Y)=k_w(wY)-1.
\end{equation}
Consider the words $g$ with a given value $k$ of $k_w^+(g)$, hence a given
value $k+1$ of $k_w(wg)$.  Their total mass (where each $g$ is counted with
mass $b^{-|g|}$) is thus bounded above by $b^p$ times $\wmu_{k+1}(\cW)$.  So,
from domination by $b^p(\wM(y)-1)/y$, for $0<y<1$, we can define a finite
measure for any $|y|<1$ as
\begin{equation}\label{eq:etaHat}
  \wH(y)
  =
  \sum_{Y\in\cW}
  y^{k_w^+(Y)}b^{-|Y|}\delta_Y
  =
  \sum_{k\geq0}y^k\weta_k,
  \qquad
  \weta_k
  =
  \sum_{k_w^+(Y)=k}b^{-|Y|}\delta_Y .
\end{equation}

We decompose $\wM(y)$ according to the first digit.  Following
\cite{burnolone42}, let $\tau(aY)=Y$ for $a\in\Sigma_b$ be called the pruning
action on $\cW\setminus\{\emptyword\}$. Pruning a word reduces the occurrence
count only if the word has $w$ as prefix:
\begin{equation}
  k_w(X)-k_w(\tau X)
  =
  \begin{cases}
    1,&\text{if $X$ begins with $w$},\\
    0,&\text{otherwise}.
  \end{cases}
\end{equation}

We let
\begin{equation}
A=\sum_{a\in\Sigma_b}P_a\,.
\end{equation}
Acting with $A$ creates non-empty words, and the $y$-weight of $\delta_X$ in
the image will be $y^{k_w(Y)}$ for $X=aY$, $Y=\tau(X)$:
\begin{equation}
  b^{-1}A\wM(y)
  =
  \sum_{\substack{X\in\cW\\X\neq\emptyword}}
  y^{k_w(\tau X)}b^{-|X|}\delta_X .
\end{equation}
In $\wM(y)-b^{-1}A\wM(y)$, the Dirac masses $\delta_X$ for non-empty words $X$
not having $w$ as prefix cancel out.  If, on the other hand, $X=wZ$, the
respective coefficients of $\delta_X$ in $\wM(y)$ and $b^{-1}A\wM(y)$ are
\begin{equation}
  b^{-|wZ|}y^{k_w(wZ)}
  \qquad\hbox{and}\qquad
  b^{-|wZ|}y^{k_w(wZ)-1}.
\end{equation}
Their difference is therefore, using the notation from
Equation~\eqref{eq:kwplus},
\begin{equation}
  (y-1)b^{-|wZ|}y^{k_w^+(Z)}.
\end{equation}
Thus
\begin{equation}\label{eq:prune}
  \bigl(I-b^{-1}A\bigr)\wM(y) = \delta_\emptyword - 
  (1-y)b^{-p}P_w\wH(y).
\end{equation}

We next compare $k_w^+$ with the original counting function $k_w$. 

Their
difference consists exactly of the occurrences of $w$ which cross the
boundary, whose location is represented by the $|$ symbol, in the word $w\,|\,Y$.

Let $q$, $0<q<p$, be an overlap period of the word $w$ (cf.\@
\cite{guibasodlyzko1981b}). If $c_q$ is the suffix of $w$ of length $q>0$,
and $r_q$ the prefix of $w$ of length $q$, then
\begin{equation}
  wc_q=r_qw,
\end{equation}
as both represent the ``gluing'' of $w$ with itself along the overlap.
Let $\cC_w^+$ be the set of such ``complement'' words $c_q$. The set is empty
if and only if $w$ does not have borders.

Several members of $\cC_w^+$ may be prefixes of the same $Y$. When they are
ordered by length, the corresponding values of $k_w^+$
decrease successively by one (compare with the proof in \cite{burnolblocks} of the
Guibas--Odlyzko formula for the generating function of the occurrence-free
language cardinalities per word length; see also \cite[\S6.4]{odlyzko1995} for the binary case). Hence, word by word,
\begin{equation}\label{eq:bordertel}
  y^{k_w(Y)}
  =
  y^{k_w^+(Y)}
  +
  (1-y)
  \sum_{\substack{c\in\cC_w^+\\Y=cZ}}
  y^{k_w^+(Z)} .
\end{equation}
Let the \emph{overlap operator} 
be defined as
\begin{equation}\label{eq:complop}
  C_w
  =
  \sum_{c\in\cC_w^+}b^{-|c|}P_c .
\end{equation}
It is thus the zero operator if and only if $w$ has no border.
Summing Equation~\eqref{eq:bordertel} over the word space yields the
relation
\begin{equation}\label{eq:MfromH}
    \wM(y)
    = \wH(y) + (1-y) C_w\wH(y).
\end{equation}

We eliminate $\wM(y)$ from the system of the two Equations
\eqref{eq:prune} and \eqref{eq:MfromH}. Defining the bounded linear operators
$\wD_w$, $\wE_w$ acting $E$ by
\begin{equation}\label{eq:wDwEw}
  \wD_w
  =
  \bigl(I-b^{-1}A\bigr)(I+C_w)+b^{-p}P_w,
  \qquad
  \wE_w
  =
  \bigl(I-b^{-1}A\bigr)C_w+b^{-p}P_w,
\end{equation}
the elimination leads to
\begin{equation}\label{eq:DH}
  (\wD_w-y\wE_w)\wH(y)
  =
  \delta_\emptyword .
\end{equation}
Extracting coefficients of $y$ gives
\begin{equation}\label{eq:etaDE}
  \wD_w\weta_0=\delta_\emptyword,
  \qquad
  \wD_w\weta_k
  =
  \wE_w\weta_{k-1}
  \quad(k\geq1).
\end{equation}

We define the \emph{tail} language
\begin{equation}
  \cR
  =
  \{r\in\cW:k_w^+(r)=0\},
\end{equation}
whose words create no new occurrence after a given occurrence of $w$.  So,
\begin{equation}
  \weta_0 = \sum_{r\in \cR}b^{-|r|}\delta_r\,.
\end{equation}

As $\wD_w$ belongs to the prefix algebra, comparison with Equation
\eqref{eq:etaDE} shows that it is invertible and that
\begin{equation}\label{eq:Dinv}
  \wD_w^{-1}
  =
  \sum_{r\in\cR}b^{-|r|}P_r .
\end{equation}
It was important for this argument that we knew already the finiteness of
$\sum_{r\in\cR}b^{-|r|}$ (as we knew already $\weta_0(\cW)<\infty$).

Define the
\emph{next-return}, or \emph{renewal}, language as
\begin{equation}
  \cG
  =
  \{g\neq\emptyword:
  k_w^+(g)=1,\;
  k_w^+(g[:-1])=0\},
\end{equation}
where $g[:-1]$, for any non-empty word $g$, denotes $g$ with its final digit deleted. 
Thus the first new occurrence is completed exactly at the final digit of
$g$. We also define the \emph{left} (or \emph{initial}) language
\begin{equation}
  \cL
  =
  \{\ell\neq\emptyword:
  k_w(\ell)=1,\;
  k_w(\ell^-)=0\},
\end{equation}
and the \emph{occurrence-free} language
\begin{equation}
  \cF
  =
  \{f\in\cW:k_w(f)=0\}.
\end{equation}
We then have unique decompositions, for every $k\geq0$,
\begin{equation}\label{eq:etafact}
  \{Y\in\cW:k_w^+(Y)=k\}
  =
  \cG^k\cR,
\end{equation}
and, for $k\geq1$,
\begin{equation}\label{eq:Wfact}
  \cW_k
  =
  \cL\cG^{k-1}\cR.
\end{equation}
The occurrence-free case has instead the decomposition
\begin{equation}\label{eq:Ffact}
  \cF
  =
  \bigsqcup_{c\in\{\emptyword\}\cup\cC_w^+}
  c\cR,
\end{equation}
and $\cW$ is partitioned into $\cF$ and the $\cL\cG^{k-1}\cR$ for $k\geq1$.

Taking $k=1$ in Equation~\eqref{eq:etafact} gives
\begin{equation}
  \weta_1
  =
  \sum_{\substack{g\in\cG\\r\in\cR}}
  b^{-|g|-|r|}\delta_{gr}.
\end{equation}
This shows in passing, knowing $\weta_1(\cW)<\infty$, that $\sum_{g\in
  \cG}b^{-|g|}$ is finite.  It also implies $\sum_{r\in \cR} b^{-|r|}<\infty$,
which we already knew.  As $\cL$ is the reversal of the $\cR$ associated with
the reversed $w$, postfixed elementwise by $w$, it too has a finite total
mass.  We will return to these key facts later.

Define the \emph{return operator} $\wK_w$ as
\begin{equation}\label{eq:wKw}
  \wK_w
  =
  \sum_{g\in\cG}b^{-|g|}P_g .
\end{equation}
So $\weta_1=\wK_w\weta_0$.

On the other hand, Equation~\eqref{eq:etaDE} for $k=1$ gives
$\wD_w\weta_1=\wE_w\weta_0$.  Since
$\weta_1=\wK_w\weta_0$ and
$\weta_0=\wD_w^{-1}\delta_\emptyword$, this reads
\begin{equation}
  \wD_w\wK_w\wD_w^{-1}\delta_\emptyword
  =
  \wE_w\wD_w^{-1}\delta_\emptyword .
\end{equation}
Hence,
$\wD_w\wK_w\wD_w^{-1}=\wE_w\wD_w^{-1}$, and therefore, in the sub-algebra of prefix operators as bounded operators on $E$:
\begin{equation}\label{eq:wKwDwEw}
  \wK_w=\wD_w^{-1}\wE_w.
\end{equation}

Using now repeatedly Equation~\eqref{eq:etaDE}, we obtain the core renewal identity:
\begin{equation}\label{eq:Kwetak}
  \weta_k
  =
  \wK_w^k\weta_0 .
\end{equation}
In the $p=1$ case, $\weta_k=\wmu_k$ for every $k\geq0$. For $p>1$, the
``correct'' measures $\weta_k$ first arose, with other notation and only for
$p=2$, $k\leq 1$, in \cite{burnolone42}.

Thus the evolution of the scalar family $(\weta_k)_{k\geq0}$ is exactly the
next-return decomposition of words \emph{when one regards every word in $\cW$
  as starting after an initial, deleted, occurrence of $w$}.

\section{Interlude: generating functions and a prefix automaton}

We consider here the ordinary generating functions attached to the languages
introduced in the previous section. They connect our constructions with work
of Guibas--Odlyzko \cite{guibasodlyzko1981b}. We shall then describe a
finite-state, digit-by-digit automaton which is attached to the same
combinatorics. It is closely related to the prefix-suffix bookkeeping
underlying the Knuth--Morris--Pratt algorithm \cite{knuth1977KMP}. We will not
need it in the sequel, for reasons given in the concluding paragraph of this section.

Put
\begin{equation}
 F(t)=\sum_{g\in\cF}t^{|g|},\quad
 L(t)=\sum_{g\in\cL}t^{|g|},\quad
 G(t)=\sum_{g\in\cG}t^{|g|},\quad
 R(t)=\sum_{g\in\cR}t^{|g|},
\end{equation}
and, for $k\geq0$,
\begin{equation}
 Z_k(t)=\sum_{k_w(X)=k}t^{|X|},
 \qquad
 N_k(t)=\sum_{k_w^+(Y)=k}t^{|Y|}.
\end{equation}
Thus $Z_0(t)=F(t)$.

From the factorizations of the preceding sections, we get
\begin{equation}
 N_k(t)=G(t)^kR(t),
 \qquad k\geq0,
\end{equation}
and
\begin{equation}
 Z_k(t)=L(t)G(t)^{k-1}R(t),
 \qquad k\geq1.
\end{equation}
Let also
\begin{equation}\label{eq:Ct}
 C(t)
 =
 1+\sum_{c\in\cC_w^+}t^{|c|}.
\end{equation}
This is precisely the Guibas--Odlyzko auto-correlation polynomial from
\cite{guibasodlyzko1981b}. Its non-zero coefficients are all equal to $1$.

There is a direct scalar specialization of the prefix algebra which recovers
the generating functions above. For every word $s$, substitute
\begin{equation}
 P_s\longmapsto (bt)^{|s|}.
\end{equation}
A term $b^{-|s|}P_s$ is thereby sent to
$t^{|s|}$. In particular,
\begin{equation}
 b^{-1}A\longmapsto bt,
 \qquad
 C_w\longmapsto C(t)-1,
 \qquad
 b^{-p}P_w\longmapsto t^p\,.
\end{equation}

In view of Equation \eqref{eq:wDwEw}, which defines the operators $\wD_w$ and
$\wE_w$, it is convenient to set
\begin{equation}\label{eq:Dt}
 D(t)=(1-bt)C(t)+t^p
\end{equation}
and
\begin{equation}\label{eq:Et}
 E(t)=(1-bt)(C(t)-1)+t^p\,.
\end{equation}
The identities obtained in the (non-commutative!) prefix algebra
then give
\begin{equation}\label{eq:RtGt}
 R(t)=\frac1{D(t)}\qquad
 G(t)=\frac{E(t)}{D(t)}\;.
\end{equation}
The occurrence-free language partition
$
 \cF
 =
 \bigsqcup_{c\in\{\emptyword\}\cup\cC_w^+}c\cR
$
gives in the same way
\begin{equation}
 F(t)=C(t)R(t)=\frac{C(t)}{D(t)}.
\end{equation}

It remains to identify $L(t)$. Let $w^{\mathrm{rev}}$ denote the reversal of
$w$. Reversal preserves the lengths of the borders, hence the overlap periods,
and therefore
\begin{equation}
 C_{w^{\mathrm{rev}}}(t)=C_w(t)=C(t).
\end{equation}
The description of $\cL$ given above says that its elements are
obtained by reversing a tail word associated with $w^{\mathrm{rev}}$, and then
adjoining the terminal occurrence $w$. Reversal does not change the length
and adjoining $w$ adds $p$ digits. Hence
\begin{equation}
 L(t)=t^pR_{w^{\mathrm{rev}}}(t) = \frac{t^p}{D_{w^{\mathrm{rev}}}(t)} = \frac{t^p}{D(t)}\;.
\end{equation}
We have thus obtained the four basic scalar generating functions
\begin{equation}
 F(t)=\frac{C(t)}{D(t)},
 \qquad
 L(t)=\frac{t^p}{D(t)},
 \qquad
 G(t)=\frac{E(t)}{D(t)},
 \qquad
 R(t)=\frac1{D(t)}.
\end{equation}
Consequently,
\begin{equation}
 N_k(t)
 =
 \frac{E(t)^k}{D(t)^{k+1}},
 \qquad k\geq0,
\end{equation}
while
\begin{equation}
 Z_0(t)=\frac{C(t)}{D(t)},
 \qquad
 Z_k(t)
 =
 \frac{t^pE(t)^{k-1}}{D(t)^{k+1}},
 \quad k\geq1.
\end{equation}
The formula for $Z_0$ is a result of Guibas--Odlyzko \cite{guibasodlyzko1981b}
(see also \cite{odlyzko1995} for the occurrence-free language in a $b=2$
context).  The formulas for $Z_k$, $k\geq1$, were already obtained in
\cite{burnolblocks}, presumably they are to be found in the earlier literature
(perhaps also the other generating functions above, but we could not locate
a reference).

Since
$
 D(b^{-1})=b^{-p}
$
and
$
 E(b^{-1})=b^{-p}
$,
one obtains
\begin{equation}
 G(b^{-1})=1,
 \qquad
 L(b^{-1})=1,
 \qquad
 R(b^{-1})=b^p,
 \qquad
 F(b^{-1})=b^pC(b^{-1}).
\end{equation}
Thus
\begin{equation}\label{eq:masses}
 \sum_{g\in\cG}b^{-|g|}=1, \quad \sum_{g\in\cL}b^{-|g|}=1, \quad \sum_{g\in\cR}b^{-|g|}=b^p,
\quad \sum_{g\in\cF}b^{-|g|}
 =
 b^pC(b^{-1}).
\end{equation}
Likewise,
\begin{equation}
 N_k(b^{-1})=b^p,
 \ (k\geq0),\quad
 Z_k(b^{-1})=b^p,
 \ (k\geq1),\quad Z_0(b^{-1})=b^pC(b^{-1}).
\end{equation}

We now describe a digit-by-digit finite-state construction which provides a
determinantal interpretation for these formulas. We shall call it the
\emph{prefix automaton} associated with $w$ (and $b$). It is tacit in the bookkeeping
underlying the Knuth--Morris--Pratt pattern-matching algorithm
\cite{knuth1977KMP}.  Its states are the integers from $0$ to $p-1$.  The
automaton is in state $i$ if, among the proper prefixes of $w$, the longest
one which is a suffix of the digits read so far has length $i$. In the special
case of $p=1$, the automaton remains permanently in its unique state.  For
$p>1$, if the first read digit is not $w_1$, the automaton remains in state
$0$ after the first digit, else it jumps to state $1$.

For $0\leq i<p$ and $a\in\Sigma_b$, let $q(i,a)$ denote the state to which the
automaton goes after reading the digit $a$ when it is in state $i$. We also
put $ c(i,a) = \delta_{(p-1,w_p)}(i,a)$.  Let $h$ be the length of the longest
border of $w$, or $0$ if $w$ has no border.
Immediately after an occurrence of $w$ has been completed, the automaton is in
state $h$.  The sole occurrence-completing transition is from state $p-1$ to
state $h$, and happens exactly when reading $a=w_p$ while in state $p-1$. In
general there may be multiple transitions from a state $i$ to a state $j$ (but
only if $j\leq i$).  If $h>0$ there is a unique manner to transition from
state $p-1$ to state $h$.

Define the $p\times p$ matrices $B_0$, $B_1$, $B_y$ by
\begin{align}
 (B_0)_{ij}
 &=
 \#\{a\in\Sigma_b:
 q(i,a)=j,\ c(i,a)=0\}\\
(B_1)_{ij}
 &=
 \#\{a\in\Sigma_b:
 q(i,a)=j,\ c(i,a)=1\}\\
B_y &= B_0 + y B_1\,.
\end{align}
Note that we are counting ways to jump from a row index state to a column
index state.  The matrix $B_1$ is $e_{p-1}e_h^T$, where $e_i$ is the $i$-th
coordinate column vector.

The exponents of the variable $t$ record the number of digits and those of $y$
the number of completed occurrences. If ${\bf 1}$ denotes the column vector
all of whose entries are $1$, summing
over all possible lengths gives
\begin{equation}
 \sum_{k\geq0}Z_k(t)y^k
 =
 e_0^T(I-tB_y)^{-1}{\bf 1}.
\end{equation}
Starting instead just after a distinguished occurrence gives
\begin{equation}
 \sum_{k\geq0}N_k(t)y^k
 =
 e_h^T(I-tB_y)^{-1}{\bf 1}.
\end{equation}
The four language generating functions themselves can also be read directly
from the finite-state system:
\begin{align}
 F(t)
 &=
 e_0^T(I-tB_0)^{-1}{\bf 1},
\\
 R(t)
 &=
 e_h^T(I-tB_0)^{-1}{\bf 1},
\\
 L(t)
 &=
 t\,e_0^T(I-tB_0)^{-1}e_{p-1},
\\
\shortintertext{and}
 G(t)
 &=
 t\,e_h^T(I-tB_0)^{-1}e_{p-1}.
\end{align}
These formulas are to be read with a Neumann series in mind, as transition
counts from the initial state (i.e.\@ on the left) to the final vector (on the
right). The first two expressions allow arbitrary continuation as long as no
new occurrence is completed. The last two stop precisely on the digit which
completes the first occurrence, with a factor of $t$ accounting for that last
digit.

The determinant $Q(t) = \det(I - tB_0)$ has a simple formula.
From
\begin{equation}
 F(t)=e_0^T(I-tB_0)^{-1}\mathbf{1}
\end{equation}
we may write $F(t)=H(t)/Q(t)$ for some polynomial $H(t)$. Since 
$F(t)=\frac{C(t)}{D(t)}$
we obtain
$
 H(t)D(t)=Q(t)C(t)
$.
Now Equation \eqref{eq:Dt} implies that the polynomials
$C$ and $D$ are relatively prime. Hence $D$ divides
$Q$. But $D$ has degree $p$, while $Q$ has degree at most $p$, and
both have constant term $1$. Therefore $Q=D$ and 
$
 \det(I-tB_0)=D(t)
$.

We can now compute more generally $\det(I - t B_y)$ for any $y$.
Recall the
general formula for rank-one perturbations (which can be seen as a corollary
to the well-known $\det(I_a+ MN) = \det(I_b + NM)$ for rectangular matrices of
sizes $a\times b$ and $b\times a$):
\begin{equation}
\det(A+uv^T)=\det(A)\bigl(1+v^TA^{-1}u\bigr)\,.
\end{equation}
Take here $A= I - t B_0$, $u =-yt\,e_{p-1}$, $v=e_h$. Then
\begin{align}
I-tB_y&=I-tB_0-yt\,e_{p-1}e_h^T\\
\det(I-tB_y)
&=
\det(I-tB_0)
\bigl(1-yt\,e_h^T(I-tB_0)^{-1}e_{p-1}\bigr)
\\
&=
 D(t)\bigl(1-yG(t)\bigr)
\\
&
 =D(t)-yE(t).
\end{align}

We conclude with some remarks. This automaton construction would make it
possible in the sequel to decompose a real number one digit at a time. The
apparent price is that every quantity becomes indexed by the current
matched-prefix state: scalar quantities are replaced by vectors with $p$
components, and scalar relations become vectorial and involve matrices. It 
turned out that the use of the measures $\weta_k$ and of the return operator
$\wK_w$ acting on scalar measures gives formulas for Irwin sums
$I(b,w,k)$ without matrices.  This does not exclude the possibility that the
prefix automaton idea could prove relevant in this context.

\section{Probability space and random variables}

Recall Equation \eqref{eq:masses} giving the respective masses of the languages $\cL$,
$\cG$, $\cR$, and $\cF$.  Regarding the latter, it has the disjoint decomposition
\begin{equation}
  \cF= \cW_0=\bigsqcup_{c\in\cC_w}c\cR,\qquad
  \cC_w=\{\emptyword\}\cup\cC_w^+\,.
\end{equation}
We equip
\begin{equation}
  \cC_w\times\cL\times\cG^{\NN}\times\cR
\end{equation}
with the product probability for which the coordinate random words
$C_0$, $L$, $G_1$, $G_2$, \dots, $R$ are independent and
\begin{equation}
  \begin{gathered}
    \PP(C_0=c)=\frac{b^{-|c|}}{C(b^{-1})},\\
    \PP(L=l)=b^{-|l|},\qquad
    \PP(G_j=g)=b^{-|g|},\qquad
    \PP(R=r)=b^{-p-|r|}\,.
  \end{gathered}
\end{equation}
Define now random variables with values in $\Ifo01$:
\begin{equation}
  X_0=x(C_0R),\qquad
  X_k=x(LG_1\cdots G_{k-1}R)\quad(k\geq1)\,.
\end{equation}
Their respective laws are proportional to the
measures $\mu_k$, $k\geq0$, from \cite{burnolblocks}.  More precisely
\begin{equation}
  \Law(X_0)
  =
  \frac{b^{-p}}{C(b^{-1})}\mu_0,\qquad
  \Law(X_k)=b^{-p}\mu_k\quad(k\geq1)\,.
\end{equation}
The Irwin sums acquire these probabilistic representations
\begin{equation}\label{eq:EEIw}
  I(b,w,k)
  =
  \begin{cases}
    b^pC(b^{-1})
    \EE\bigl(\Un_{\{X_0\geq b^{-1}\}}X_0^{-1}\bigr),
      & k=0,\\[1mm]
    b^p
    \EE\bigl(\Un_{\{X_k\geq b^{-1}\}}X_k^{-1}\bigr),
      & k\geq1\,,
  \end{cases}
\end{equation}
which provide a way to study quantitatively the large-$k$ behavior, by comparison with
the similar expectations, but using the random variable
\begin{equation}\label{eq:Xinf}
  X_\infty=x(LG_1G_2\cdots)\,.
\end{equation}
It has values in $\Iff01$ and law $\Leb$.
Indeed, an infinite sequence of independent uniform radix-$b$ digits can
almost surely be cut at its first occurrence of $w$, and then at its
successive occurrences; the successive blocks have exactly the laws assigned
above to $L,G_1,G_2,\ldots$.

For $k\geq1$, $X_k$ and $X_\infty$ have the common prefix
$LG_1\cdots G_{k-1}$. Since $|L|\geq p$ and every $G_j$ is non-empty, this
prefix has length at least $p+k-1$, and consequently
$
  |X_k-X_\infty|\leq b^{-p-k+1}
$.
Thus $X_k(\omega)\to X_\infty(\omega)$ for every $\omega$ in the underlying
probability space and dominated convergence leads to, on $\Iff01$,
$
  \mu_k\Rightarrow b^p\Leb
$.
This convergence is a result from \cite{burnolblocks}, here obtained with
hardly any new computations beyond those involved in discovering the language
masses.

For every $k\geq1$, define truncations at the stopping times defined by the
occurrence locations,
\begin{equation}
  Y_k=x(LG_1\cdots G_{k-1}),\qquad
  \tau_k=|LG_1\cdots G_{k-1}|,
\end{equation}
and the matching tail variables
\begin{equation}
  Z_k=x(G_kG_{k+1}\cdots).
\end{equation}
Then
\begin{equation}
  \begin{aligned}
    X_k &=Y_k+b^{-\tau_k}x(R)\\
    X_\infty&=Y_k+b^{-\tau_k}Z_k\,.
  \end{aligned}
\end{equation}
The random variable $Z_k$ is independent of $(Y_k,\tau_k)$ and has the same
law as $x(G_1G_2\cdots)$. Moreover, for $k\geq1$,
\begin{equation}
  \{X_k\geq b^{-1}\}
  =
  \{Y_k\geq b^{-1}\}
  =
  \{X_\infty\geq b^{-1}\}.
\end{equation}
We have on this event
\begin{equation}
  0\leq
  \frac{b^{-\tau_k}x(R)}{Y_k}< b^{1-p},\qquad
  0\leq
  \frac{b^{-\tau_k}Z_k}{Y_k}\leq b^{1-p}\,.
\end{equation}
and, assuming here $p>1$ (the $p=1$ situation was given a detailed study in
Part~1 in a slightly different context), we can expand in uniformly convergent
geometric series:
\begin{equation}
  X_k^{-1}
  =
  \sum_{j\geq0}
  (-1)^jb^{-j\tau_k}x(R)^jY_k^{-j-1},\qquad
  X_\infty^{-1}
  =
  \sum_{j\geq0}
  (-1)^jb^{-j\tau_k}Z_k^jY_k^{-j-1}\,.
\end{equation}

In the term with inverse power $s=j+1$, a characteristic $b^{-(s-1)\tau_k}$
factor arises, as in the $p=1$ earlier treatment. If
\begin{equation}
  LG_1\cdots G_{k-1}=lg_1\cdots g_{k-1},
\end{equation}
multiplying the original probability weight by this factor replaces
\begin{equation}
  b^{-|l|}\prod_{i=1}^{k-1}b^{-|g_i|}
  \quad\textrm{by}\quad
  b^{-s|l|}\prod_{i=1}^{k-1}b^{-s|g_i|}\,.
\end{equation}
So, the initial block is weighted by $b^{-s|l|}$ and each repeated
``next-return'' block by $b^{-s|g|}$. This leads, regarding the return
language $\cG$, to the numbers
\begin{equation}\label{eq:lambdasw}
  \lambda_s(w)
  =
  \sum_{g\in\cG}b^{-s|g|}\qquad(s\geq1)
\end{equation}
and to defining exponentially tilted probabilities (see, e.g.,
\cite[\S1b, Eq.~(1.2)]{asmussenGlynn2007}):
\begin{equation}\label{eq:tiltedlaw}
  \PP_s(G=g)
  =
  \frac{b^{-s|g|}}{\lambda_s(w)}
  \qquad(g\in\cG)\,.
\end{equation}
We skip the $p=1$ discovery path, with its triangular system obtained from
higher inverse powers, and continue directly with the study of the role of
Equations~\eqref{eq:lambdasw} and~\eqref{eq:tiltedlaw}.

For $s\geq1$, let $G_1,G_2,\ldots$ be independent next-return words (we drop
the $s$ from their notation) with the probability
law~\eqref{eq:tiltedlaw}, and let $\nu_s$ be the probability measure on
$[0,1]$ which is the law of
\begin{equation}\label{eq:nusw}
  x(G_1G_2G_3\cdots)\,.
\end{equation}
With the same argument as above with $X_\infty$
(but now cutting the radix expansions of real numbers on the premise that
there was an immediately preceding occurrence of $w$), one has here
$\nu_1=\Leb$.  Conditioning on the first word $g\in\cG$ gives, for every
bounded Borel function $f$,
\begin{equation}\label{eq:sintlem}
  \lambda_s(w)\int_0^1 f(x)\,d\nu_s(x)
  =
  \sum_{g\in\cG}b^{-s|g|}
  \int_0^1
  f\bigl(x(g)+b^{-|g|}x\bigr)\,d\nu_s(x)\,.
\end{equation}
Define directly on the unit interval the return operator $\sK_w$, acting (in a
first incarnation), on bounded continuous functions, and in particular on
polynomial functions, by
\begin{equation}\label{eq:Kwf}
  (\sK_wf)(x)
  =
  \sum_{g\in\cG}b^{-|g|}
  f\bigl(x(g)+b^{-|g|}x\bigr)\,.
\end{equation}
Assuming $f$ to be $m$ times continuously differentiable, termwise derivation
is allowed and gives
\begin{equation}
  (\sK_wf)^{(m)}(x)
  =
  \sum_{g\in\cG}b^{-(m+1)|g|}
  f^{(m)}\bigl(x(g)+b^{-|g|}x\bigr)\,.
\end{equation}
Applying Equation~\eqref{eq:sintlem} with $s=m+1$ yields
\begin{equation}\label{eq:eigendistw}
  \frac1{m!}\int_0^1(\sK_wf)^{(m)}(x)\,d\nu_{m+1}(x)
  =
  \lambda_{m+1}(w)
  \frac1{m!}\int_0^1f^{(m)}(x)\,d\nu_{m+1}(x)\,.
\end{equation}
Equivalently, $D^m\nu_{m+1}$ is an eigendistribution, with eigenvalue
$\lambda_{m+1}(w)$, for the transpose $K_w$ of $\sK_w$. This $K_w$ is related
to $\wK_w$ via the intertwining $ x_*\wK_w = K_w x_*$, with $x_*$ denoting
push-forward for measures under $x:\cW\to\Iff01$.

The same modal numbers occur
on the function side. Indeed,
\begin{equation}
  \sK_w(x^m)
  =
  \sum_{j=0}^m\binom mjx^j
  \sum_{g\in\cG}
  b^{-(j+1)|g|}x(g)^{m-j}\,,
\end{equation}
so
\begin{equation}
  \sK_w(x^m)
  =
  \lambda_{m+1}(w)x^m
  +\text{a polynomial of degree at most }m-1\,.
\end{equation}
Thus the return operator $\sK_w$, acting on the polynomials of degree
at most $m$, is triangular, with diagonal
$\lambda_1(w),\ldots,\lambda_{m+1}(w)$. Since these numbers are distinct,
there is a unique monic polynomial $\cB_m$ of degree $m$ satisfying
\begin{equation}\label{eq:eigenpolw}
  \sK_w(\cB_m)
  =
  \lambda_{m+1}(w)\cB_m\,.
\end{equation}
Equations~\eqref{eq:eigendistw}
and~\eqref{eq:eigenpolw}, together with the fact that $\cB_n$ is monic, give
\begin{equation}\label{eq:biorthw}
  \frac1{m!}
  \int_0^1
  \cB_n^{(m)}(x)\,d\nu_{m+1}(x)
  =
  \delta_{mn}\,.
\end{equation}

Finally, let $\eta_k$ be the push-forward to $[0,1]$ of the measures
$\widehat\eta_k$ defined earlier. The identity already proved on word measures
becomes, after transposition and for every continuous function on $\Iff01$:
\begin{equation}\label{eq:etaw}
  \int_0^1f(x)\,d\eta_{k+1}(x)
  =
  \int_0^1(\sK_wf)(x)\,d\eta_k(x)\,.
\end{equation}
Applying this to the eigenpolynomials gives, for every $m,k\geq0$,
\begin{equation}\label{eq:etakBmw}
  \int_0^1\cB_m(x)\,d\eta_k(x)
  =
  \lambda_{m+1}(w)^k
  \int_0^1\cB_m(x)\,d\eta_0(x)\,.
\end{equation}

\section{Concentration phenomenon: moment estimates}

Let $g_w$ be the shortest word of $\cG$. It is indeed unique: if $w$ is
unbordered, it is $w$ itself, else it is the complement of the longest
border. Let $\tau_w=|g_w|$ be its length.

Fix $n\geq1$. Let $G_1,G_2,\ldots$ be independent $\cG$-valued random
words with the probability law defined
by Equation~\eqref{eq:tiltedlaw} for $s=n$, and let
\begin{equation}
  X_n = x(G_1G_2G_3\dots),
\end{equation}
so that $X_n$ has law $\nu_n$.
In a sense, which we are going to make quantitative,
the most probable value for $X_n$ (in the sense of probability concentration)
is $x(g_wg_wg_wg_w\cdots)$, because $g_w$ is
the most probable word in $\cG$ (which is prefix-free) for each one of the $G_i$'s.

Thus, let $x_w=x(g_wg_wg_w\cdots)$ and define
\begin{equation}
  M^*_{n;q} = \EE\bigl(|X_n - x_w|^q\bigr).
\end{equation}
For $n=1$, $\nu_1$ is Lebesgue measure on $[0,1]$, and
therefore, for $q\geq0$,
\begin{equation}
  M^*_{1;q}
  =
  \int_0^1 |x-x_w|^q\,dx
  =
  \frac{x_w^{q+1}+(1-x_w)^{q+1}}{q+1}.
\end{equation}
We establish for $n\geq2$:
\begin{prop}\label{prop:Mnqmaj}
Define $R_w= \max(x_w,1-x_w)$. For $n\geq2$ and $q\geq1$:
  \begin{equation}\label{eq:Mnqmaj}
    M^*_{n;q}
    \leq
    R_w^q
    \Bigl(
      1+
      \frac{1}{(1-b^{-n})(b^{\tau_w}-1)}
    \Bigr)
    \sum_{r\geq1}
    b^{-(n-1)r}(1-b^{-r})^q\,.
\end{equation}
\end{prop}
For $n$ going to infinity and $q$ fixed, Equation \eqref{eq:Mnqmaj} implies
exponential decrease of $M^*_{n;q}$. We will show in the next section, that it
does so in a way which allows us to obtain exponential decrease for some column
and row sums of a triangular matrix related to how the
$K_w$-eigendistributions differ from the successive (normalized) derivatives
of the Dirac point mass at $x_w$.
\begin{proof}[Proof of Proposition~\ref{prop:Mnqmaj}]
The number $x_w$ is the fixed point of the affine map
$\phi_{g_w}(x) = x(g_w) + b^{-\tau_w}x$:
\begin{equation}\label{eq:xw}
  x_w
  =
  x(g_w)+b^{-\tau_w}x_w
  =
  \frac{x(g_w)}{1-b^{-\tau_w}}
  =
  \frac{n(g_w)}{b^{\tau_w}-1}.
\end{equation}
Every $g\in\cG$ ends in $g_w$. Indeed, $wg$ ends in $w$, and the length of 
$g$ is at least the one of $g_w$, so trimming the added $w$ the conclusion
follows. Set 
\begin{equation}
  g=g'g_w, \qquad r(g) = |g'| = |g| - |g_w|.
\end{equation}
Thus $r(g)\geq1$ for $g\in \cG\setminus\{g_w\}$.
We compute:
\begin{equation}
  \phi_g(x_w)
  =
  \phi_{g'}\bigl(\phi_{g_w}(x_w)\bigr)
  =
  \phi_{g'}(x_w) = x(g') + b^{-r(g)}x_w.
\end{equation}
As $0\leq x(s) \leq 1 - b^{-|s|}$ for any word $s$, we deduce
\begin{equation}
  b^{-r(g)}x_w
  \leq
  \phi_{g}(x_w)
  \leq
  1-b^{-r(g)}+b^{-r(g)}x_w.
\end{equation}
The distance of the two endpoints from $x_w$, which is in the same interval,
are respectively $x_w(1-b^{-r(g)})$ and $ (1-x_w)(1-b^{-r(g)})$, and we obtain
\begin{equation}
  | \phi_g(x_w) - x_w |\leq R_w(1 -b^{-r(g)}).
\end{equation}
Defining
$
  d(g)=\phi_g(x_w)-x_w,
$
we have obtained
\begin{equation}\label{eq:dgwbound}
   |d(g)| \leq R_w(1 - b^{-r(g)}).
\end{equation}
Combining in a telescopic manner
\begin{equation}
  \begin{aligned}
    x_w+d(G_1)&= x(G_1) + b^{-|G_1|}x_w \\
    x_w+d(G_2)&= x(G_2) + b^{-|G_2|}x_w \\
    \dots &= \dots\\
    x_w+d(G_i)&= x(G_i) + b^{-|G_i|}x_w,
  \end{aligned}
\end{equation}
we get
\begin{equation}
    x_w + \sum_{1\leq j \leq i} b^{-\sum_{q<j}|G_q|} d(G_j) 
     = b^{-|G_1|-\dots - |G_i|}x_w + 
    \sum_{1\leq j \leq i} b^{-\sum_{q<j}|G_q|}x(G_j).
\end{equation}
Letting $i\to\infty$,
\begin{equation}
  x_w
+  \sum_{j\geq1}
  d(G_j)b^{-|G_1|-\cdots-|G_{j-1}|}
  = X_n,
\end{equation}
and consequently, using immediately Equation \eqref{eq:dgwbound},
\begin{equation}
  |X_n - x_w| \leq R_w 
     \sum_{j\geq1} (1 - b^{-r(G_j)})b^{-|G_1|-\cdots-|G_{j-1}|}\,.
\end{equation}
The series is pointwise convergent everywhere.  It represents a random
real number $Y_n$  whose radix expansion contains stretches of $(b-1)$'s and
between them $\tau_w$ $0$'s.  It verifies a simple auto-similarity affine rule:
\begin{equation}
  Y_n = 1 - b^{-r(G_1)} + b^{-r(G_1) - \tau_w}Z_n, \qquad Z_n\sim Y_n\,.
\end{equation}
The random variable $Z_n$ is independent of $G_1$,
has the same law as $Y_n$, and takes its values in $[0,1]$. For $r\geq0$ and
$0\leq t\leq1$, convexity gives
\begin{equation}
  \begin{aligned}
    \bigl(1-b^{-r}+b^{-r-\tau_w}t\bigr)^q
    &\leq
    (1-b^{-r})^q
    \\
    &\quad+
    \frac{b^{1-\tau_w}t}{b-1}
    \Bigl(
      (1-b^{-(r+1)})^q-(1-b^{-r})^q
    \Bigr).
  \end{aligned}
\end{equation}
So, applying this and using the independence of
$Z_n$ and $G_1$:
\begin{equation}\label{eq:Ynbounda}
  \begin{aligned}
    \EE(Y_n^q)
    &\leq
    \EE\bigl((1-b^{-r(G_1)})^q\bigr)
    \\
    &\quad+
    \frac{b^{1-\tau_w}}{b-1}\EE(Y_n)
    \EE\left(
      (1-b^{-(r(G_1)+1)})^q-(1-b^{-r(G_1)})^q
    \right).
  \end{aligned}
\end{equation}
There are at most $b^r$ words $g\in\mathcal G$ such that $r(g)=r$, while
$g_w$ alone contributes $b^{-n\tau_w}$ to $\lambda_n(w)$. Hence, for
$r\geq1$,
$
  \PP\bigl(r(G_1)=r\bigr)
  \leq
  \frac{b^r b^{-n(r+\tau_w)}}{\lambda_n(w)}
  \leq
  b^{-(n-1)r}\,,
$
and
\begin{equation}\label{eq:Ynboundb}
  \EE\bigl((1-b^{-r(G_1)})^q\bigr)
  \leq
  \sum_{r\geq1}
  b^{-(n-1)r}(1-b^{-r})^q\,.
\end{equation}
For the second expectation in Equation \eqref{eq:Ynbounda}, using the
previous inequality, and only the trivial bound $1$ for the probability that
$r(G_1)=0$, we obtain
\begin{equation}\label{eq:Ynboundc}
  \begin{aligned}
    &\EE\left(
      (1-b^{-(r(G_1)+1)})^q-(1-b^{-r(G_1)})^q
    \right)
    \\
    &\quad\leq
    (1-b^{-1})^q
    +
    \sum_{r\geq1}b^{-(n-1)r}
    \Bigl(
      (1-b^{-(r+1)})^q-(1-b^{-r})^q
    \Bigr)
    \\
    &\quad=
    (b^{n-1}-1)
    \sum_{r\geq1}b^{-(n-1)r}(1-b^{-r})^q.
  \end{aligned}
\end{equation}
It remains to estimate the first moment of $Y_n$. Taking expectations in
the (stochastic) affine auto-similarity gives
\begin{equation}
  \EE(Y_n)
  =
  \frac{\EE(1-b^{-r(G_1)})}
  {1-b^{-\tau_w}\EE(b^{-r(G_1)})}.
\end{equation}
By Equation~\eqref{eq:Ynboundb} with $q=1$,
\begin{equation}
    \EE(1-b^{-r(G_1)})
    \leq
    \sum_{r\geq1}b^{-(n-1)r}(1-b^{-r})
    =
    \frac{1}{b^{n-1}-1}-\frac{1}{b^n-1}.
\end{equation}
So
\begin{equation}
  \EE(Y_n)
  \leq
  \frac{1}{1-b^{-\tau_w}}
  \left(
    \frac{1}{b^{n-1}-1}-\frac{1}{b^n-1}
  \right).
\end{equation}
Consequently,
\begin{equation}\label{eq:Ynboundd}
  \frac{b^{1-\tau_w}}{b-1}
  (b^{n-1}-1)\EE(Y_n)
  \leq
  \frac{1}{(1-b^{-n})(b^{\tau_w}-1)}.
\end{equation}
We obtain finally from Equations~\eqref{eq:Ynbounda}, \eqref{eq:Ynboundb},
\eqref{eq:Ynboundc} and \eqref{eq:Ynboundd}
\begin{equation}
  \EE(Y_n^q)
  \leq
  \left(
    1+
    \frac{1}{(1-b^{-n})(b^{\tau_w}-1)}
  \right)
  \sum_{r\geq1}
  b^{-(n-1)r}(1-b^{-r})^q.
\end{equation}
From $|X_n-x_w|\leq R_wY_n$, we get Equation~\eqref{eq:Mnqmaj}.
\end{proof}

Using $n\geq2$, $\tau_w\geq1$, $b\geq2$, we can bound the prefactor by
$\frac73$. And using a comparison with an integral on $\Iff01$, as in \cite[Proof of
Prop.\@ 2.9]{burnolzeta}, but now for an upper, not a lower bound, we obtain
a slightly weaker bound for the moments, which will prove useful in the next
section: 
\begin{prop}\label{prop:Mnqmaj2} For $n\geq2$,
  $q\geq1$,
\begin{equation}\label{eq:Mnqmaj2}
  \EE(|X_n - x_w|^q)\leq \frac73
  \frac{(n-1)R_w^q}{b^{n-1}-1}\int_0^1 x^{n-2}(1 -b^{-1}x)^q\dx.
\end{equation}
\end{prop}
\begin{proof}
  We use $(b^{n-1}-1)b^{-(n-1)r} = \int_{b^{-r}}^{b^{1-r}}(n-1)x^{n-2}\dx$ and
  then the fact that $(1-b^{-r})^q\leq (1-b^{-1}x)^q$ for $b^{-r}\leq x \leq
  b^{1-r}$, $r\geq1$.
\end{proof}

\section{Exponentially small off-diagonal row and column sums}

We choose some $R>R_w=\max(x_w,1-x_w)$ and set $\rho=R_w/R<1$.  We work in the
Hardy space $\cH_R =H^2\bigl(D(x_w,R)\bigr)$ and its dual $\cH_R'$.  By
Fréchet-Riesz, this dual is canonically conjugate-linear isomorphic with
$\cH_R$, but we will not use this identification.  Scalar products are denoted
$\scalp{\,{\cdot}\!}{\!{\cdot}\,}$ and are complex linear in the second entry.

The standard orthonormal basis of $\cH_R$ is given by the polynomial functions
\begin{equation}
  e_m(z)=\frac{(z-x_w)^m}{R^m}\;,
  \qquad m\geq0.
\end{equation}
For $m\geq0$, we let $\ve_m\in\cH_R'$ be defined by
\begin{equation}
  \ve_m(f) = \frac{R^{m}}{m!} f^{(m)}(x_w),
\end{equation}
so $(\ve_m)_{m\geq0}$ is the canonical orthonormal basis of
$\cH_R'$.  It is the dual system to $(e_m)_{m\geq0}$.

Recall that $\sigma_n$, $n\geq1$, is defined to be the distribution
$(-1)^{n-1}\nu_n^{(n-1)}/(n-1)!$, so, on test functions,
\begin{equation}
   \sigma_n(f) = \frac{1}{(n-1)!}  \int_0^1 f^{(n-1)}(x)\,d\nu_n(x).
\end{equation}
This linear functional is bounded for the $\cH_R$ norm (as $\Iff01$ is a
compact subset of $D(x_w,R)$), and we also view $\sigma_n$ as an element of
the dual $\cH_R'$.  The above formula remains valid as is for $f\in\cH_R$.

Consider now the infinite matrix $M^{(R)}=(m_{p,n}^{(R)})_{p\geq0,n\geq1}$ with
\begin{equation}
  m_{p,n}^{(R)} = \scalp{\ve_p}{R^{n-1}\sigma_n}_{\cH_R'} = R^{n-1}\sigma_n(e_p).
\end{equation}
We thus have $m_{pn}=0$ for $p<n-1$ and, for $p\geq n-1$,
\begin{equation}
  m_{p,n}^{(R)} = R^{n-1-p}\binom{p}{n-1}\int_0^1(x-x_w)^{p-n+1}\dnu_n(x).
\end{equation}
The matrix $M$ is lower triangular and its diagonal (given by the condition
$p=n-1$ on the row and column indices) has only $1$'s.

Write $q = p-n+1\geq1$. It indices the sub-diagonals. From Proposition
\ref{prop:Mnqmaj2}, for every $q\geq1$ and every $n\geq2$:
\begin{equation}\label{eq:corebound}
  |m_{q-1+n,n}^{(R)}|\leq \frac73\frac{(q+1)_{n-1}\rho^q}{(b^{n-1}-1)(n-2)!}
                    \int_0^1 x^{n-2}(1 -b^{-1}x)^q\dx.
\end{equation}
Keeping $n\geq2$ fixed, we make the summation over $q\geq1$:
\begin{align}
  \sum_{q=1}^\infty |m_{q-1+n,n}^{(R)}| &\leq
  \frac{7(n-1)}{3(b^{n-1}-1)}\int_0^1 x^{n-2} \left(\frac{1}{\bigl(1 - \rho(1
      - b^{-1}x)\bigr)^n} - 1\right)\dx
  \\
  &=\frac{7}{3(b^{n-1}-1)}\left(
    \int_1^\infty\frac{(n-1)\du}{\bigl((1-\rho)u+b^{-1}\rho\bigr)^n}
    -1\right)
  \\
  &=\frac{7}{3(b^{n-1}-1)}\left(\frac1{(1-\rho)(1-\rho +
      b^{-1}\rho)^{n-1}}-1\right)
  \\
  &\leq\frac{7b}{3(b-1)(1-\rho)}\frac1{((1-\rho)b + \rho)^{n-1}}\,.
\end{align}
We used $(b^{n-1}-1)^{-1}\leq b^{-(n-2)}/(b-1)$ in the last step.
This is $O(\theta_1^n)$ with $\theta_1 = \bigl((1-\rho)b + \rho\bigr)^{-1}<1$.
So the $\ell^1$ norms of the columns (strictly below the diagonal) are
exponentially decreasing for $n\geq2$. 

For $n=1$, $\nu_1=\Leb$ and we have the exact formula
\begin{equation}
  m_{p,1}^{(R)} = R^{-p}\int_0^1(x-x_w)^{p}\dx
                = \frac{(1-x_w)^{p+1} +(-1)^p x_w^{p+1}}{(p+1)R^p}
\;,
\end{equation}
so
\begin{equation}\label{eq:mp1bound}
  |m_{p,1}^{(R)}|\leq \frac{2R\rho^{p+1}}{p+1}\;,
\end{equation}
which is a summable sequence.  So the off-diagonal matrix elements of $M^{(R)}$
are summable.

We now estimate the row sums majorants $|m_{p,1}^{(R)}| + h_p^{(R)}$,  with
\begin{equation}
  h_p^{(R)}= \sum_{2\leq n \leq p} |m_{p,n}^{(R)}|,
\end{equation}
for $p\geq1$.  From the core bound Equation \eqref{eq:corebound}, and using
again $(b^{n-1}-1)^{-1}\leq b^{-(n-2)}/(b-1)$, there holds,
\begin{align}
  h_p^{(R)} &\leq \frac{7}{3(b-1)}\int_0^1 p\Bigl(\frac xb + \rho - \rho\frac xb\Bigr)^{p-1}\dx
\\
&\leq\frac{7b}{3(b-1)(1-\rho)}
 \left(\bigl((1-\rho)b^{-1} + \rho\bigr)^p - \rho^{p}\right).
\end{align}
So, combining with the bound for $|m_{p,1}^{(R)}|$:
\begin{equation}
  |m_{p,1}^{(R)}| + h_p^{(R)} = O(\theta_2^p), \qquad \theta_2 = (1-\rho)b^{-1} + \rho<1.
\end{equation}

Using Proposition~\ref{prop:rieszhilbert}, the summability of the off-diagonal
coefficients implies that $(R^{n-1}\sigma_n)_{n\geq1}$ is a Riesz basis in
$\cH_R'$, and that the dual system of normalized eigenpolynomials
$(\cB_m/R^m)_{m\geq0}$ is a Riesz basis in $\cH_R$.  More precisely, following
the proof of Proposition~\ref{prop:rieszhilbert}, let $S$ be the bounded
invertible linear map on $\cH_R'$ which sends $\ve_{n-1}$, $n\geq1$, to
$R^{n-1}\sigma_n$. Denote by $S^T$ its transposed action on $\cH_R$ (so
$\lambda(S^Tv) = (S(\lambda))(v)$ for every $\lambda\in \cH_R'$ and $v\in
\cH_R$). Then $R^{-m}\cB_m = (S^{-1})^Te_m $.  It follows that
\begin{equation}
  R^{-m}\cB_m-e_m = - (S^{-1})^T(S^T-I)e_m
\end{equation}
and therefore
\begin{equation}
  \begin{split}
    \|R^{-m}\cB_m-e_m\|_{\cH_R}
    &\leq
    \|S^{-1}\|\Bigl(
    \sum_{n=1}^{m}
    |R^{n-1}\sigma_n(e_m)|^2
    \Bigr)^{1/2}
\\
    &\leq   \|S^{-1}\|
    \sum_{n=1}^{m}
    |R^{n-1}\sigma_n(e_m)|.
  \end{split}
\end{equation}
This is $O(\theta_2^m)$, and 
we have established
\begin{prop}\label{prop:wBmRiesz}
  For any $R>R_w$, the normalized eigenpolynomials $(R^{-m}\cB_m)_{m\geq0}$
  are exponentially close in $\cH_R$-norm, as $m\to\infty$, to the normalized
  monomials $(e_m)$ giving the canonical orthonormal basis.  They are a Riesz
  basis of $\cH_R$.
\end{prop}

\section{Main Theorem: the modal expansion for \texorpdfstring{$k\geq1$}{k at least 1}}
\label{sec:part2main}

Recall $\cW_k = \{g \in \cW, k_w(g)=k\}$ and its language factorization $
\cW_k = \cL\cG^{k-1}\cR$, which applies for $k\geq1$.  The middle term $\cG$
is directly attached to the return operator ($\wK_w$ on finite measures on the
word space, $K_w$ on the measures and distributions supported on $\Iff01$,
$\sK_w$ on functions in suitable spaces, in particular in Hardy spaces
$H^2(D(x_w,R))$, $R>R_w$).  We now associate to $\cL$ an ``initial state'',
and to $\cR$ a ``final'' one, represented by a linear functional. This will
lead us directly and with surprising ease to a modal expansion
for block Irwin sums $I(b,w,k)$, $k\geq1$.

Recall the notation $n(g)$ (sometimes also $\overline g$) for the integer
given by a word $g$ (with leading digits on the left).  Also, $g_1$ below is
the leading digit of $g$. Attach to the left language $\cL$ the following
``initial state'', represented by a complex analytic function of the variable
$z$:
\begin{equation}\label{eq:fLdef}
  f_{\cL}(z)
  =
  \sum_{\substack{g\in\cL\\ g_1>0}}
  \frac1{n(g)+z}.
\end{equation}
As $\sum_{g\in\cL, g_1>0} n(g)^{-1}<\infty$ follows immediately from comparing
with $\sum_{g\in\cL}b^{-|g|} = 1 < \infty$ (see Equation \eqref{eq:masses}),
it is immediate that the series defining $f_{\cL}$ converges uniformly on any
compact subset of the complex plane avoiding the set $\{-n(g),
g\in\cL, g_1>0\}$, and more precisely that it defines a meromorphic function with simple
poles at these locations.

Define the positive integer $n_0 = \min_{\substack{g\in\cL\\g_1>0}} n(g) $.
Fortunately, the key inequality below holds, which says that the closest pole
to $x_w$ is outside the smallest closed disk centered at $x_w$ and containing
the unit interval:
\begin{equation}\label{eq:mLradius}
  n_0+x_w>R_w.
\end{equation}
If $x_w>0$, this follows at once from $n_0\geq1$ and $R_w\leq1$. Consider the
case $x_w=0$, which means $w=0^p$; any element of $\cL$ with non-zero first
digit must have length at least $p+1$, and the integer represented by it is at
least $b^p$. And $b^p>1=R_w$, so Equation \eqref{eq:mLradius} holds.

Take now $R\in \Ioo{R_w}{n_0+x_w}$, which is possible by \eqref{eq:mLradius}.
By the Riesz basis theorem Proposition~\ref{prop:wBmRiesz}, and the
biorthogonality between the eigenpolynomials and eigendistributions of the
return operators $\sK_w$ and $K_w$, respectively, we can expand $f_{\cL}$ in
the Hardy space $\cH_R = H^2(D(x_w,R))$ as a norm convergent series
\begin{equation}
  f_{\cL}
  =
  \sum_{j=1}^{\infty}
  \bigl(\int_0^1
    \frac{f_{\cL}^{(j-1)}(x)}{(j-1)!}\dnu_j(x)\bigr)
  \cB_{j-1}.
\end{equation}
Of course, $R$ has disappeared here when pairing the normalizations: we are
always handling the same analytic function, independent of $R$.  Define, for
$j\geq1$
\begin{equation}\label{eq:frakLj}
  \fL_j(b,w)
  =
  \frac{1}{(j-1)!}
  \int_0^1
  f_{\cL}^{(j-1)}(x)\,d\nu_j(x),
\end{equation}
so the modal expansion of the ``initial state'' reads
\begin{equation}\label{eq:fLmodal}
  f_{\cL}
  =
  \sum_{j=1}^{\infty}
  \fL_j(b,w)\cB_{j-1}.
\end{equation}

Let $k\geq1$. Recall that, by its definition and $\{g\in\cW, k_w^+(g)=k-1\} =
\cG^{k-1}\cR$, $\eta_{k-1}$ is the discrete measure
\begin{equation}
  \sum_{s\in\cG^{k-1}\cR}
  b^{-|s|}\delta_{x(s)}.
\end{equation}
Combining with the definition of $f_{\cL}$,
\begin{equation}
  \int_0^1 f_{\cL}(x)\,d\eta_{k-1}(x)
  =
  \sum_{\substack{g\in\cL\\g_1>0}}
  \sum_{s\in\cG^{k-1}\cR}
  \frac{b^{-|s|}}{n(g)+x(s)}
  =
  \sum_{\substack{g\in\cL,\,
                    s\in\cG^{k-1}\cR\\
                    g_1>0}}
  \frac1{n(gs)}\;,
\end{equation}
which is exactly $I(b,w,k)$, from the language decomposition $\cW_k = \cL\cG^{k-1}\cR$.

Integration against the finite measure $\eta_{k-1}$, supported on $[0,1]$, is
a bounded linear functional on this Hardy space $\cH_R$. We may therefore
integrate Equation~\eqref{eq:fLmodal}. From Equation~\eqref{eq:etakBmw},
\begin{equation}
  \int_0^1
  \cB_{j-1}(x)\,d\eta_{k-1}(x)
  =
  \lambda_j(w)^{k-1}
  \int_0^1
  \cB_{j-1}(x)\,d\eta_0(x).
\end{equation}
Define now
\begin{equation}\label{eq:frakRj}
  \fR_j(b,w)
  =
  \int_0^1
  \cB_{j-1}(x)\deta_0(x).
\end{equation}
It follows that
\begin{equation}
  I(b,w,k)
  =
  \sum_{j=1}^{\infty}
  \fL_j(b,w)
  \lambda_j(w)^{k-1}
  \fR_j(b,w),
\end{equation}
with absolute convergence from Cauchy--Schwarz.

We identify the first mode $j=1$ in the above series. As $\cB_0=1$,
$\lambda_1(w)=1$, and $\eta_0$ has total mass $b^p$, $\fR_1(b,w)=b^p$.  Also
$\nu_1$ is Lebesgue measure, and therefore
\begin{equation}
  \fL_1(b,w)
  =
  \sum_{\substack{g\in\cL\\g_1>0}}
  \int_0^1\frac{dx}{n(g)+x}.
\end{equation}
The term indexed by $g$ is the integral of $dt/t$ over the cylinder
$\Cyl_g$. These cylinders partition $[b^{-1},1)$ up to a set of Lebesgue
measure zero. Hence $\fL_1(b,w) = \log b$.

We have achieved what was the main initial objective of this paper.
We use temporarily the notation $\lambda_j(b,w)$, and not the abridged one $\lambda_j(w)$.
\begin{theo}[The modal expansion of block-Irwin sums]\label{thm:part2main}
For every $k\geq1$,
\begin{equation}
  I(b,w,k)
  =
  \sum_{j=1}^{\infty}
  \fL_j(b,w)\,
  \lambda_j(b,w)^{k-1}\,
  \fR_j(b,w).
\end{equation}
The series is absolutely convergent.  The quantities
\begin{equation}
  1 = \lambda_1(b,w)>\lambda_2(b,w)>\dots>\lambda_j(b,w)>\dots \to_{j\to\infty}0,
\end{equation}
are $\lambda_j(b,w)=G(b^{-j})$, $j\geq1$, with $G$ the generating function
for the cardinality counts per length of the return language $\cG$.  The
first term is $b^p\log(b)$.
\end{theo}
Admittedly, this is so far a rather theoretical result because we have not
said how we actually compute the coefficients $\fL_j(b,w)$ and $\fR_j(b,w)$
from their definitions Equations~\eqref{eq:frakLj} and \eqref{eq:frakRj}.  For
$\fR_j(b,w)$, this may be conceived as a matter of computing the eigenpolynomials and
knowing the moments of $\eta_0$. Both tasks ultimately reduce to triangular
systems for moment values.  For $\fL_j(b,w)$, we can start with writing it as
\begin{equation}
  \label{eq:fLSj}
  \fL_j(b,w) = (-1)^{j-1}
   \sum_{\substack{g\in\cL\\ g_1>0}}\int_0^1\frac{\dnu_j(x)}{(n(g) + x)^{j}}\;.
\end{equation}
and proceed from there,
as we will examine later.

There is also an indirect numerical way, which we have tested successfully,
to evaluate numerically the products
$\fL_j(b,w)\fR_j(b,w)$.  Indeed, the eigenvalues $\lambda_j(w)$ are
explicitly known from the overlap periods of the word $w$.  Assume that we
have some other means to evaluate directly finitely many values
\begin{equation}
  I(b,w,1),\ I(b,w,2),\ldots.
\end{equation}
Then we deduce numerical approximations of $ \fL_j(b,w)\fR_j(b,w)$.  For
example, get some multi-precision values of $I(b,w,k)$ for $1\leq k\leq 6$,
solve by a Vandermonde system, keep only the first four, check to which
fixed-point precision they work to give the value for $k=5$, and then consider
that for all higher $k$'s the same fixed-point precision will be achieved.
This works well in practice.

But for this, we do need some other means to compute numerically the block-Irwin values.
This is the topic of the next section, which expands upon \cite{burnolirwin,burnolone42}.

\section{Stieltjes functions and level-raising}

For $k\geq0$, define the Stieltjes function of $\eta_k$ by
\begin{equation}\label{eq:Vkz}
  V_k(z)
  =
  \int_{[0,1)}
  \frac{\deta_k(x)}{z+x},
  \qquad
  z\in\CC\setminus[-1,0].
\end{equation}
If $p=1$ this is the same as $U_k(z)$ from Equation~\eqref{eq:Ukz}. For $p>1$,
the core building blocks are the $\eta_k$ not the $\mu_k$.  We will mainly use
$V_k(z)$ for $z=n$ a positive integer.  Since $\eta_k([0,1))=b^p$,
$0<V_k(n)\leq\frac{b^p}{n}$. We also set
\begin{equation}
  V_{-1}(z)=0,
  \qquad
  \Delta V_k(z)=V_k(z)-V_{k-1}(z),
  \qquad k\geq0.
\end{equation}
For $|y|<1$, let $\rM_y$ and $\rH_y$ denote the push-forwards to $[0,1)$
of $\wM(y)$ and $\wH(y)$ respectively, under $x:\cW\to\Ifo01$. So, 
\begin{equation}
  \label{eq:MyHy}
  \rM_y=\sum_{k\geq0}y^k\mu_k,\qquad
  \rH_y=\sum_{k\geq0}y^k\eta_k\,.
\end{equation}
Define
\begin{equation}
  S_y(z)
  =
  \int_{[0,1)}
  \frac{\drH_y(x)}{z+x}
  =
  \sum_{k\geq0}y^kV_k(z).
\end{equation}

Let us express $I(b,w,k)$ in terms of the
functions $V_k$. Put
\begin{equation}
  \cI(y)=\sum_{k\geq0}I(b,w,k)y^k.
\end{equation}
Recall that for $s$ a word we have defined an operator $P_s$ acting on Borel
measures on $\Iff01$ via the push-forward under the associated affine map
$\phi_s(t) = x(s)+b^{-|s|}t$. Writing $A=\sum_{a=0}^{b-1}P_a$, we have from
Equation \eqref{eq:prune}
\begin{equation}\label{eq:pruneon01}
  \rM_y
  =
  \delta_0
  +b^{-1}A\rM_y
  -(1-y)b^{-p}P_w\rH_y.
\end{equation}
Let us integrate against the bounded Borel function
$x^{-1}\Un_{\Ifo{b^{-1}}{1}}$. We compute for $a\in\Sigma_b$, $a>0$, 
\begin{equation}
  \begin{split}
    \int_{\Ifo{b^{-1}}{1}}
    \frac{d\bigl(P_a\rM_y\bigr)(x)}{x}
    &=
    \int_{\Ifo{b^{-1}a}{b^{-1}(a+1)}}
    \frac{d\bigl(P_a\rM_y\bigr)(x)}{x}
\\    &=
    \int_{\Ifo{0}{1}}
    \frac{\drM_y(t)}{b^{-1}a + b^{-1}t}
    =
    \int_{[0,1)}
    \frac{b\drM_y(x)}{a+x}.
  \end{split}
\end{equation}
For $a=0$, the integral vanishes because the function is zero on the support
of $P_a\rM_y$.  The term prefixed by $w$ in Equation \eqref{eq:pruneon01}
contributes to the integral against $x^{-1}\Un_{\Ifo{b^{-1}}{1}}$
only when $w_1\neq0$; in that case, the computation is
\begin{equation}
  \begin{split}
    \int_{\Ifo{b^{-1}}{1}}
    \frac{d\bigl(P_w\rH_y\bigr)(x)}{x}
    &=
    \int_{\Cyl_w}
    \frac{d\bigl(P_w\rH_y\bigr)(x)}{x}
=    \int_{\Ifo{0}{1}}
    \frac{\drH_y(t)}{b^{-p}n(w) + b^{-p}t}
\\    &=
    \int_{\Ifo{0}{1}}
    \frac{b^p\drH_y(x)}{n(w)+x} = b^p S_y(n(w)).
  \end{split}
\end{equation}
Recall that we sometimes also use the notation $\overline g$ for $n(g)\in\NN$, $g\in\cW$.
We have therefore obtained
\begin{equation}\label{eq:irwinstiel}
  \cI(y)
  =
  \sum_{a=1}^{b-1}
  \int_{[0,1)}
  \frac{d\rM_y(x)}{a+x}
  -
  (1-y)\Un_{\{w_1\neq0\}}S_y(\overline w).
\end{equation}
From Equation~\eqref{eq:MfromH} on the word space we get after push-forward to
the unit interval
\begin{equation}\label{eq:MfromH01}
  \rM_y
  =
  \rH_y
  +(1-y)
  \sum_{c\in\cC_w^+}
  b^{-|c|}P_c\rH_y.
\end{equation}
Here, $\cC_w^+$ is the set of the complements of borders of
$w$.  Integrating against $x\mapsto 1/(n+x)$, where $n$ is a positive integer, we get
\begin{equation}
  \int_{\Ifo{0}{1}}
  \frac{\drM_y(x)}{n+x}
  =
  S_y(n)
  +(1-y)
  \sum_{c\in\cC_w^+}
  S_y\bigl(b^{|c|}n+\overline c\bigr).
\end{equation}
Taking $n=a$, $1\leq a<b$, and observing that $
  b^{|c|}a+\overline c=\overline{ac}$, we obtain:
\begin{equation}\label{eq:Iystiel}
  \cI(y)
  ={}
  \sum_{a=1}^{b-1}S_y(a)
  +(1-y)
  \sum_{c\in\cC_w^+}\sum_{a=1}^{b-1}
  S_y(\overline{ac})
  -
  (1-y)\Un_{\{w_1\neq0\}}S_y(\overline w).
\end{equation}
Extracting the coefficient of $y^k$ proves the following proposition:
\begin{prop}\label{prop:Iwkstiel}
For $b>1$ an integer, $w$ a non-empty word, with first digit $w_1$, $k$ a non-negative
integer, there holds
\begin{equation}\label{eq:Iwkstiel}
  I(b,w,k)
  =
  \sum_{a=1}^{b-1}V_k(a)
  +
  \sum_{c\in\cC_w^+}\sum_{a=1}^{b-1}
  \Delta V_k(\overline{ac})
  -
  \Un_{\{w_1\neq0\}}
  \Delta V_k(\overline w),
\end{equation}
where $V_k$ is the Stieltjes function associated with $\eta_k$ (and $V_{-1}$
is defined to be the zero function), $\Delta V_k = V_k - V_{k-1}$, and
$\cC_w^+$ is the set of complements of borders of $w$.
\end{prop}
Still following the framework originating in
\cite{burnolkempner,burnolirwin,burnolone42}, we next derive the level-raising
identity (the term \emph{level} originates in the convergent power series
provided in the quoted works, which can be triggered after an arbitrary number 
of iterations of the procedure we describe here; at the level of
Stieltjes functions, the integer arguments become larger, but the indices
either remain identical or decrease by one unit). Combining
Equations~\eqref{eq:pruneon01} and \eqref{eq:MfromH01} gives
\begin{equation}
  \begin{split}
    \rH_y - b^{-1} A\rH_y 
    &+ 
    (1-y)\sum_{c\in\cC_w^+}b^{-|c|}P_c
    \rH_y
\\
    &-(1-y)b^{-1}A\sum_{c\in\cC_w^+}b^{-|c|}P_c
    \rH_y
    +(1-y)b^{-p}P_w\rH_y
    =
    \delta_0\,.
  \end{split}
\end{equation}
We evaluate this against the Stieltjes kernel $(n+x)^{-1}$, where $n$
is a positive integer. The first two terms give
$S_y(n)-\sum_{a=0}^{b-1}S_y(bn+a)$.  The border term gives $(1-y)
\sum_{c\in\cC_w^+} S_y\bigl(b^{|c|}n+\overline c\bigr)$, whereas prefixing
once more by a digit gives $(1-y) \sum_{a=0}^{b-1}\sum_{c\in\cC_w^+}
S_y\bigl(b^{|c|+1}n+\overline{ac}\bigr)$.  Finally the $w$-term gives $
(1-y)S_y\bigl(b^pn+\overline w\bigr) $, and the Dirac mass at the empty word
becomes $1/n$. Moving all terms other than $S_y(n)$ to the other side of the
equation we get
\begin{equation}
\begin{split}
  S_y(n) = \frac1n
  +
  \sum_{a=0}^{b-1}S_y(bn+a)
  &-
  (1-y)
  \sum_{c\in\cC_w^+}
  S_y\bigl(b^{|c|}n+\overline c\bigr)
\\
  &+
  (1-y)
  \sum_{a=0}^{b-1}\sum_{c\in\cC_w^+}
  S_y\bigl(b^{|c|+1}n+\overline{ac}\bigr)
\\
  &-
  (1-y)S_y\bigl(b^pn+\overline w\bigr).
  \end{split}
\end{equation}
Coefficient extraction now gives the next Proposition:
\begin{prop}[Level raising]\label{prop:levelraising}
  For every $k\geq0$:
\begin{equation}\label{eq:levelraising}
  \begin{split}
  V_k(n) = \delta_{0,k}\frac1n
  +
  \sum_{a=0}^{b-1}V_k(bn+a)
  &-
  \sum_{c\in\cC_w^+}
  \Delta V_k\bigl(b^{|c|}n+\overline c\bigr)
  \\
  &+
  \sum_{a=0}^{b-1}\sum_{c\in\cC_w^+}
  \Delta V_k\bigl(b^{|c|+1}n+\overline{ac}\bigr)
  \\
  &-
  \Delta V_k\bigl(b^pn+\overline w\bigr).
  \end{split}
\end{equation}
\end{prop}
This combinatorial identity is the fully general extension of those from
\cite{burnolirwin}, \cite{burnollargebirwin} ($p=1$), and \cite{burnolone42}
($p=2$, $k\leq1$).

Using the identity, $V_k(n)$ is expressed as a finite linear combination of
Stieltjes values whose arguments are all at least equal to $bn$.  Hence the
expression \emph{raising the level}: those numbers have at least one
more digit than $n$ in radix $b$.  In the case $k=0$, and only then, there is
also a constant term $n^{-1}$. One application uses ``occurrence indices'' $k$
and $k-1$, repeated applications will use indices from $k$ down to possibly
$0$, and applying once more to a term at occurrence index $0$ inserts in the
expanded expression inverses of certain positive integers.

We can numerically compute the various $V_j(N)$, $0\leq j \leq k$, using the
geometric expansion of the kernel $(N+x)^{-1}$ in powers of $x/N$ and moment
values. As one iteration already guarantees all involved $N$'s are at least
$b$, we get geometric convergence. This was already the case without any
level-raising for the quantities involved in Proposition~\ref{prop:Iwkstiel},
\emph{except} for the kernel $1/(1+x)$, if present. In conclusion,
Propositions~\ref{prop:Iwkstiel} and \ref{prop:levelraising} allow one to
compute numerically to arbitrary precision the block Irwin sums from the
knowledge of the moments of the measures $\eta_j$, $j\leq k$. These moments
obey a triangular system of recurrences in their orders and in the index $j$.
This is the topic of the next section.

\section{Finite triangular moment recurrences}

Write
\begin{equation}\label{eq:vkm}
  v_{k;m}
  =
  \int_0^1x^m\,\deta_k(x),
  \qquad k,m\geq0.
\end{equation}
Equation~\eqref{eq:etaDE} gave us on the word space the relation for $k\geq1$
between $\weta_{k}$ and $\weta_{k-1}$, together with a defining equation for
$\weta_0$.  It uses operators $\wD_w$ and $\wE_w$ (on finite measures on the
word space) which are finite linear combinations of prefixing operators, as
given in Equation~\eqref{eq:wDwEw}.  As we have defined an action of the
prefixing operators on the Banach space of Borel measures on $\Iff01$, we can
define operators $D_w$, $E_w$, on this space expressed in terms of prefixing
operators there.  Recall the notation $A = \sum_{a\in\Sigma_b}P_a$, and that
for any word $s$, $P_s$ on Borel measures is the push-forward under the affine
map $\phi_s(t) = x(s) + b^{-|s|}t$,
\begin{align}
\label{eq:Dw01}
  D_w
  &=
  \bigl(I-b^{-1}A\bigr)(I+C_w)+b^{-p}P_w,
  \\
\label{eq:Ew01}
  E_w
  &=
  \bigl(I-b^{-1}A\bigr)C_w+b^{-p}P_w,
\\
D_w\eta_0&=\delta_0,
\\
  D_w\eta_k
  &=
  E_w\eta_{k-1}
\end{align}
It proves convenient to use some $|y|<1$ and to work, as in the preceding section,
with $\rH_y = \sum_{k=0}^\infty y^k \eta_k$.  The above is simply
\begin{equation}\label{eq:fooDwEwy}
  (D_w - y E_w)\rH_y = \delta_0\,.
\end{equation}
Define generally
for $\nu$ a complex
Borel measure on $\Iff01$, and $m\geq0$, 
\begin{equation}
  q_m(\nu) = \int_{\Iff01}x^m \dnu(x),
\end{equation}
and set in particular
\begin{equation}\label{eq:foovm}
  \sv_m = \sum_{k\geq0} y^k v_{k;m} = \int_{\Iff01}x^m \drH_y(x) = q_m(\rH_y).
\end{equation}
For any finite word $s$, and Borel measure $\nu$:
\begin{equation}\label{eq:fooPs}
  \begin{split}
    q_m(P_s\nu)
    &=
    \int_0^1x^m\,d(P_s\nu)(x)
    = \int_0^1 (x(s) + b^{-|s|} t)^m\dnu(t)
\\   &    =
    \sum_{i=0}^m\binom{m}{i}b^{-m|s|}n(s)^{m-i}q_i(\nu)
  \end{split}
\end{equation}
Expanding Equations~\eqref{eq:Dw01} and \eqref{eq:Ew01} to express $D_w$ and
$E_w$ as finite linear combinations of prefixing operators, we are led to
define the following quantities for $m\geq0$, $0\leq i \leq m$, where we set
$t_m = b^{-m-1}$ and use $0^0=1$:
\begin{gather}\label{eq:dmrblock}
    d_{m,i}
    =
    \delta_{mi}
    \begin{aligned}[t]
      &-t_m\sum_{a=0}^{b-1}a^{m-i}
      +\sum_{c\in\cC_w^+}
      t_m^{|c|}n(c)^{m-i}
      \\&-
      \sum_{\substack{a\in \Sigma_b\\ c\in\cC_w^+}}
      t_m^{|c|+1}n(ac)^{m-i}
      +t_m^pn(w)^{m-i}\,,
    \end{aligned}
    \\\label{eq:emrblock}
  e_{m,i}
  =
  \sum_{c\in\cC_w^+}
  t_m^{|c|}n(c)^{m-i}
  -
  \sum_{\substack{a\in \Sigma_b\\ c\in\cC_w^+}}
  t_m^{|c|+1}n(ac)^{m-i}
  +t_m^pn(w)^{m-i}\,.
\end{gather}

With this notation, Equation~\eqref{eq:fooDwEwy} gives, using
Equation~\eqref{eq:fooPs} and the $q_m$ functional:
\begin{equation}
  \sum_{i=0}^m \binom{m}{i}\bigl(d_{m,i}-y e_{m,i})\sv_i = \delta_{0m}\,.
\end{equation}
Recall from Equations~\eqref{eq:Dt} and \eqref{eq:Et} the polynomials $D
=(1-bt)C+t^p$ and $E =(1-bt)(C-1)+t^p$, where $C$ is the Guibas-Odlyzko
auto-correlation polynomial defined by Equation~\eqref{eq:Ct}.
There holds, using also Equation~\eqref{eq:RtGt},
\begin{equation}\label{eq:demdiag}
  d_{m,m}
  =
  D(t_m),
  \qquad
  e_{m,m}
  =
  E(t_m),
  \qquad
  \frac{e_{m,m}}{d_{m,m}}
  =
  \lambda_{m+1}(w).
\end{equation}
In particular $d_{m,m}\neq0$.   Indeed,
$D(t_0)=b^{-p}$, and for $m\geq1$ we have $0<t_m<b^{-1}$ and
$D(t_m)=(1-bt_m)C(t_m)+t_m^p>0$.

Define the infinite-dimensional lower-triangular matrices
$\brD$ and $\brE$ by their coefficients
\begin{equation}\label{eq:brDEentries}
  \brD_{m,i}
  =
  \binom{m}{i} d_{m,i},
  \qquad
  \brE_{m,i}
  =
  \binom{m}{i} e_{m,i},
  \qquad
  0\leq i\leq m.
\end{equation}
Represent $\sv_m$ as the row $(v_{0;m},v_{1;m},\ldots)$, and assemble these
rows, indexed by $m\geq0$, from top to bottom in a matrix $\brV$.
Define also the infinite-dimensional matrix $\brY$ such that
$\brY_{ab} = \delta_{a+1,b}$, so that it is upper-triangular with $1$'s only
on the first superdiagonal.  We obtain the matrix equation
\begin{equation}
  \brD\brV - \brE\brV\brY = (\delta_{0m}\delta_{0k})_{m\geq0,k\geq0}\,.
\end{equation}
The lower-triangular matrix $\brD$ with non-zero diagonal coefficients is
invertible.  Define
\begin{equation}\label{eq:brKDE}
  \brK=\brD^{-1}\brE.
\end{equation}
The matrix $\brK$ is again lower-triangular, and Equation~\eqref{eq:demdiag}
gives
\begin{equation}\label{eq:brKdiag}
  \brK_{m,m}
  =
  \frac{E(t_m)}{D(t_m)}
  =
  \lambda_{m+1}(w).
\end{equation}

If $v_k=(v_{k;0},v_{k;1},\ldots)^{\sT}$ denotes the $k$-th column of
$\brV$, the matrix equation above says
\begin{equation}\label{eq:vkbrK}
  \brD v_0
  =
  \begin{pmatrix}
    1\\0\\0\\\vdots
  \end{pmatrix},
  \qquad
  v_k=\brK v_{k-1}\quad(k\geq1).
\end{equation}
Thus
\begin{equation}\label{eq:vkdirect}
  v_k=\brK^k v_0.
\end{equation}
This is of course simply $\eta_k = K_w^{k}(\eta_0)$ at the level of the
moments; see Equation~\eqref{eq:Kwetak} for the corresponding identity on the
word space. We have followed the above presentation so that it becomes
actually possible to make a numerical implementation computing exactly the
quantities involved.

\begin{prop}[Recurrence formulas for the moments]
  The moments $(v_{k;m})$, $k\geq0$, $m\geq0$, of the measures $\eta_k$, can be
  computed via the following recurrences, which use the coefficients
  explicitly given by Equations~\eqref{eq:dmrblock} and \eqref{eq:emrblock}.
  For $k=0$ and
  $m\geq0$:
\begin{equation}\label{eq:v0mtriang}
  d_{m,m}v_{0;m}
  =
  \delta_{0m}
  -
  \sum_{r=1}^{m}
  \binom{m}{r}d_{m,m-r}v_{0;m-r}.
\end{equation}
For $k\geq1$,
\begin{equation}\label{eq:vkmmtriang}
  d_{m,m}v_{k;m}
  =
    -
  \sum_{r=1}^{m}
  \binom{m}{r}
    d_{m,m-r}v_{k;m-r}
+
  \sum_{r=0}^{m}
  \binom{m}{r}
    e_{m,m-r}v_{k-1;m-r}\,.
\end{equation}
One has $d_{m,m}=D(b^{-m-1})$, with $D$ the polynomial from Equation~\eqref{eq:Dt}. 
In particular $v_{0;0} = 1/D(b^{-1}) = b^p$, and, as $d_{0,0}=e_{0,0}=b^{-p}$,
$v_{k;0} = b^p$ for every $k\geq1$.
\end{prop}
\begin{proof}
  The case $k=0$ follows from the first relation in Equation~\eqref{eq:vkbrK}.
  For $k\geq1$ one uses the equality $\brD v_k=\brE v_{k-1}$ which is
  another way to express the second relation in that equation.  It is also
  possible to use $v_k = \brK v_{k-1}$, this gives an
  alternative linear recurrence, once
  the matrix coefficients of $\brK$ are known, a topic we turn to below.
\end{proof}
The above numerical scheme is the
generalization to the most general word $w$, of the
recurrences on moments given in \cite{burnolirwin} for $p=1$ and in
\cite{burnolone42} for $p=2$ but only $k\leq1$.

Equation~\eqref{eq:vkdirect} is an embodiment of the return operator idea,
which allows to better understand the dependence on $k$.  The matrix
coefficients of $\brK$ have an interesting interpretation.
For $j\geq1$ and $m\geq0$, consider the quantities
\begin{equation}\label{eq:Lambdajm}
  \Lambda_j^m(w)
  =
  \sum_{g\in\cG}b^{-j|g|}x(g)^m.
\end{equation}
We hope that the superscript $m$ will not create confusion with powers, it is
to be understood as tensor notation.  For $m=0$, we recover the eigenvalues:
\begin{equation}\label{eq:Lambdaj0}
  \Lambda_j^0(w)
  =
  \sum_{g\in\cG}b^{-j|g|}
  =
  \lambda_j(w).
\end{equation}
Thus $\Lambda_j^m(w)$, $m\geq1$, may be viewed as
arithmetically weighted variants of the purely combinatorial quantities
$\lambda_j(w)$. We now relate these values to the matrix elements of the return operator 
$\sK_w$ in the monomial basis of $\CC[X]$. Indeed, there holds
\begin{equation}\label{eq:Kwxmmom}
  \sK_w(x^m)
  =
  \sum_{g\in\cG}b^{-|g|}
  \bigl(x(g)+b^{-|g|}x\bigr)^m
  =
  \sum_{i=0}^m
  \binom{m}{i}
  \Lambda_{i+1}^{m-i}(w)x^i.
\end{equation}
Consequently its matrix in the monomial basis $(1,x,x^2,\ldots)$ is upper
triangular. Its $(i,m)$-entry, $m\geq0$, $0\leq i\leq m$, is
$\binom{m}{i}\Lambda_{i+1}^{m-i}(w)$.

On the other hand, $\brK=\brD^{-1}\brE$ is the transpose of this matrix,
acting on moment columns.
It follows from Equation~\eqref{eq:Kwxmmom} that
\begin{equation}\label{eq:brKLambda}
  \brK_{m,i}
  =
  \binom{m}{i}\Lambda_{i+1}^{m-i}(w),
  \qquad 0\leq i\leq m.
\end{equation}
Hence
\begin{equation}\label{eq:DELambda}
  \bigl(\brD^{-1}\brE\bigr)_{m,i}
  =
  \binom{m}{i}\Lambda_{i+1}^{m-i}(w).
\end{equation}
Thus the quantities $\Lambda_j^m(w)$, as well as the moments of the distributions
$\eta_k$, can be computed via a finite algebraic process involving the
quantities from Equations~\eqref{eq:dmrblock} and \eqref{eq:emrblock}.

It is possible to be a bit more explicit here than simply saying to invert the
triangular matrix $\brD$: we now show that the identity $\brD\brK=\brE$ yields a triangular
recurrence for the quantities $\Lambda_j^m(w)$, for \emph{each fixed $j\geq1$}.  The
initial value at $m=0$ is the known $\lambda_j(w)$. Indeed, taking the
$(m+j-1,j-1)$ entry of $\brE=\brD\brK$, and using
Equation~\eqref{eq:brKLambda}, we obtain
\begin{equation}
  \sum_{q=0}^m
  \binom{m+j-1}{q+j-1}d_{m+j-1,q+j-1}
  \binom{q+j-1}{j-1}\Lambda_j^q(w)
  =
  \binom{m+j-1}{j-1}e_{m+j-1,j-1}.
\end{equation}
Using
\[
  \binom{m+j-1}{q+j-1}\binom{q+j-1}{j-1}
  =
  \binom{m+j-1}{j-1}\binom{m}{q},
\]
this becomes
\begin{equation}\label{eq:Lambdatriang}
  \sum_{q=0}^m
  \binom{m}{q}
  d_{m+j-1,q+j-1}\Lambda_j^q(w)
  =
  e_{m+j-1,j-1}.
\end{equation}
Since
$d_{m+j-1,m+j-1}=D(t_{m+j-1}) =D(b^{-m-j})$, we get
\begin{equation}\label{eq:Lambdatriangrec}
  D(b^{-m-j})\Lambda_j^m(w)
  =
  e_{m+j-1,j-1}
  -
  \sum_{r=1}^{m}
  \binom{m}{r}
  d_{m+j-1,m+j-r-1}\Lambda_j^{m-r}(w).
\end{equation}
Thus, for each fixed $j\geq1$, Equation~\eqref{eq:Lambdatriangrec} computes
successively
$\Lambda_j^0(w)$, 
  $\Lambda_j^1(w)$,
  $\Lambda_j^2(w)$, \dots.
It recovers the $m=0$ value to be
$
  \Lambda_j^0(w)
  =
  \frac{E(t_{j-1})}{D(t_{j-1})}
  =
  \lambda_j(w)
$,
in agreement with Equation~\eqref{eq:Lambdaj0}.

\section{Effective computation of the modal coefficients}

We explain how the modal coefficients $\fL_j(b,w)$ and $\fR_j(b,w)$ can be
computed effectively.  There is a clear difference between the two: the latter
one is a rational number, and it can be computed exactly by a finite process;
the former one can be approximated numerically arbitrarily well, but does not
appear to be rationally related to $(b,w)$.

We start with
\begin{equation}
    \fR_j(b,w)
  =
  \int_0^1\cB_{j-1}(x)\,d\eta_0(x).
\end{equation}
Recall that $\sigma_j=(-1)^{j-1}\nu_j^{(j-1)}/(j-1)!$ and that, 
on the polynomials of degree at most $m$, the functionals
$\sigma_1,\ldots,\sigma_{m+1}$ form the dual basis to
$\cB_0,\ldots,\cB_m$.
We use this to eliminate the eigenpolynomials from the
computation of $\fR_j(b,w)$.  For $0\leq i\leq m$,
\begin{equation}
  \sigma_{i+1}(x^m)
  =
  \binom{m}{i}
  \int_0^1x^{m-i}\dnu_{i+1}(x)
  =
  \binom{m}{i} M_{i+1;m-i}.
\end{equation}
We used above the same notation as in Part~1:
\begin{equation}
  M_{j;m}
  =
  \int_0^1 x^m\dnu_j(x),
  \qquad j\geq1,\quad m\geq0.
\end{equation}
Thus $M_{j;0}=1$.
So,
\begin{equation}
  x^m
  =
  \sum_{i=0}^m
  \binom{m}{i}
  M_{i+1;m-i}\cB_i(x).
\end{equation}
Integrating against $\eta_0$, and recalling the notation 
$v_{0;m}=\int_0^1x^m\,d\eta_0(x)$,
we obtain
\begin{equation}
  v_{0;m}
  =
  \sum_{i=0}^m
  \binom{m}{i}
  M_{i+1;m-i}\fR_{i+1}(b,w).
\end{equation}
The coefficient of $\fR_{m+1}(b,w)$ in the right-hand side is $M_{m+1;0}=1$. So:
\begin{prop}[Computation of the right modal coefficients]%
  The right modal coefficients $\fR_{m+1}(b,w)$, $m\geq0$, can be computed,
  once the moments $M_{j;i}$ (with $j+i=m+1$) of the
  probability measures $\nu_j$, $j\geq1$, are known, via the following
  recurrence, which also uses the moments $v_{0;m}$ of $\eta_0$, already given
  successively by Equation~\eqref{eq:v0mtriang}:
\begin{equation}\label{eq:Rrec}
  \fR_{m+1}(b,w)
  =
  v_{0;m}
  -
  \sum_{i=1}^m
  \binom{m}{i}
  M_{m-i+1;i}\fR_{m-i+1}(b,w),
  \qquad m\geq1.
\end{equation}
The initial value is $\fR_1(b,w)=v_{0;0}=b^p$.
\end{prop}

We need to compute the moments
$M_{j;m}$.  We derive for them a self-contained recurrence,
using the matrices of the preceding section.
The column of moments of the eigendistribution
$\sigma_j$ has zero entries in degrees $0,\ldots,j-2$, while in degree
$j+m-1$ its entry is
\begin{equation}
  \sigma_j(x^{j+m-1})
  =
  \binom{j+m-1}{j-1}M_{j;m}.
\end{equation}
Since $\brK=\brD^{-1}\brE$ is the transpose, on moment
columns, of $\sK_w$ acting linearly on $\CC[X]$, this column is
an eigenvector of $\brK$ with eigenvalue $\lambda_j(w)$. Equivalently,
\begin{equation}
  \brE\,\sigma_j
  =
  \lambda_j(w)\brD\,\sigma_j.
\end{equation}
Using
Equation~\eqref{eq:brDEentries} together with the vanishing of the first
$j-1$ entries, and taking the row of index $m+j-1$, gives
\begin{equation}
  0
  =
  \sum_{i=0}^m
  \binom{j+m-1}{j+i-1}
  \bigl(
    e_{j+m-1,j+i-1}
    -
    \lambda_j(w)d_{j+m-1,j+i-1}
  \bigr)
  \binom{j+i-1}{j-1}M_{j;i}.
\end{equation}
After simplification of the binomial coefficients, we obtain
\begin{equation}
  \sum_{i=0}^m
  \binom{m}{i}
  \bigl(
    e_{j+m-1,j+i-1}
    -
    \lambda_j(w)d_{j+m-1,j+i-1}
  \bigr)
  M_{j;i}
  =
  0.
\end{equation}
For $i=m$, Equation~\eqref{eq:demdiag} gives
  \begin{equation}
    e_{j+m-1,j+m-1}
    -
    \lambda_j(w)d_{j+m-1,j+m-1}
    =
    D(b^{-j-m})
    \bigl(\lambda_{j+m}(w)-\lambda_j(w)\bigr).
  \end{equation}
In conclusion:
\begin{prop}[Moments of the exponentially tilted measures]%
\label{prop:nujmoments}
  For every $j\geq1$, the moments $M_{j;m}$, $m\geq0$,
   of the measure $\nu_j$, can
  be computed for $m\geq1$ via the following recurrence with
  the initialization $M_{j;0}=1$:
\begin{equation}\label{eq:Mrec}
  \begin{split}
    D(b^{-j-m})
    \bigl(\lambda_j(w)-\lambda_{j+m}(w)\bigr)M_{j;m}
\\    =
    \sum_{i=1}^m
    \binom{m}{i}
    \bigl(
    e_{j+m-1,j+m-i-1}
    -
    \lambda_j(w)d_{j+m-1,j+m-i-1}
    \bigr)
    M_{j;m-i}.
  \end{split}
\end{equation}
\end{prop}
\begin{proof}
Equation~\eqref{eq:Mrec} uses the quantities $d_{m,i}$ and
$e_{m,i}$ from Equations~\eqref{eq:dmrblock} and \eqref{eq:emrblock},
and the polynomial $D$ from Equation~\eqref{eq:Dt}.
We saw already $D(b^{-j-m})>0$, and
$\lambda_j(w)>\lambda_{j+m}(w)$ for $m\geq1$.
\end{proof}

We turn to the left modal coefficient.  As in Part~1, put for $j,n\geq1$
\begin{equation}\label{eq:Ajn}
  A_j(n)
  =
  \int_0^1\frac{d\nu_j(x)}{(n+x)^j}.
\end{equation}
These quantities can be numerically approximated to arbitrary precision once
one has evaluated the moments of $\nu_j$.  Indeed, suppose first that $n>1$,
then, uniformly for $0\leq x\leq1$,
\begin{equation}
  \frac1{(n+x)^j}
  =
  \sum_{q=0}^{\infty}
  (-1)^q
  \binom{j+q-1}{q}
  \frac{x^q}{n^{j+q}},
\end{equation}
and hence
\begin{equation}
  A_j(n)
  =
  \sum_{q=0}^{\infty}
  (-1)^q
  \binom{j+q-1}{q}
  \frac{M_{j;q}}{n^{j+q}},
  \qquad n>1.
\end{equation}
This series has the simple property, already encountered in Part~1 whenever
using Newton's binomial series, that its partial sums are alternately upper
and lower bounds for the full sum: $S_{2N+1}\leq A_j(n)\leq S_{2N}$ (here
$S_{2N}$ has $2N+1$ terms).  This is because Newton's series has the property
pointwise and we are integrating against a positive measure.  In particular,
for any $N$,
\begin{equation}
  \left|A_j(n)-S_N\right|
  \leq
  \left|S_{N+1}-S_N\right|
  =
  \binom{j+N}{N+1}
  \frac{M_{j;N+1}}{n^{j+N+1}}.
\end{equation}

We explain how to boost the speed of convergence, up to the price of having to
handle more series.  This is the same ``level-raising'' technique previously
encountered in the section on the Stieltjes transforms of the $\eta_k$
measures.  This process allows in particular to obtain for $n=1$ also
geometrically convergent representations.

For a finite word $s$, define the linear map on $C^{j-1}([0,1])$
\begin{equation}
  (\sT_s f)(x)
  =
  b^{-|s|}
  f\bigl(x(s)+b^{-|s|}x\bigr).
\end{equation}
Let
$\sC_w
  =
  \sum_{c\in\cC_w^+}\sT_c$,
$  \sA
  =
  \sum_{a=0}^{b-1}\sT_a$.
Transposing the finite measure-side maps $D_w,E_w$ gives the linear maps
on $C^{j-1}([0,1])$
\begin{equation}
  \sD_w
  =
  (I+\sC_w)(I-\sA)+\sT_w,
  \qquad
  \sE_w
  =
  \sC_w(I-\sA)+\sT_w.
\end{equation}
The order of the factors is reversed by transposition.  Since on finite
measures on $[0,1]$ the return operator is $D_w^{-1}E_w$, its transposed action
on functions is
\begin{equation}
  \sK_w=\sE_w\sD_w^{-1}.
\end{equation}
The modal relation
$\sigma_j(\sK_w f)=\lambda_j(w)\sigma_j(f)$ is therefore equivalently
\begin{equation}
  \sigma_j(\sE_w f)
  =
  \lambda_j(w)\sigma_j(\sD_w f).
\end{equation}
Expanding this identity gives
\begin{equation}
  \sigma_j(f)
  =
  \sigma_j(\sA f)
  +
  \frac{1-\lambda_j(w)}{\lambda_j(w)}
  \sigma_j\bigl(
    \sC_w(I-\sA)f+\sT_wf
  \bigr).
\end{equation}

Apply it to
\[
  f_n(x)=\frac1{n+x}.
\]
For every finite word $s$,
\begin{equation}
  \sT_sf_n(x)
  =
  \frac1{b^{|s|}n+n(s)+x},
\end{equation}
and therefore
\begin{equation}
  \sigma_j(\sT_sf_n)
  =
  (-1)^{j-1}
  A_j\bigl(b^{|s|}n+n(s)\bigr).
\end{equation}
We also have $
  \sT_c\sT_a=\sT_{ac}$.
It follows that, for every $j\geq1$ and every positive integer $n$,
\begin{equation}\label{eq:Ajnlevelraising}
  \begin{split}
  A_j(n)
  ={}&
  \sum_{a=0}^{b-1}A_j(bn+a)
  \\
  &+
  \frac{1-\lambda_j(w)}{\lambda_j(w)}
  \sum_{c\in\cC_w^+}
  A_j\bigl(b^{|c|}n+n(c)\bigr)
  \\
  &-
  \frac{1-\lambda_j(w)}{\lambda_j(w)}
  \sum_{a\in\Sigma_b, \,c\in\cC_w^+}
  A_j\bigl(b^{|c|+1}n+n(ac)\bigr)
  \\
  &+
  \frac{1-\lambda_j(w)}{\lambda_j(w)}
  A_j\bigl(b^pn+n(w)\bigr).
  \end{split}
\end{equation}
Every argument on the right-hand side is at least $bn$. For $j=1$,
$\lambda_1(w)=1$, and the identity reduces to
$A_1(n)=\sum_{a=0}^{b-1}A_1(bn+a)$.

We now express $\fL_j(b,w)$ by finitely many of these
(higher) Stieltjes values. In the following few lines, identify a language $\cA$ with the
corresponding element
$
  \sum_{g\in\cA}b^{-|g|}P_g
$
of the prefix algebra.
Extracting the coefficient of $y$ in Equation~\eqref{eq:MfromH} gives
\begin{equation}
  \wmu_1
  =
  \weta_1+C_w(\weta_1-\weta_0).
\end{equation}
The language factorizations established earlier give
  $\wmu_1=\cL\cR$,
  $\weta_1=\cG\cR$,
  $\weta_0=\cR$.
Substitution yields
\begin{equation}
  \cL\cR
  =
  \bigl(\cG+C_w(\cG-I)\bigr)\cR.
\end{equation}
The prefix element associated with $\cR$ is invertible: it is
$\wD_w^{-1}$.  Cancelling it on the right gives
\begin{equation}
  \cL
  =
  \cG+C_w(\cG-I).
\end{equation}
This is a signed identity in the prefix algebra.

We now retain only those words of $\cL$ whose first digit is non-zero. Let
$\Sigma_b^+=\{1,\ldots,b-1\}$, $\cL_+ = \{g\in\cL:g_1>0\}$.

If $w_1=0$,
prefixing a word of $\cL$ by a digit in $\Sigma_b^+$ cannot create a new
occurrence of $w$ starting at the first position. Hence
$\cL_+=\Sigma_b^+\cL$.

Suppose now that $w_1\neq0$.  A word in $\Sigma_b^+\cL$ fails to belong to
$\cL_+$ precisely when the newly prefixed digit creates an occurrence of $w$
at the first position, which means exactly that it is of the shape $wg$, $g\in
\cG$ (by definition, $g$ belongs to $\cG$ if and only if $wg$ has exactly two
occurrences of $w$, one as prefix, the other as suffix).  Therefore
$\Sigma_b^+\cL \setminus \cL_+ = w\cG$. The only element of $\cL_+$ not in
$\Sigma_b^+\cL$ is $w$, because it is the sole possible element of length at
most $p$.  So $\cL_+ = (\Sigma_b^+\cL \setminus w\cG)\cup\{w\}$, and in the
prefix algebra, we get if $w_1\neq0$:
\begin{equation}
  \cL_+
  =
  \Sigma_b^+\cL-w(\cG-I),
  \qquad w_1\neq0.
\end{equation}
Both cases are thus expressed simultaneously as
\begin{equation}
  \cL_+ 
  =
  \Sigma_b^+\cG
  +
  \bigl(
    \Sigma_b^+C_w
    -
    \mathbf1_{\{w_1\neq0\}}w
  \bigr)
  (\cG-I).
\end{equation}

To pass from this prefix identity to Stieltjes functions, associate with a
signed prefix element
\[
  Q=\sum_s q_s b^{-|s|}P_s,
\]
supported on words having a non-zero first digit, the function
\begin{equation}
  \Phi(Q)(z)
  =
  \sum_s\frac{q_s}{n(s)+z}.
\end{equation}
Recall that by definition the weights of $Q$ are supposed to be summable,
$\sum_s |q_s|b^{-|s|}<\infty$, so the above makes sense as a meromorphic
function in the entire complex plane, with simple poles.  For every such $Q$,
\begin{equation}
  \Phi(Q\cG)=\sK_w\Phi(Q),
\end{equation}
because
\[
  \sT_g\frac1{n(s)+z}
  =
  \frac1{n(sg)+z}.
\]
Put
\begin{equation}
  f_{\Sigma_b^+}
  =
  \sum_{a=1}^{b-1}\frac1{a+z}
\end{equation}
and
\begin{equation}
  h_w
  =
  \sum_{c\in\cC_w^+}\sum_{a=1}^{b-1}
  \frac1{n(ac)+z}
  -
  \mathbf1_{\{w_1\neq0\}}\frac1{n(w)+z}.
\end{equation}
We obtain from the signed representation of $\cL_+$, $f_{\cL} =
\sK_wf_{\Sigma_b^+}+(\sK_w-I)h_w$, or
\begin{equation}
  f_{\cL}+h_w
  =
  \sK_w(f_{\Sigma_b^+}+h_w).
\end{equation}
We express, introducing a new variable, this last relation through the finite
combination of prefix operators $\sD_w$ and $\sE_w$.  Set
\begin{equation}
  u=\sD_w^{-1}(f_{\Sigma_b^+}+h_w).
\end{equation}
Since $\sK_w=\sE_w\sD_w^{-1}$, we obtain the finite system
\begin{equation}
  \sD_wu=f_{\Sigma_b^+}+h_w,
  \qquad
  \sE_wu=f_{\cL}+h_w.
\end{equation}
Applying $\sigma_j$ and using
$
  \sigma_j(\sE_wu)
  =
  \lambda_j(w)\sigma_j(\sD_wu)
$
gives
\begin{equation}
  \sigma_j(f_{\cL})
  =
  \lambda_j(w)\sigma_j(f_{\Sigma_b^+})
  +
  \bigl(\lambda_j(w)-1\bigr)\sigma_j(h_w).
\end{equation}
Now
\[
  \sigma_j\left(\frac1{n+z}\right)
  =
  (-1)^{j-1}A_j(n),
\]
and Equation~\eqref{eq:frakLj} identifies
$\sigma_j(f_{\cL})$ with $\fL_j(b,w)$. We have therefore
proved
\begin{prop}[Computation of the left modal coefficients]
  The left modal coefficients $\fL_j(b,w)$ can be expressed as
  the following finite linear combinations of the 
  coefficients $A_j(n)$, for various positive integers $n$:
\begin{equation}\label{eq:Lcoef}
  \begin{split}
  \fL_j(b,w)
  =
  (-1)^{j-1}\Bigg(
  &\lambda_j(w)
  \sum_{a=1}^{b-1}A_j(a)
  \\
  &+
  \bigl(\lambda_j(w)-1\bigr)
  \sum_{c\in\cC_w^+}\sum_{a=1}^{b-1}
  A_j(n(ac))
  \\
  &-
  \bigl(\lambda_j(w)-1\bigr)
  \mathbf1_{\{w_1\neq0\}}A_j(n(w))
  \Bigg).
  \end{split}
\end{equation}
\end{prop}
The $A_j(n)$ are defined in Equation~\eqref{eq:Ajn}, and, as previously
explained, they are numerically reducible to geometrically convergent series
involving the moments $M_{j;m}$ (see Proposition~\ref{prop:nujmoments} for the
latter), $m\geq0$, of the probability measures $\nu_j$, either directly when
$n>1$, or indirectly via Equation~\eqref{eq:Ajnlevelraising} which allows to
use larger $n$'s giving better numerical convergence. Regarding $A_j(1)$, we
could expand with center at $\frac12$, but this looks inelegant. And anyhow,
the ``level-raising'' transformation Equation~\eqref{eq:Ajnlevelraising}
is recommended for every $n<b$.

We have thus reduced both modal coefficients to triangular recurrences of
moment data.  The moments $v_{0;m}$ are obtained from
Equation~\eqref{eq:v0mtriang}, the moments $M_{j;m}$ from
Equation~\eqref{eq:Mrec}, and the right coefficients directly from the unit
triangular recurrence~\eqref{eq:Rrec}, without needing to compute the
eigenpolynomials.  Of course, the latter can also be computed if one is
interested in them, and then the numerical scheme can be reorganized to use
that information.

\section{What happens for \texorpdfstring{$k=0$}{k=0}}

Our main Theorem~\ref{thm:part2main} is for $k\geq1$.  It is intimately related
to the language factorization $\cW_k = \cL\cG^{k-1}\cR$ which does not apply
for $k=0$.  In fact $\cW_0=\cF$ is the occurrence-free language. Nevertheless
one may wonder what happens if we formally set $k=0$ in the modal expansion,
thus considering
\begin{equation}
  S = \sum_{j\geq1}
  \fL_j(b,w)\lambda_j(w)^{-1}\fR_j(b,w).
\end{equation}
Here is the answer.

If $w$ has a border (so $p>1$, necessarily), let $v$ be the
longest one.  Then $S$ is absolutely convergent and converges to
$I(b,w,0)-I(b,v,0)$.

Suppose now that $w$ is unbordered.
If $w_1=0$, then $S$ is absolutely convergent and converges to $I(b,w,0)$.

Finally, if $w$ is unbordered and $w_1\neq 0$ then the terms of $S$ do not
tend to zero, but to $x_w^{-1}$.  An alternative absolutely convergent modal
formula for $I(b,w,0)$ is available.  Establishing the preceding divergence,
as well as the modified and convergent formula, requires a somewhat long
argument which we have decided not to include here for reasons of the size
already attained by this paper.  The case with $p=1$ is simpler and was
handled fully in Part~1.

\part{The modal expansion is block-count directed Euler--Maclaurin }

\section{Euler--Maclaurin as a diffused Taylor expansion}

Euler--Maclaurin is often presented as a means to compare $S_N = \sum_{n=0}^N
f(n)$ with $I_N=\int_0^N f(t)\dt$ as $N\to \infty$, which often leads to a
valid asymptotic expansion:
\begin{equation}
  \begin{split}
    S_N = I_N + \frac{f(0)+f(N)}{2} &+ b_2\frac{f'(N)-f'(0)}{2}
    \\
    &+ b_4\frac{f'''(N)-f'''(0)}{4!} \\&+ \dots
    \\
    &+ b_{2J}\frac{f^{(2J-1)}(N)-f^{(2J-1)}(0)}{(2J)!} + R_{2J+1}(N).
  \end{split}
\end{equation}
Here $b_2=\frac16$, $b_4=-\frac1{30}$, \dots, are the Bernoulli numbers. In the
early literature the remainder is not really written out.  But we know that it
involves the Bernoulli \emph{polynomials} $B_j(t)$, $j\geq0$, which can be
defined by
\begin{equation}
  \frac{t\e^{xt}}{e^t-1} = \sum_{j=0}^\infty \frac{B_j(x)}{j!}t^j\,,
\end{equation}
and
\begin{equation}
  \begin{split}
    R_{2J+1}(N) &= \frac{-1}{(2J)!}\int_0^Nf^{(2J)}(x) B_{2J}(\{x\})\dx
\\
    &= \frac{1}{(2J+1)!}\int_0^Nf^{(2J+1)}(x)B_{2J+1}(\{x\})\dx. 
  \end{split}
\end{equation}
Here $\{x\} = x- \lfloor x\rfloor$ is the fractional part function.

The formula above is not really what we want to talk about.  Let us simply take
$N=1$. Then, the second manner to introduce a reader to Euler--Maclaurin is to
relate it to quadrature formulas, as a way to measure here the difference
between $\int_0^1f(x)\dx$ and the trapezoidal rule $(f(0) + f(1))/2$:
\begin{equation}
  \begin{split}
    \int_0^1 f(t)\dt = \frac{f(0)+f(1)}{2} &- b_2\frac{f'(1)-f'(0)}{2}
    \\
    &- b_4\frac{f'''(1)-f'''(0)}{4!}
    \\&- \dots
    \\
    &- b_{2J}\frac{f^{(2J-1)}(1)-f^{(2J-1)}(0)}{(2J)!} - r_{2J+1},
  \end{split}
\end{equation}
with a Bernoulli polynomial allowing to express the remainder
\begin{equation}
  r_{2J+1} = \frac{1}{(2J+1)!}\int_0^1f^{(2J+1)}(x) B_{2J+1}(x)\dx. 
\end{equation}
This still isn't the way we want to look at this.

Observe that the various terms have the shape
$-(j!)^{-1}B_j(0)\int_0^1f^{(j)}(x)\dx$, and as the odd-number indexed $B_j(x)$
polynomials verify $B_j(0)=0$ ($j>1$, odd), we can generally speaking think of
each term as being $-(j!)^{-1}B_j(0)L(f^{(j)})$ where $L$ is integration against
Lebesgue measure. And, using $B_1(x)= x-\frac12$, $b_1=B_1(0)=-\frac12$, we
can manipulate the first term to be $f(0)- B_1(0)L(f')$, so that the formula
now looks:
\begin{equation}
  L(f) = f(0) - \sum_{j=1}^{2J} \frac{B_j(0)}{j!}L(f^{(j)}) - r_{2J+1}\,.
\end{equation}
This formula really is asking us to move the $L$'s to
all be on the same side, so it becomes
\begin{equation}
  f(0) = \sum_{j=0}^{2J} \frac{B_j(0)}{j!}L(f^{(j)}) +  r_{2J+1}\,,
\end{equation}
where it is convenient that $B_0(x)=1=b_0$.

But now, why focus at $x=0$ only? The linear form $L=\Leb$ is quite agnostic
and handles all points the same. Perhaps the ``correct'' formula is rather a
function expansion on the full interval $\Iff01$.
\begin{equation}\label{eq:eumac1}
  f = \sum_{j=0}^{n} L(f^{(j)})\frac{B_j}{j!} +  r_{n+1}(f)\,.
\end{equation}
Observe that suddenly odd indices are equally important and that we
surreptitiously allow both parities for the number of terms.
 
Recall the Taylor formula with remainder
\begin{equation}
  f(x) = \sum_{j=0}^{n} f^{(j)}(a)\frac{(x-a)^j}{j!} 
    +  \int_a^x\frac{(x-t)^n}{n!}f^{(n+1)}(t)\dt.
\end{equation}
Equation~\eqref{eq:eumac1}, once actually established, would be very similar,
with Bernoulli polynomials (which are
monic, and verify $B_{n+1}'=(n+1)B_n$) in the place of the monic monomials,
and with $\Leb$ replacing the Dirac point mass at $a$.

The question about Equation~\eqref{eq:eumac1} now becomes another one: what is
the function $r_{n+1}(f)$? Whatever this functional $r_{n+1}$ is, it will have to have
the property to send polynomials of degree at most $n$ to the zero function.
Indeed the formula is certainly true with $r_{n+1}(f)=0$ for $f$ any polynomial
$P$ of degree at most $n$. This is because any such polynomial can be written uniquely
\begin{equation}\label{eq:PB}
  P = c_0 B_0 + \dots + c_n \frac{B_n}{n!}\,.
\end{equation}
Then, take the derivative $j$ times, using $B'_k = kB_{k-1}$, and apply the
linear form $L= \Leb$.  As is well-known, Bernoulli polynomials $B_j$,
$j\geq1$, verify $L(B_j)=0$. Even better, they are the unique sequence of
monic polynomials with this property and the rule $B'_k = kB_{k-1}$ (an Appell
sequence).  Similarly, the monomials $T_{a,j}$ are the unique sequence of monic
polynomials with the Appell property and the vanishing condition for $j>0$,
under the linear form $\delta_a$.

Going back to Equation \eqref{eq:PB}, we deduce from the Appell and vanishing
properties that necessarily $c_j = L(P^{(j)})$.  This tells us that the
remainder functional $r_{n+1}$ vanishes on all polynomials of degrees up to $n$.
This suggests that it should be expressed as some integral of a kernel
$k_n(x,t)$ against $f^{(n+1)}$. This is, at any rate, exactly what happens for
the Taylor expansion at $a$.  We now confirm that it does work for $\Leb$
similarly as it does for $\delta_a$.

\begin{prop}[Euler--Maclaurin as diffused Taylor]
  Let $L_{\Leb}$ be the linear form $f\mapsto \int_0^1 f(t)\dt$.  Let
  $\phi_j(t) = \frac{B_j(\{t\})}{j!}$ for $j\geq1$.  Let $n\geq1$ and $f\in
  C^{n}(\Iff01)$. Then, for every $x\in \Iff01$,
\begin{align}\label{eq:eumactay}
  f(x)
  &=
  \sum_{j=0}^{n}
  L_{\Leb}(f^{(j)})\frac{B_j(x)}{j!}
  +
  r_{n+1}(f)(x),
\\\shortintertext{with}\label{eq:rn0}
  r_{n+1}(f)(x) &= - \int_0^1 \phi_{n}(x-t)f^{(n)}(t)\dt.
\end{align}
If $f\in C^{n+1}(\Iff01)$, there holds
\begin{equation}\label{eq:rn1}
  r_{n+1}(f)(x) = \int_0^1 \bigl(\phi_{n+1}(x) - \phi_{n+1}(x-t)\bigr)f^{(n+1)}(t)\dt.
\end{equation}
\end{prop}
\begin{proof}
  Remark first that the periodic functions $\phi_j$ are continuous on the real
  line when $j\geq2$, and also for $j=1$ but with the exception of the jump
  discontinuities at the integers.  Take $n=1$. From Equation~\eqref{eq:rn0}
  \begin{align}
    r_2(f)(x) &= -\int_0^x(x-t-\frac12)f'(t)\dt - \int_x^1(x -t + \frac12)f'(t)\dt
\\
&= -(x-\frac12)\int_0^1f'(t)\dt + f(x) - \int_0^1 f(t)\dt,
  \end{align}
  so Equation~\eqref{eq:eumactay} holds for $n=1$ and $f$ of class $C^1$. In
  the computation above the difference between $\phi_1(1)$ and $B_1(1)$ does
  not matter, because $\phi_1$ is integrated, not evaluated, and the factor
  $x-\frac12=B_1(x)$ is correct also for $x=1$ or $x=0$.

  Suppose now that $f$ is of class $C^2$.  Then $\phi_2$ is a continuous
  function, which has a continuous derivative $\phi_1$ at every non-integer
  point, and has left and right derivatives at the integer points.  So we can
  integrate by parts:
  \begin{equation}
    -\int_0^1 \phi_1(x-t)f'(t)\dt = \Bigl[\phi_2(x-t)f'(t)\Bigr]_0^1
             - \int_0^1\phi_2(x-t)f''(t)\dt,
  \end{equation}
   which gives
   \begin{align}
     r_2(f)(x) &= \phi_2(x)\bigl(f'(1)-f'(0)\bigr)  - \int_0^1\phi_2(x-t)f''(t)\dt
\\
&=\int_0^1\bigl(\phi_2(x)- \phi_2(x-t))f''(t)\dt,
   \end{align}
   and this is Equation \eqref{eq:rn1} for $n=1$.

   More generally, the analogous computation shows that for every $n\geq1$, the
   quantity $r_{n+1}(f)$ from Equation~\eqref{eq:rn0} verifies, if $f\in
   C^{n+1}(\Iff01)$, also Equation~\eqref{eq:rn1}.  Suppose now that
   Equation~\eqref{eq:eumactay} holds for $n$ and that $f$ is of class
   $C^{n+1}$.  Equation~\eqref{eq:rn1} says
   \begin{equation}
     r_{n+1}(f) = L_{\Leb}(f^{(n+1)})\phi_{n+1}(x) + r_{n+2}(f),
   \end{equation}
   which establishes Equation~\eqref{eq:eumactay} for $n+1$, hence by induction
   on $n$, for every $n\geq1$.
\end{proof}
\begin{rema}
  Let $f$ be continuous on $\Iff01$.  Define $\phi_0$ to be the locally finite
  Borel measure which is the distributional derivative of $\phi_1$, i.e.\@
  $\phi_0 = 1 - \sum_{n\in\ZZ}\delta_n$.  Interpret the right-hand side of
  Equation~\eqref{eq:rn0} with $n=0$ as the pairing of $-\phi_0(x-t)$ with
  $\Un_{0\leq t < 1}f(t)$. Assuming moreover $0\leq x < 1$, we get $r_1(f)(x) =
  -L_{\Leb}(f)+f(x)$, so $f(x) = L_{\Leb}(f) + r_1(f)(x)$ and
  Equation~\eqref{eq:eumactay} holds.  Further, if $f$ is $C^1$, the value of
  $-\int_0^1\phi_1(x-t)f'(t)\dt$ was computed in the proof above to be
  $-B_1(x)L_{\Leb}(f')+f(x)-L_{\Leb}(f)$, so
  \begin{equation}
    \int_0^1\bigl(B_1(x) - \phi_1(x-t)\bigr)f'(t)\dt = f(x)-L_{\Leb}(f) = r_1(f)(x),
  \end{equation}
  which establishes, for $0\leq x < 1$, the validity of
  Equation~\eqref{eq:rn1} also for $n=0$.
\end{rema}

\section{General stationary and non-stationary expansions}

It is well-known that the core idea behind the Bernoulli polynomials is to
take at each step the primitive which has a vanishing integral on $\Iff01$.
On the other hand the Taylor monomials are obviously obtained by taking at
each step the primitive which vanishes under the Dirac point mass at $a$.  It
is now almost immediate to put the two procedures under the same roof.

In the next proposition, $L$ is not supposed to be bounded with respect to the
sup-norm of functions.
\begin{prop}
  Let $L$ be a linear form on $C(\Iff01)$ such that $L(1)=1$, and define the
  linear map $J:L^1(\Iff01)\to AC(\Iff01)$ by
  \begin{equation}
    Jh=h^{(-1)}-L\bigl(h^{(-1)}\bigr),
  \end{equation}
  where $h^{(-1)}$ is the function $x\mapsto \int_0^x h(t)\dt$ (or any other
  one differing by a constant).
  Then $J^j(1)$ is a polynomial function of degree $j$, and, defining
  \begin{equation}
    \label{eq:Pj}
    P_j=j!J^j(1),
  \end{equation}
  $(P_j)_{j\geq0}$ is the unique sequence of monic polynomials such that
  \begin{equation}\label{eq:appell}
    P_0=1,\qquad
    P_j'=jP_{j-1},\qquad
    L(P_j)=0\quad(j\geq1).
  \end{equation}
  Let $n\geq1$ and $f\in C^n(\Iff01)$. Then
  \begin{gather}\label{eq:gentml}
    f(x) = \sum_{j=0}^{n} L(f^{(j)})\frac{P_j(x)}{j!} +
    r_{n+1}(f)(x),
   \\\shortintertext{where}
\label{eq:genrn0}
    r_{n+1}(f)
    =
    J^n\bigl(f^{(n)}-L(f^{(n)})\bigr).
  \end{gather}
  If moreover $f\in C^{n+1}(\Iff01)$, then
  \begin{equation}\label{eq:genrn1}
    r_{n+1}(f) = J^{n+1}(f^{(n+1)}).
  \end{equation}
\end{prop}
\begin{proof}
  That the $(P_n)$ are polynomials, monic, and the unique solution to
  Equation~\eqref{eq:appell} is obvious: the point is that among the
  primitives of a given continuous $f$, the unique one in the null-space of $L$
  is $J(f)$.

  The existence and uniqueness of the $(P_n)$ sequence having
  been established, let now $g\in C^1(\Iff01)$. Then $g-L(g)$ is in
  the null-space of $L$ and has derivative $g'$. So:
  \begin{equation}\label{eq:Jid}
    g=L(g)+J(g').
  \end{equation}
  Let $D$ be the operator of differentiation, we rewrite this as
  \begin{equation}
    I =\Un L + JD,
  \end{equation}
  where $\Un L$ is the rank-one operator $g\mapsto L(g)\Un$ (in
  Equation~\eqref{eq:genrn0} from the proposition statement, the slightly
  pedantic $\Un$ is omitted). Observe that the composite $JD$ does act on
  $C^j(\Iff01)$ for any $j\geq1$.  On $C^n(\Iff01)$, composing on the right
  with $D$ and on the left with $J$ we obtain in succession
    \begin{align}
    JD &=J\Un LD + J^2D^2\\
    J^2D^2 &= J^2\Un LD^2 + J^3D^3\\\notag
    &\dots\\
    J^{n-1}D^{n-1} &= J^{n-1}\Un LD^{n-1} + J^nD^n,
  \end{align}
  so, when acting on functions of $C^n(\Iff01)$,
  \begin{equation}\label{eq:gen}
    I = \Un L + P_1 LD + \frac12 P_2 LD^2 + \dots + \frac1{(n-1)!}P_{n-1}LD^{n-1} 
       + J^nD^n\,.
  \end{equation}
  This, for $n+1$, and applied to $f\in C^{n+1}(\Iff01)$, establishes
  Equations~\eqref{eq:gentml} and \eqref{eq:genrn1}. Supposing only $f\in
  C^n(\Iff01)$, there holds
  \begin{equation}
    J^n(f^{(n)}) = L(f^{(n)})\frac{P_n}{n!} + J^n\bigl(f^{(n)}-L(f^{(n)})\Un\bigr),
  \end{equation}
  so Equation~\eqref{eq:gen}
  proves for $f$ of class $C^n$ Equations~\eqref{eq:gentml} and\eqref{eq:genrn0}.
\end{proof}

  For $L=\delta_a$, one has
  \begin{equation}
    J(h)(x)=\int_a^x h(t)\dt,
    \qquad
    P_j(x)=(x-a)^j,
  \end{equation}
  and the proposition gives Taylor's expansion. For $L=L_{\Leb}$, the $P_j$'s
  are the classical Bernoulli polynomials and the expansion is the one from
  the previous section, with however a more obfuscated form of the remainder.

  We now examine more closely the remainder, and assume from here on that $L$ is
  a bounded linear form on $C([0,1])$, so by Riesz theorem,
  $L(h)=\int_{\Iff01} h(t)\dmu(t)$ for some complex Borel measure $\mu$.

  As $\mu$ may have atoms we have to be careful below when constructing the
  kernels $k_n(x,t)$ which involve (multiple) integrations against
  $\mu$. We thus choose carefully the indicator functions helping to express the kernel as a multiple integral. Define, for $a,b,t\in\Iff01$,
\begin{equation}
  \chi_{a,b}(t)=
  \begin{cases}
    \Un_{(a,b]}(t),&a\leq b,\\
    -\Un_{(b,a]}(t),&b<a.
  \end{cases}
\end{equation}
Thus, in particular,
\begin{equation}\label{eq:intabh}
  \int_a^b h(t)\dt = \int_0^1\chi_{a,b}(t)h(t)\dt.
\end{equation}

\begin{prop}
  Let $\mu$ be a
  complex Borel measure on $\Iff01$ and $L$ the associated bounded
  linear form.  Assume $L(1)=1$.
  Define for $(x,t)\in \Iff01^2$,
  \begin{equation}
    k_1(x,t)
    =
    \int_{0\leq y\leq 1}\chi_{y,x}(t)\,d\mu(y)
=
\begin{cases}
  \mu(\Ifo{0}{t})& 0\leq t\leq x
\\
  \mu(\Ifo{0}{t})-1& x<t\leq 1
\end{cases}
  \end{equation}
  and for $n\geq2$,
  \begin{equation}\label{eq:genkn}
    \begin{split}
      k_n(x,t)
      ={}&
      \int_{\Iff01^n}
      \int_{\Iff01^{n-1}}
      \chi_{y_1,x}(t_1)
      \chi_{y_2,t_1}(t_2)\cdots
      \chi_{y_{n-1},t_{n-2}}(t_{n-1})
      \chi_{y_n,t_{n-1}}(t)
      \\
      &
      \dt_1\cdots \dt_{n-1}\,
      \dmu(y_1)\cdots \dmu(y_n).
    \end{split}
  \end{equation}
  The polynomial functions $P_n$, $n\geq1$, from the previous proposition are
\begin{equation}
  P_n(x) = n!\int_0^1k_n(x,t)\dt.
\end{equation}
  If $f\in C^n(\Iff01)$, the remainder in Equation~\eqref{eq:gentml} has the integral form
  \begin{equation}
    r_{n+1}(f)(x)
    =
    \int_0^1 k_n(x,t)
    \bigl(f^{(n)}(t)-L(f^{(n)})\bigr)\dt.
  \end{equation}
  If $f\in C^{n+1}(\Iff01)$, it also has the form
  \begin{equation}
    r_{n+1}(f)(x)
    =
    \int_0^1k_{n+1}(x,t)f^{(n+1)}(t)\dt.
  \end{equation}
\end{prop}
\begin{proof}
  Let $h\in L^1(\Iff01)$.  Since $\mu(\Iff01)=L(1)=1$, we have
  \begin{align}
    J(h)(x)
    &=
    \int_{\Iff01}
    \bigl(h^{(-1)}(x)-h^{(-1)}(y)\bigr)\dmu(y)
    \notag\\
    &=
    \int_{\Iff01}\int_y^x h(t)\dt\,\dmu(y).
  \end{align}
  Using Equation~\eqref{eq:intabh} and Fubini therefore gives
  \begin{equation}\label{eq:Jk1}
    J(h)(x)
    =
    \int_0^1
    \left(
      \int_{0\leq y\leq 1}\chi_{y,x}(t)\dmu(y)
    \right)
    h(t)\dt
    =
    \int_0^1k_1(x,t)h(t)\dt.
  \end{equation}
  Verifying the explicit expression for $k_1$, as given in the statement, is
  left to the reader.
  Fubini applied to Equation~\eqref{eq:genkn} gives, for every $n\geq1$,
  \begin{equation}\label{eq:krec}
    k_{n+1}(x,t)
    =
    \int_0^1 k_1(x,s)k_n(s,t)\,ds.
  \end{equation}
  Indeed, this amounts simply to separating in the defining multiple
  integral the variables $y_1$ and $t_1$ from all the subsequent variables.

  Equations~\eqref{eq:Jk1} and \eqref{eq:krec}, followed by another
  application of Fubini, give by induction
  \begin{equation}\label{eq:Jkn}
    J^n(h)(x)
    =
    \int_0^1k_n(x,t)h(t)\dt,
    \qquad n\geq1.
  \end{equation}
  All these applications of Fubini are justified by the finiteness of the
  total variation of $\mu$ and the integrability of $h$.

  Taking $h=1$ in Equation~\eqref{eq:Jkn} and using
  Equation~\eqref{eq:Pj} gives
  \begin{equation}
    P_n(x)
    =
    n!\int_0^1k_n(x,t)\dt.
  \end{equation}

  Finally, applying Equation~\eqref{eq:Jkn} to the right-hand side of
  Equation~\eqref{eq:genrn0} gives
  \begin{equation}
    r_{n+1}(f)(x)
    =
    \int_0^1
    k_n(x,t)
    \bigl(f^{(n)}(t)-L(f^{(n)})\bigr)\dt.
  \end{equation}
  If $f\in C^{n+1}(\Iff01)$, applying it instead to
  Equation~\eqref{eq:genrn1} gives
  \begin{equation}
    r_{n+1}(f)(x)
    =
    \int_0^1k_{n+1}(x,t)f^{(n+1)}(t)\dt.
  \end{equation}
\end{proof}

A generalization presents itself: in place of coefficients $L(f)$, $L(f')$,
$L(f'')$, \dots, we can search a polynomial basis $(P_n/n!)$ such that the
expansion of $f$ involves coefficients $L_1(f)$, $L_2(f')$, \dots,
$L_{n+1}(f^{(n)})$, $n\geq0$.  As the next proposition explains, this indeed
is an extension of the earlier ``stationary'' scheme, but there is a subtlety:
it is \emph{not} the case that the next polynomial $P_n$ will be determined
from being in the null-space of $L_n$; it will rather always, when
$n\geq1$, be in the null-space of $L_1$.  See also comments below.
\begin{prop}\label{prop:nonstatTML}
  Let $(L_j)_{j\geq1}$ be linear forms on $C(\Iff01)$ such that
  $L_j(1)=1$.  For $j\geq1$, define the linear map
  $J_j:L^1(\Iff01)\to AC(\Iff01)$ by
  \begin{equation}\label{eq:nJ}
    J_jh=h^{(-1)}-L_j\bigl(h^{(-1)}\bigr).
  \end{equation}
  Define the monic polynomials
  \begin{equation}\label{eq:nB}
    \mathcal B_0=1,\qquad
    \mathcal B_j=j!J_1J_2\cdots J_j(1),\qquad j\geq1.
  \end{equation}
  Let $n\geq1$ and $f\in C^n(\Iff01)$. Then
  \begin{equation}\label{eq:nexp}
    f(x)
    =
    \sum_{j=0}^n
    L_{j+1}(f^{(j)})\frac{\mathcal B_j(x)}{j!}
    +
    r_{n+1}(f)(x),
  \end{equation}
  where
  \begin{equation}\label{eq:nonstatrn0}
    r_{n+1}(f)
    =
    J_1J_2\cdots J_n
    \bigl(f^{(n)}-L_{n+1}(f^{(n)})\bigr).
  \end{equation}
  If moreover $f\in C^{n+1}(\Iff01)$, then
  \begin{equation}\label{eq:nonstatrn1}
    r_{n+1}(f)
    =
    J_1J_2\cdots J_{n+1}(f^{(n+1)}).
  \end{equation}

  Suppose now that all the $L_j$ are bounded, and let $\nu_j$ be the complex
  Borel measure representing $L_j$, so that
  \begin{equation}
    L_j(h)=\int_{\Iff01}h(t)\,d\nu_j(t),
    \qquad
    \nu_j(\Iff01)=1.
  \end{equation}
  Put
  \begin{equation}
    \kappa_j(x,t)
    =
    \int_{0\leq y\leq1}\chi_{y,x}(t)\,d\nu_j(y)
    =
    \begin{cases}
      \nu_j(\Ifo{0}{t}),&0\leq t\leq x,\\
      \nu_j(\Ifo{0}{t})-1,&x<t\leq1.
    \end{cases}
  \end{equation}
  Define $k_1=\kappa_1$ and, for $n\geq1$,
  \begin{equation}\label{eq:nonstatnk}
    k_{n+1}(x,t)
    =
    \int_0^1k_n(x,s)\kappa_{n+1}(s,t)\,ds,
  \end{equation}
  equivalently $k_{n} = \kappa_1\star\dots\star\kappa_n$ for $n\geq1$,
  with $\star$ the composition of integral kernels.
  Then
  \begin{equation}\label{eq:nonstatnkernel}
    J_1J_2\cdots J_n(h)(x)
    =
    \int_0^1k_n(x,t)h(t)\dt,
    \qquad n\geq1,
  \end{equation}
  and consequently
  \begin{equation}
    \mathcal B_n(x)
    =
    n!\int_0^1k_n(x,t)\dt.
  \end{equation}
  In this case Equation~\eqref{eq:nexp} takes the form
  \begin{equation}\label{eq:nonstatTML}
    f(x)
    =
    \sum_{j=0}^n
    \Bigl(\int_{\Iff01}f^{(j)}(t)\,d\nu_{j+1}(t)\Bigr)
    \frac{\mathcal B_j(x)}{j!}
    +
    r_{n+1}(f)(x).
  \end{equation}
  For $f\in C^n(\Iff01)$, its remainder is
  \begin{equation}\label{eq:nonstatkrn0}
    r_{n+1}(f)(x)
    =
    \int_0^1
    k_n(x,t)
    \left(
      f^{(n)}(t)
      -
      \int_{\Iff01}f^{(n)}(s)\,d\nu_{n+1}(s)
    \right)\dt,
  \end{equation}
  and, if $f\in C^{n+1}(\Iff01)$,
  \begin{equation}\label{eq:nonstatkrn1}
    r_{n+1}(f)(x)
    =
    \int_0^1k_{n+1}(x,t)f^{(n+1)}(t)\dt.
  \end{equation}
\end{prop}
\begin{proof}
  The proof is omitted because it is a straightforward adaptation of the one
  we gave in the case where all linear forms are the same.
\end{proof}
We call this a \emph{non-stationary} Taylor--Euler--Maclaurin expansion.

We draw the reader's attention to the fact that in this non-stationary
case, the polynomials $(\mathcal B_n)_{n\geq0}$ do \emph{not} form an Appell
sequence in general.  Indeed, we have for $n\geq2$
\begin{equation}\label{eq:footired}
  \cB_{n}'
  =
  n!\,J_2J_3\cdots J_{n}(1), 
\end{equation}
while
\begin{equation}
  \cB_{n-1}
  =
  (n-1)!\,J_1J_2\cdots J_{n-1}(1). 
\end{equation}

In the special case where each measure $\nu_j$ is a Dirac point mass
$\delta_{\alpha_j}$, the monic polynomials $\cB_j$ defined above form
what is called in the literature a Goncharov sequence.  Indeed, the polynomial
$\cB_n$, $n\geq1$, verifies the conditions
\begin{equation}
  \cB_n(\alpha_1)
  = \cB_n'(\alpha_2)
  = \dots
  = \cB_n^{(n-1)}(\alpha_n)
  = 0.
\end{equation}
For example $\cB_{n}(\alpha_1)=0$ by construction, as it is in the range of
$J_1$, and $\cB_{n}'(\alpha_2)=0$ from Equation~\eqref{eq:footired}, which
shows that $\cB_n'$ is in the range of $J_2$, and so on.  In the literature
such Goncharov polynomials are also defined with the $\alpha_j$'s being
complex numbers.  Formulas for the remainders in this case are well-known
facts going back to the origins of the topic itself,
\cite[pp.\@ 10--11]{gontcharoff1930}.

\section{The block-directed Euler--Maclaurin expansion}

We revisit now the apparatus we developed in Part~2 to establish the modal
expansion Theorem~\ref{thm:part2main} for block-Irwin harmonic sums $I(b,w,k)$,
$k\geq1$. Recall that we associated with the non-empty word $w$ in radix $b$ a
sequence of (exponentially tilted) probability measures $(\nu_n)_{n\geq1}$ on
the unit interval, with $\nu_1 = \Leb$.  We defined a certain operator $\sK_w$
acting on functions, in particular on polynomials, and obtained a sequence
$(\cB_m)_{m\geq0}$ of monic eigenpolynomials.
\begin{prop}
  The monic eigenpolynomials $(\cB_m)_{m\geq0}$ of $\sK_w$ are the ones
  associated in Proposition~\ref{prop:nonstatTML} with the sequence
  $(\nu_n)_{n\geq1}$. Thus, with the notation of that proposition,
  \begin{equation}\label{eq:foo2}
    \frac{\cB_m}{m!}=J_1J_2\cdots J_m(1),
    \qquad m\geq1,
  \end{equation}
  and, for $f\in C^n(\Iff01)$,
  \begin{equation}
    f(x)
    =
    \sum_{m=0}^n
    \bigl(\int_0^1 f^{(m)}(t)\dnu_{m+1}(t)\bigr)
    \frac{\cB_m(x)}{m!}
    +
    r_{n+1}(f)(x),
  \end{equation}
  with the remainder given by Equation~\eqref{eq:nonstatkrn0} of
  Proposition~\ref{prop:nonstatTML}, or also by
  Equation~\eqref{eq:nonstatkrn1} if $f$ of class $C^{n+1}$.
\end{prop}
\begin{proof}
  The polynomial $P_m$ defined by the right-hand side of Equation~\eqref{eq:foo2}
  is monic of degree $m$ and satisfies
  \begin{equation}
    \int_0^1 P_m^{(j)}(x)\dnu_{j+1}(x)=0,
    \qquad 0\leq j<m.
  \end{equation}
  These are precisely the orthogonality conditions already established in
  Part~2 for $(\cB_m)$ and $(\nu_j^{(j-1)})$, $m\neq j-1$, and they determine
  a monic polynomial of degree $m$ uniquely.  The second assertion is the
  application of the general non-stationary expansion.
\end{proof}

Recall that we associated to $w$ the word $g_w$ which is: $w$ if $w$ is unbordered,
else the complement in $w$ of its longest border.  And we then defined 
$x_w = x(0.g_wg_wg_w\dots)\in \Iff01$.
\begin{theo}[Block-directed Euler--Maclaurin--Taylor expansion]
  Let $f$ be analytic at $x_w$, and suppose that the radius of convergence $r$
  of the Taylor series at $x_w$ verifies $r>R_w = \max(x_w,1-x_w)$, i.e.\@
  that $\Iff01\subset D(x_w,r)$.  Then
  \begin{equation}\label{eq:fTML}
   \forall x\in D(x_w,r)\quad f(x)
    =
    \sum_{m=0}^\infty
    \Bigl(\int_0^1 f^{(m)}(t)\dnu_{m+1}(t)\Bigr)
    \frac{\cB_m(x)}{m!},
  \end{equation}
  with uniform geometric convergence on the compact subsets of $D(x_w,r)$.
\end{theo}
\begin{proof}
  Let $K$ be a compact subset of $D(x_w,r)$, and choose $R\in \Ioo{R_w}{r}$
  such that $K\subset D(x_w, R)$.  As $R>R_w$, the normalized polynomials
  $(R^{-m}\cB_m)_{m\geq0}$, by Proposition~\ref{prop:wBmRiesz}, are a Riesz
  basis of $H^2(D(x_w,R))$, and the dual system is
  $(R^{m}(-1)^{m}(m!)^{-1}\nu_{m+1}^{(m)})_{m\geq0}$. So the series on the
  right-hand side in Equation~\eqref{eq:fTML} converges in $H^2(D(x_w,R))$ to
  $f$. Hence, it converges also in the sup-norm on $K$, and in particular
  pointwise.  Call $(a_m)_{m\geq0}\in \ell^2$ the coefficient sequence of $f$
  in the Riesz basis $(R^{-m}\cB_m)_{m\geq0}$. Choose $R_w<R_2<R$ such that
  $D(x_w,R_2)$ still contains the compact set $K$.  The Riesz basis property
  at radius $R_2$ and the continuous embedding of $H^2(D(x_w,R_2))$ into
  $C_{\CC}(K)_{\|\cdot\|_\infty}$ give, for some constant $C$,
  \begin{equation}
    R_2^{-m}\|\cB_m\|_\infty\leq C.
  \end{equation}
  Therefore, $|a_m|\|R^{-m}\cB_m\|_\infty = o(\rho^m)$, $\rho=R_2/R<1$.  This
  concludes the proof and the paper.
\end{proof}

\singlespacing

\bigskip
\noindent
  Université de Lille,
  Faculté des Sciences et technologies,
  Département de mathématiques,
  Cité Scientifique,
  F-59655 Villeneuve d'Ascq cedex,
  France.
\newline
\strut \texttt{jean-francois.burnol@univ-lille.fr}

\end{document}